\documentclass[10pt,twoside, a4paper, english, reqno]{amsart}
\usepackage[dvips]{epsfig}
\usepackage{amscd}
\usepackage{stmaryrd}
\usepackage{amssymb}
\usepackage{amsthm}
\usepackage{amsmath}
\usepackage{graphicx,enumitem,dsfont}
\usepackage{latexsym}

\usepackage{upref}
\usepackage[colorlinks]{hyperref}
\usepackage{color}
\usepackage{subcaption}
\usepackage[usenames,dvipsnames]{xcolor}
\theoremstyle{plain}
\newtheorem{thm}{Theorem}[section]
\theoremstyle{plain}
\newtheorem{lem}[thm]{Lemma}
\newtheorem{prop}[thm]{Proposition}
\newtheorem{cor}[thm]{Corollary}

\theoremstyle{definition}
\newtheorem{defi}{Definition}[section]
\newtheorem{rem}{Remark}[section]

\newtheorem{cond}{Condition}[section]
\newtheorem*{maintheorem*}{Main Theorem}

\newenvironment{Assumptions}
{
\setcounter{enumi}{0}

\begin{enumerate}}
{\end{enumerate} }

\newcommand{\R}{\ensuremath{\mathbb{R}}}

\newcommand{\goto}{\ensuremath{\rightarrow}}

\newcommand{\eps}{\ensuremath{\varepsilon}}

\def\N{{I\!\!N}}

\numberwithin{equation}{section} \allowdisplaybreaks

\title[LDP, CLT and MDP for non-local stochastic balance laws]
{Stochastic Fractional Conservation Laws: Large deviation principle, Central limit theorem and Moderate deviation principle}

\date{}

\author[ S. R. Behera ]{Soumya Ranjan Behera}
\address[Soumya Ranjan Behera] {\newline 
Department of Mathematics,
Indian Institute of Technology Delhi,
Hauz Khas, New Delhi, 110016, India.}
\email[] {maz198759@iitd.ac.in}

\author[A. K. Majee]{Ananta K. Majee}
\address[Ananta K. Majee]{\newline
Department of Mathematics,
Indian Institute of Technology Delhi,
Hauz Khas, New Delhi, 110016, India. }
\email[]{majee@maths.iitd.ac.in}


\thanks{}
\begin{document}
\begin{abstract}
  In this article, we establish the Freidlin-Wentzell type large deviation principle and central limit theorem for stochastic fractional conservation laws with small multiplicative noise in kinetic formulation framework. The weak convergence method and doubling variables method play a crucial role. As a consequence, we also establish moderate deviation principle for the underlying problem.
\end{abstract}
\maketitle
\section{Introduction}
Let $(\Omega, \mathcal{F}, \mathcal{P}, (\mathcal{F}_t)_{t \in [0,T]})$ be a stochastic basis and $T >0$. We are interested in the study of small noise asymptotic behaviour of the fractional scalar conservation laws. We consider the following  fractional conservation laws with stochastic forcing,
\begin{equation}\label{eq:fractionalSCL}
\begin{cases}
     \displaystyle du + \text{div}_x F(u(t,x))\,dt + (-\Delta)^\theta[{\tt \Phi}(u(t,\cdot))](x)\,dt = {\tt h}(u(t,x))\,dW(t), \\
    u(0,x) = u_0(x),
\end{cases}
\end{equation}
with $x \in \mathbb{T}^d,  t \in [0,T]$, and $\mathbb{T}^d \subset \R^d$ 
denotes the $d-$dimensional torus with periodic length $1$. In \eqref{eq:fractionalSCL}, $u_0$ is the given initial function, $F: \R \rightarrow \R^d$, ${\tt \Phi}: \R \rightarrow \R$ are given sufficiently smooth function and fulfil certain conditions specified later. Moreover, ${\tt h}$ is the given noise coefficient , $W$ is a cylindrical Wiener process defined in a separable Hilbert space $\mathcal{H}$ with the form $W(t) = \sum_{n \ge 1}\beta_n(t)e_n$, where $(e_n)_{n \ge 1}$ is a complete orthonormal basis in $\mathcal{H}$ and $(\beta_n)_{n \ge 1}$ is a sequence of independent real valued Brownian motion. In \eqref{eq:fractionalSCL}, $(-\Delta)^\theta$ denotes the fractional Laplacian operator of order $\theta \in (0,1)$, defined by
\begin{align}
    (-\Delta)^\theta[\Gamma](x) := -{\rm P.V.} \int_{\R^d}\big(\Gamma(x+z)-\Gamma(x)\big)\gamma(z)\,dz, \quad \forall \quad x \in \mathbb{T}^d,\notag
\end{align}
where $\gamma(z) = \frac{1}{|z|^{d+2\theta}}$, $z \neq 0$, $\gamma(0) = 0$ and $\Gamma$  is a sufficiently regular function.
\vspace{0.1cm}

    The problem \eqref{eq:fractionalSCL} could be seen as a stochastic perturbation of non-local conservation laws. The equation \eqref{eq:fractionalSCL} becomes standard conservation laws in the absence of non-local and noise terms i.e. when ${\tt \Phi} = 0 = {\tt h}$. The well-posedness theory for hyperbolic conservation laws is well established; see \cite{dafermos,godu,Kruzkov} and the references there in. The entropy solution theory for fractional balance laws was studied by Alibaud in \cite{Alibaud 2007} and Cifani et al. in \cite{cifani}. Lions et al. \cite{Lions} was first to introduce the kinetic solution theory for conservation laws. Later on, Chen et al. \cite{Chen-2003} studied the wellposedness of kinetic solution for non-isotropic degenerate parabolic-hyperbolic equation. Alibaud et al. \cite{Alibaud 2020} developed the well-possedness theory for non-local conservation laws in kinetic formulation framework.
    \vspace{0.1cm}

The study of conservation laws driven by stochastic forcing has grown significantly over the past few years. The well-posedness analysis for conservation laws driven by noise have been researched by many authors, see \cite{Bauzet-2012,Bauzet-2015,Majee-2015,Majee-2014, BKM-2015,Majee-2019, Chen:2012fk,Hofmanova-2016, Vovelle2010,Vovelle-2018, nualart:2008,kim,Majee-2017}. Koley et al.  in \cite{frac lin, frac non}  proved the existence and uniqueness of entropy solution  for stochastic non-local conservation laws in unbounded spatial domain. In \cite{AC_1}, the author has just lately studied the degenerate fractional conservation laws driven by noise and established the well-posedness of  kinetic solution and the $L^1$-continuous dependence estimations by means of BV-estimates. By using of kinetic formulation, Chaudhary \cite{AC_2} has explored the existence-uniqueness and long time behavior of solution for \eqref{eq:fractionalSCL}.
\vspace{0.1cm}

The asymptotic behaviour of solution of stochastic PDEs while vanishing the noise force is fascinating to study from the perspective of statistical mechanics. Therefore, the establishment of the large and moderate deviation principle (in short LDP/MDP) and the central limit theorem (in short CLT) are  crucial for finer understanding of the description of certain physical phenomena. For more insights regarding LDP, MDP and there applications,  we refer \cite{J-stat-12,Ann-stat-91,Varadhan-1,Varadhan-2, Ellis-85} and the references there in. There have been many advancements in the field of asymptotic analysis of stochastic PDEs during the past few decades; for LDP, see \cite{Zhang-JD-17,Liu-10,Matoussi,kavin-22} and for CLT and MDP, see \cite{dong-2017,Fatheddin-16,zhang-mdp-21,
zhang-21-mdp-clt}. To the best of our knowledge, Mariani \cite{Mariani} was the first to examine the large deviation for stochastic conservation laws, in which the author treated a family of stochastic balance laws as  parabolic SPDEs with an additional viscosity term. The weak convergence approach obtained by Dupuis et al. \cite{Dupuis-97}, is a crucial tool in establishing LDP, CLT and MDP. Using weak convergence approach Zhang et al. \cite{Zhang-2020} established LDP for hyperbolic scalar conservation laws perturbed by small multiplicative noise. In \cite{Dong-21}, Zhang et al. established LDP for quasi-linear parabolic stochastic PDE. Very recently, in \cite{CLT} the authors have studied CLT and MDP for problem \eqref{eq:fractionalSCL} with ${\tt \Phi} = 0$ and initial value $u_0$ is equal to constant.
\subsection{Aim and outline of the paper}
Our main goal is to study the large deviation principle, central limit theorem and moderate deviation principle for \eqref{eq:fractionalSCL} driven by small multiplicative noise. More precisely, we are concerned in asymptotic behavior of the trajectories, 
$$z^\eps(t):= \frac{1}{\sqrt{\eps}\Lambda(\eps)}\big(u^\eps(t) -\hat{u}(t)\big), \quad t \in [0,T], \quad \eps \in (0,1),$$
for some deviation scale $\Lambda(\eps)$ which affects the asymptotic behaviour of $z^\eps$, where $u^\eps$ is the kinetic solution of \eqref{eq:small-noise} and $\hat{u}$ is the solution corresponding to deterministic fractional conservation laws. Specifically, the following three cases are involved.
\begin{itemize}
    \item [i)] When $\Lambda(\eps) = \frac{1}{\sqrt{\eps}}$, it gives the LDP which is proved in Section \ref{sec:LDP}
    \item [ii)] When $\Lambda(\eps) =1$, it gives the CLT which is established in Section \ref{sec:CLT}.
    \item[iii)] When $\Lambda(\eps) \rightarrow +\infty$ and  $\sqrt{\eps}\Lambda(\eps) \rightarrow 0$ as $ \eps \rightarrow 0$, it gives the MDP---which in principle fill the gap 
between CLT scale~($\Lambda(\eps)=1$) and LDP scale~($\Lambda(\eps) = \frac{1}{\sqrt{\eps}}$). We have demonstrated MDP for \eqref{eq:fractionalSCL} in Section \ref{sec:MDP}.
\end{itemize}
\vspace{0.1cm}

In order to establish LDP, it is crucial to obtain the well-possedness of skeleton equation \eqref{eq:skeleton}. We apply doubling variables method to prove the uniqueness of skeleton equation and to show, existence we apply vanishing viscosity method which depends upon the extraction of a kinetic solution from the approximate sequence. In this regard, one requires the existence and uniqueness of regularized/viscous problem \eqref{eq:regularized-skeleton} which is typically based upon monotone method or semi-implicit time discretization technique. However, the presence of non-linear fractional diffusion term in \eqref{eq:regularized-skeleton}, makes it impossible to use such methods. Thus, we consider a fourth order singular perturbation problem \eqref{eq:singular perturbation-skeleton} and use compactness argument along with a-priori estimates to get existence of a viscous solution--- which is a critical and novel feature of this article. Again, doubling of variables method gives the uniqueness of viscous problem. Thanks to Zhang et al. \cite{Matoussi}, to prove LDP, it is enough to validate the Condition \ref{cond:2}. For validation of ${\rm ii)}$, Condition \ref{cond:2}, we prove the strong convergence of $u_{\ell_\eps}$ to $u_\ell$ in $L^1([0,T]; L^1(\mathbb{T}^d))$ where $u_{{\ell}_\eps}$ and $u_\ell$ are the kinetic solution of skeleton equation \eqref{eq:skeleton} corresponding to ${\ell}_\eps$ and $\ell$ respectively with ${\ell}_\eps \rightarrow \ell$ weakly in $L^2([0,T];\mathcal{H}))$. In obtaining the convergence result, an essential part is to show, for fix $\Bar{\eps} > 0$, $u_{{\ell}_\eps}^{\Bar{\eps}} \rightarrow u_{\ell}^{\Bar{\eps}}$ in $L^1([0,T]; L^1(\mathbb{T}^d))$ as $\eps \rightarrow 0$ and to prove this one couldn't follow exactly the same line of arguments as done in  \cite{Zhang-2020} due to the presence of non-liner fractional term in \eqref{eq:regularized-skeleton}. Therefore, we use  compactness argument and  uniqueness of  \eqref{eq:regularized-skeleton} to get the required result; see Proposition \ref{prop:convergence}. In view of Girsanov's theorem and well-posedness theory of \eqref{eq:skeleton} and \eqref{eq:fractionalSCL}, to validate ${\rm i)}$ of condition \ref{cond:2}, we prove that $\|\bar{u}^\eps(\cdot) - v^\eps(\cdot)\|_{\Large L^1([0,T];L^1(\mathbb{T}^d))} \rightarrow 0 \quad \text{in probability},$ using doubling the variables method, where $\bar{u}^\eps$ and $v^\eps$ are kinetic solution of \eqref{eq:small-noise-ldp} and \eqref{eq:skeleton} with ${\ell}$ replaced by ${\ell}_\eps$ respectively.

We establish the CLT and MDP for \eqref{eq:fractionalSCL} perturbed by small noise, with initial function $u_0$ as constant (for simplicity we take $u_0 = 1$). Due to the lack of viscous term in \eqref{eq:fractionalSCL}, the kinetic solution $u$ of \eqref{eq:fractionalSCL} exists in a rather irregular space. Only $W^{\mu,1}$-regularity for $u$ in a spatial variable with $\mu \in (0,1)$ could be anticipated, which is insufficient to prove the well-possedness theory for \eqref{eq:star-FSCL} and \eqref{eq:mdp-skeleton}, limiting our ability to investigate the CLT and MDP with general initial data $u_0$. To obtain CLT, our aim is to show the convergence of $\frac{1}{\sqrt{\eps}}(u^\eps -\hat{u})$ to  $u^\star$, where $u^\eps, \hat{u}$ and  $u^\star$ are the unique kinetic solutions to problems \eqref{eq:small-noise}, \eqref{eq:deterministic-FCL}
and \eqref{eq:star-FSCL} respectively. As discussed in \cite{CLT}, the associated equations of $\frac{1}{\sqrt{\eps}}(u^\eps -\hat{u})$ and  $u^\star$ lacks symmetry, making it impossible to obtain the convergence result by directly applying doubling variables techniques. Thus, we introduce some auxiliary approximation processes and establish corresponding convergence results. Likewise in \cite[Proposition $3.3$]{CLT}, one couldn't apply It\^o-formula to conclude the convergence of $w^{\eps, \eta}$ to $u^{\star, \eta}$ in $L^1(\Omega \times [0,T]; L^1(\mathbb{T}^d))$, where $w^{\eps, \eta} := \frac{u^{\eps,\eta}-\hat{u}^\eta}{\sqrt{\eps}}$ and $u^{\star, \eta}$
satisfy \eqref{eq:w-eps-eta} and \eqref{eq:viscous-star-FSCL} respectively. This approach demands higher regularity of $w^{\eps, \eta}$ as compare to the existing one (see, Remark \ref{rem:about-phi-theta}) due to the non-linear fractional term. To tackle this technical difficulty, we use compactness argument and uniform estimates to infer the tightness of sequence of laws of $w^{\eps, \eta}$. Then using Prokhorov compactness theorem and a modified version of Skorokhod representation theorem, we get  a.s. convergence of a new random variable $\tilde{w}^{\eps,\eta}$ to $\tilde{w}^{\eta}$ in a new probability space $(\tilde{\Omega}, \tilde{\mathcal{F}}, \Tilde{\mathbb{P}})$. With this in hand, we apply Vitali convergence theorem along with a-priori estimates of $\tilde{w}^\eta$ to conclude 
existence of a martingale solution to \eqref{eq:viscous-star-FSCL}. Gy\"ongy-Krylov's characterization of convergence in probability and uniqueness of \eqref{eq:viscous-star-FSCL} yield the desired convergence result for $w^{\eps, \eta}$.
\vspace{0.1cm}

To establish MDP, we show that $\frac{u^\eps - \hat{u}}{\sqrt{\eps}\Lambda(\eps)}$ satisfies LDP in $L^1([0,T];L^1(\mathbb{T}^d))$, with $\sqrt{\eps}\Lambda(\eps) \rightarrow 0$ as $\eps \rightarrow 0$. We encounter the same challenges in the proof of MDP as seen in CLT. Thus, we introduce some suitable parabolic approximation equations. Using existence uniqueness of skeleton equation and mimicking the proof of CLT, we obtain MDP for \eqref{eq:fractionalSCL} for constant initial data.
\vspace{0.1cm}

The article is structured as follows. The kinetic formulation framework, the assumptions, previously obtained results, mathematical framework of LDP, CLT and MDP and the main results of this paper are stated in Section \ref{sec:technical}. We establish the well-posedness of regularized skeleton equation and skeleton equation in Section \ref{sec:skeleton}. The proofs of LDP, CLT and MDP are provided in Sections \ref{sec:LDP}, \ref{sec:CLT} and \ref{sec:MDP} respectively.

\section{Preliminaries And Technical Framework}\label{sec:technical} 
In this paper, we use the letters $C$ to denote various generic constants. Let  $\mathcal{L}(H_1, H_2)$ (resp. $\mathcal{L}_2(H_1, H_2)$) be the space of bounded (resp. Hilbert Schmidt) linear operator from a Hilbert spaces $H_1$ to $H_2$ with norm $\|\cdot\|_{\mathcal{L}(H_1, H_2)}$(resp. $\|\cdot\|_{\mathcal{L}_2(H_1, H_2)}$). For all $\alpha \ge 1$, $H^\alpha(\mathbb{T}^d) = W^{\alpha,2}(\mathbb{T}^d)$ be the usual Sobolev space of order $\alpha$ and  $H^{-\alpha}(\mathbb{T}^d)$ be the topological dual of $H^{\alpha}(\mathbb{T}^d)$, with norm $\|\cdot\|_{H^{-\alpha}}$.  Moreover, the brackets $\langle\cdot,\cdot \rangle$ denotes the duality between $C_c^\infty(\mathbb{T}^d \times \R)$ and the space of distribution over $\mathbb{T}^d \times \R$. For $1 \le p \le \infty$ and $q : = \frac{p}{p-1},$ the conjugate exponent of $p$, we denote
\begin{align*}
    \langle g, \Bar{g} \rangle := \int_{\mathbb{T}^d}\int_{\R}g(x,\xi)\Bar{g}(x, \xi)\,dx\,d\xi, \quad g \in L^p(\mathbb{T}^d \times \R), \Bar{g} \in L^q(\mathbb{T}^d \times \R), 
\end{align*}
and also for a measure $m$ on the Borel measurable space  $\mathbb{T}^d \times [0,T] \times \R$,
\begin{align}\label{eq:measure-m}
    m(\phi) := \langle m, \phi \rangle = \int_{\mathbb{T}^d \times [0,T] \times \R} \phi(x,t,\xi)\,dm(x,t,\xi), \quad \phi \in C_b(\mathbb{T}^d \times [0,T] \times \R),
\end{align}
where $C_b$ represents the space of bounded continuous function. Furthermore, we define  the canonical space $\mathcal{H} \subset \mathcal{H}_0$ via, 
\begin{align*}
     \mathcal{H}_0 = \Big\{w = \underset{n \ge 1}{\sum} w_ne_n \,\Big | \,\underset{n \ge 1}{\sum}\frac{w_n^2}{n^2} < \infty\Big\}\quad \text{with the norm} \quad \|w\|_{\mathcal{H}_0}^2 = \underset{n \ge 1}{\sum}\frac{w_n^2}{n^2}.
\end{align*}
Observe that the embedding $\mathcal{H} \hookrightarrow \mathcal{H}_0$ is Hilbert-Schmidt. The Wiener process $W$ has $\mathcal{P}$-a.s. continuous trajectories in $C([0,T]; \mathcal{H}_0)$. We consider the map ${\tt h}: \mathcal{H} \rightarrow L^2(\mathbb{T}^d)$ defined by ${\tt h}(u)e_n = {\tt h}_n(\cdot, u)$. Thus we define, 
\begin{align*}
    {\tt h}(u) = \underset{n \ge 1}{\sum} {\tt h}_n(\cdot, u)e_n.
\end{align*}
Now we make the following assumptions which are essential in obtaining our desire results.
\begin{Assumptions}
\item \label{A1} $F: \R \rightarrow \R^d$ is $C^2$ and Lipschitz continuous function. 
\item  \label{A2}${\tt \Phi} : \R \rightarrow \R$ is a non-decreasing Lipschitz continuous function. 
\item  \label{A3}${\tt h}_k \in C(\mathbb{T}^d \times \R)$ satisfies the following bound: for $x,\Bar{x} \in \mathbb{T}^d$ and $y, \Bar{y} \in \R$,
\begin{align*}
   &\hspace{1cm}G^2(x,y):= \underset{k \ge 1}{\sum}\big|{\tt h}_k(x,y)\big|^2 \le C(1 + |y|^2), \\
   &\underset{k \ge 1}{\sum} \big|{\tt h}_k(x,y)-{\tt h}_k(\Bar{x},\Bar{y})\big|^2 \le C\big(|x-\Bar{x}|^2  + |y -\Bar{y}|^2 
   \big).
\end{align*} 
\end{Assumptions}
We recall the basic definition of kinetic solution and kinetic function from \cite{AC_2}. Let $\mathcal{R}^+([0,T] \times \mathbb{T}^d \times \R)$ be the collection of non-negative Radon measures over 
$[0,T] \times \mathbb{T}^d \times \R$.
\begin{defi}\label{def:kinetic-mes} A map $m$ from $\Omega$ to $\mathcal{R}^+([0,T] \times \mathbb{T}^d \times \R)$ is a kinetic measure, if
\begin{itemize}
\item[i)] For each $\pi \in C_0([0,T] \times \mathbb{T}^d \times \R)$, the map $m(\pi): \Omega \rightarrow \R$ is measurable.
\item[ii)] $\mathbb{E}[m([0,T] \times \mathbb{T}^d \times B_R^c)] \rightarrow 0$ as $R \rightarrow +\infty,$ for  $B_R^c := \{z \in \R : |z| \ge R\}$.
\end{itemize}
\end{defi}

\begin{defi}[Kinetic solution]\label{def:kinetic-sol} Let $u_0 \in L^2(\mathbb{T}^d)$ be the deterministic initial data. A $L^1(\mathbb{T}^d)-$valued $\{\mathcal{F}_t\}$-adapted stochastic process $(u(t))_{t \in [0,T]}$ is said to be a kinetic solution of \eqref{eq:fractionalSCL} with initial data $u_0$, if for $f(t):= {\bf 1}_{u(t) > \xi} $, there holds
\begin{itemize}
\item[i)] for all $\psi \in C_c(\mathbb{T}^d \times \R)$, $\mathcal{P}$-a.s, $t \mapsto \langle f(t), \psi \rangle$ is a c\'adl\'ag,
\item[ii)]  there exists  a constant $C_2 \ge 0$ such that
$\mathbb{E}\Big[\underset{0 \le t \le T}{sup}\|u(t)\|_{L^2(\mathbb{T}^d)}^2\Big] \le C_2$,
\item[iii)] there exists a random kinetic measure $m$ in the sense of Definition \ref{def:kinetic-mes} such that $\mathcal{P}$-a.s., $m \ge \eta_1$, where $\eta_1\in \mathcal{R}^+([0,T] \times \mathbb{T}^d \times \R)$ is defined by
\begin{align}
    \eta_1(x,t,\xi) := \int_{\R^d}\big|{\tt \Phi}(u(t, x+z)) - {\tt \Phi}(\xi)\big|{\bf 1}_{{\rm Conv}\{u(t,x), u(t,x+z)\}}(\xi)\gamma(z)\,dz,  \label{def:eta_1}
\end{align}
with ${\rm Conv}\{c,d\} := \big(\min\{c,d\}, \max\{c,d\}\big)$,
\item[iv)] the pair $(f, m)$ satisfies the following formulation: for all $\psi \in C_c^2(\mathbb{T}^d \times \R)$, $t \in [0, T]$, 
\begin{align}\label{eq:kinetic-formulation}
 &\langle f(t), \psi \rangle = \,\langle f_0, \psi \rangle  + \int_0^t \langle f(s), F'(\xi)\cdot \nabla\psi \rangle \,ds - \int_0^t  \langle f(s), {\tt \Phi}'(\xi)(-\Delta)^\theta[\psi] \rangle \,ds \\
& +\, \sum_{k=1}^{\infty}\int_0^t\int_{\mathbb{T}^d} {\tt h}_k(x,u(x,s))\psi(x, u(x,s))\,dx\,d\beta_k(s) \notag \\
& + \, \frac{1}{2}\int_0^t\int_{\mathbb{T}^d} \partial_\xi\psi(x,u(x,s))G^2(x, u(x,s))\,dx\,ds - m(\partial_\xi\psi)([0,t]), \quad \mathcal{P} \mbox{-a.s},\notag
\end{align}
where $f_0 = {\bf 1}_{u_0 > \xi},$ and $m(\psi)$ is defined as in \eqref{eq:measure-m}. 
\end{itemize}
\end{defi}
\begin{defi}[Kinetic function] Let $(\mathcal{X}, \nu)$ be a finite measure space. A measurable function $f 
: \mathcal{X} \times \R \rightarrow [0,1]$ is said to be a {\bf kinetic function} if there exists a Young measure $\mathcal{V}$ (see, \cite[Definition 3]{Vovelle2010}) on $\mathcal{X}$ vanishing at infinity such that, for $\nu$-a.e. $\omega \in \mathcal{X}$ and for all $\zeta \in \R$,  $ f(\omega, \zeta) = \mathcal{V}_\omega(\zeta, \infty).$
\end{defi}
\begin{rem}\label{rem:kintetic-formulation}Let $u$ be a kinetic solution of \eqref{eq:fractionalSCL}. If $f(x,t, \xi) = {\bf 1}_{u(x,t) > \xi}$, then we have $\partial_{\xi}f(x,t,\xi) = \delta_{u(x,t) = \xi}$, where $\mathcal{V} = \delta_{u(x,t) = \xi}$ is a Young measure on $\Omega \times [0,T] \times \mathbb{T}^d$. Thus, we can re-write \eqref{eq:kinetic-formulation} as follows: for all $t \in [0,T]$, for all $\psi \in C_c^2(\mathbb{T}^d \times \R)$,
\begin{align}\label{eq:revised-kinetic-formulation}
    &\langle f(t), \psi \rangle = \,\langle f_0, \psi \rangle  + \int_0^t \langle f(s), F'(\xi)\cdot \nabla\psi \rangle \,ds - \int_0^t  \langle f(s), {\tt \Phi}'(\xi)(-\Delta)^\theta[\psi] \rangle \,ds  \\
& +\, \sum_{k=1}^{\infty}\int_0^t\int_{\mathbb{T}^d} {\tt h}_k(x,\xi)\psi(x, \xi)\,d\mathcal{V}_{s,x}(\xi)\,dx\,d\beta_k(s) \notag \\
& + \, \frac{1}{2}\int_0^t\int_{\mathbb{T}^d} \partial_\xi\psi(x,\xi)G^2(x,\xi)\,d\mathcal{V}_{s,x}(\xi)\,dx\,ds - m(\partial_\xi\psi)([0,t]), \quad \mathcal{P}\mbox{-a.s}\,.\notag
\end{align}
\end{rem}
We recall the result of well-posedness of kinetic solution for \eqref{eq:fractionalSCL}; see \cite[Thorem 2.4]{AC_2}.
\begin{thm}[Existence and uniqueness]\label{thm:Existence and uniqueness}
Let the assumptions \ref{A1}-\ref{A3} be true. Let $\theta \in (0,1)$ and $u_0 \in L^2( \mathbb{T}^d)$. Then there exists a unique kinetic solution to \eqref{eq:fractionalSCL} in the sense of Definition \ref{def:kinetic-sol} and it has almost surely continuous trajectories in $L^2(\mathbb{T}^d)$. Moreover, if $u_1, u_2$ be kinetic solutions to \eqref{eq:fractionalSCL} with initial data $u_{1,0}$ and $u_{2,0}$ respectively, then for all $t \in [0,T]$, 
    \begin{align*}
        \mathbb{E}\big[\|u_1(t) - u_2(t)\|_{L^1(\mathbb{T}^d)}\big] \le  \mathbb{E}\big[\|u_{1,0} - u_{2,0}\|_{L^1(\mathbb{T}^d)}\big].
    \end{align*}
\end{thm}
\subsection{Freidlin-Wentzell type LDP} In this subsection, we start by recalling some standard definitions and results for the large deviation theory. Let $\{X^{\eps}\}$ be a family of random variables defined on  $(\Omega, \mathcal{F}, \mathcal{P})$ taking value in a Polish space $(\mathcal{E}, {\tt d})$ with metric ${\tt d}$.
\begin{defi}
A function ${\tt I}$ : $\mathcal{E}$ $\rightarrow$ [0,$\infty$] is called a rate function if ${\tt I}$ is lower semi-continuous. ${\tt I}$ is called a good rate function if the level set $\{x\in \mathcal{E}: {\tt I}(x)\leq N \}$ is compact for each $N <\infty$.
\end{defi}
\begin{defi}[Large deviation principle] The  $\{X^\eps\}_{\eps > 0}$  is said to satisfy the LDP with rate function ${\tt I}$ if for each Borel subset $B$ of $\mathcal{E}$ 
\[\underset{x\in B^0}{-\text{inf}} ({\tt I}(x)) \leq \underset{\eps \longrightarrow 0}{\liminf} (\eps\log\mathbb{P}(X^\eps \in B)) \leq    \underset{\eps \longrightarrow 0}{\limsup} (\eps \log \mathbb{P}(X^\eps \in B)) \leq  \underset{x\in \bar{B}}{-\text{inf}}({\tt I}(x)),\]
where $B^0$ and $\bar{B}$ are interior and closure of $B$ in $\mathcal{E}$ respectively.
\end{defi}
Recall that, $W(t)$ is a Wiener process on a Hilbert space $\mathcal{H}$ and the path of $W$ takes values in $C([0,T], \mathcal{H}_0)$. To proceed further, we define the following sets (see, \cite{Zhang-2020}):
\begin{align}
      &\mathcal{A} := \big\{ \upsilon : \mbox{ $\upsilon$ is a $\mathcal{H}$-valued $\{\mathcal{F}_t\}$-predictable process with $\mathcal{P}$-a.s. $\int_0^T \|\upsilon(s)\|_\mathcal{H}^2\,ds <  \infty$ } \big\}, \notag \\ &S_N := \big\{{\ell} \in L^2([0,T];\mathcal{H}) : \int_0^T\|{\ell}(s)\|_\mathcal{H}^2\,ds\leq N\big\},\notag \\
    &\mathcal{A}_N :=\big\{\upsilon \in \mathcal{A} : \upsilon(\omega) \in S_N ~~\text{for $\mathcal{P}$-a.s. $\omega \in \Omega$}\big\}. \notag
\end{align}
Suppose for each  $\eps> 0$, $\mathcal{G^{\eps}} : C([0,T], \mathcal{H}_0) \rightarrow \mathcal{E}$ is a measurable map and $X^{\eps} := \mathcal{G}^{\eps}(W)$. Now, we present the following sufficient condition for LDP of the sequence $X^\eps$ as $\eps$ $\rightarrow$ 0. 
\begin{cond}\label{cond:1}
There exists a measurable function  $\mathcal{G}^{0} : C([0,T], \mathcal{H}_0) \rightarrow \mathcal{E}$ such that 
\begin{itemize}
\item[i)] for every $N < \infty$, let $\{{\ell}_\eps : \eps > 0\} \subset \mathcal{A}_N$. If ${\ell}_\eps$ converges to ${\ell}$ as $S_N$-valued random element in distribution, then $\mathcal{G}^\eps\Big(W(\cdot) + \frac{1}{\sqrt{\eps}}\displaystyle\int_0^\cdot {\ell}_\eps(s)\,ds\Big)$ converges in distribution to $\mathcal{G}^0\Big(\displaystyle\int_0^\cdot {\ell}(s)\,ds\Big)$,
\item[ii)] for every $N  < \infty$, the set $\Big\{\mathcal{G}^0\Big(\displaystyle\int_0^\cdot {\ell}(s)\,ds\Big):~~ {\ell} \in S_N\Big\}$ is a compact subset of $\mathcal{E}.$
\end{itemize}
\end{cond}
The following theorem regarding the LDP is due to Budhiraja et al. in \cite{Budhiraja}.
\begin{thm}
 If $\mathcal{G}^{\eps}$ satisfies Condition \ref{cond:1}, then $X^\eps$ satisfies LDP on $\mathcal{E}$ with the good rate function ${\tt I}$ defined by
 \begin{equation}\label{eq:good-rate-fun}
     {\tt I}(f) = \begin{cases}
 \underset{\{{\ell} \in L^2([0,T];\mathcal{H}):~~ f =  \mathcal{G}^0(\int_0^{\cdot}{\ell}(s)ds)\}}{\inf}  \Big \{ \frac{1}{2} \displaystyle\int_0^T\|{\ell}(s)\|_\mathcal{H}^2\,ds \Big \}, \quad \forall f \in \mathcal{E}\\
 \infty, \quad \text{if} \quad \big\{{\ell} \in L^2([0,T];\mathcal{H}):~~ f = \mathcal{G}^0(\int_0^{\cdot}{\ell}(s)ds)\big\} = \emptyset.
\end{cases}
 \end{equation}
\end{thm}
Thanks to \cite{Matoussi}, we state another set of sufficient conditions to establish LDP.
\begin{cond}\label{cond:2}
There exists a measurable function  $\mathcal{G}^{0} : C([0,T], \mathcal{H}_0)
\rightarrow \mathcal{E}$ equip with metric ${\tt d}$ such that 
\begin{itemize}
\item[i)] for every  $N < \infty$, any family \{${\ell}_\eps:~\eps>0\} \subset \mathcal{A}_N$ and any $\delta> 0$, 
$\underset{\eps \rightarrow 0}{\text{lim}}\mathcal{P}\big(d(Y_{\eps}, Z_{\eps}) > \delta \big) = 0,$ where $Y_{\eps} := \mathcal{G^{\eps}} \Big( W(\cdot) + \frac{1}{\sqrt{\eps}} \displaystyle\int_0^{\cdot} {\ell}_{\eps}(s)\,ds \Big)$, $Z_\eps := \mathcal{G}^0\Big(\displaystyle \int_0^{.}{\ell}_{\eps}(s)\,ds\Big).$
\item[ii)] for every $N < \infty$ and any family $\big\{{\ell}_\eps; \eps > 0\big\} \subset S_N$ that converges to some element ${\ell}$ as $\eps\rightarrow 0$, $\mathcal{G}^0\Big(\displaystyle\int_0^{\cdot}{\ell}_{\eps}(s)\,ds\Big)$ converges to $\mathcal{G}^0\Big(\displaystyle\int_0^{\cdot} {\ell}(s)\,ds\Big)$ in the space $\mathcal{E}$.
\end{itemize}
\end{cond}
In this article, we wish to study LDP for the following equation driven by small multiplicative noise:
given $\eps>0$, for $(x,t)\in \mathbb{T}^d \times [0,T]$
\begin{equation}\label{eq:small-noise}
 du^\eps + \text{div}_xF(u^\eps(t,x))\,dt + (-\Delta)^\theta[{\tt \Phi}(u^\eps(t,\cdot))](x)\,dt = \sqrt{\eps}{\tt h}(u^\eps(t,x))\,dW(t),  
\end{equation}
with $ u^\eps(0,x) = u_0(x),$ and $u_0 \in L^2(\mathbb{T}^d)$. Under the assumptions \ref{A1}-\ref{A3}, there exists a unique kinetic solution $u^\eps$ of \eqref{eq:small-noise} 
  (cf.~Theorem \ref{thm:Existence and uniqueness}) and a Borel-measurable function 
$$\mathcal{G}^\eps: C([0,T]; \mathcal{H}_0) \rightarrow  \mathcal{E}_1:=L^1([0,T];L^1(\mathbb{T}^d))$$
exists such that $u^\eps(\cdot) = \mathcal{G}^\eps\big(W(\cdot)\big).$ For ${\ell} \in L^2([0,T],\mathcal{H})$, we consider the following skeleton equation
\begin{equation}\label{eq:skeleton}
     \displaystyle du_{\ell}+ \text{div}_xF(u_{\ell}(t,x))\,dt + (-\Delta)^\theta[{\tt \Phi}(u_{\ell}(t,\cdot))](x)\,dt = {\tt h}(u_{\ell}(t,x)){\ell}(t)\,dt,
\end{equation}
with $ u_{\ell}(0,x) = u_0(x)$. The solution $u_{\ell}$ of \eqref{eq:skeleton} (for existence and uniqueness, see Section \ref{sec:skeleton}) gives a measurable mapping
$$\mathcal{G}^0: C([0,T]; \mathcal{H}_0) \rightarrow \mathcal{E}_1$$
such that $\mathcal{G}^0\Big(\displaystyle\int_0^{\cdot}{\ell}(s)\,ds\Big):= u_{\ell}(\cdot)$.
Now we state one of the main results of this paper.
\begin{thm}\label{thm:ldp}
Let the assumptions \ref{A1}-\ref{A3} be true and $u_0 \in L^2(\mathbb{T}^d)$. Then $u^\eps$ satisfies the large deviation principle on $\mathcal{E}_1$ with good rate function ${\tt I}$ defined in \eqref{eq:good-rate-fun}.   
\end{thm}
\subsection{Central limit theorem} For any $\eps \in (0,1)$, let $u^{\eps}$ be the unique kinetic solution to \eqref{eq:small-noise} with a constant initial value $u_0={\pmb a}$~(for simplicity, we take ${\pmb a} =1$). Let $\hat{u}$ be the kinetic solution of the deterministic factional conservation laws
\begin{equation}\label{eq:deterministic-FCL}
     \displaystyle d\hat{u} + \text{div}F(\hat{u}(t,x))\,dt + (-\Delta)^\theta[{\tt \Phi}(\hat{u}(t,\cdot))](x)\,dt = 0, \quad 
    \hat{u}(0,x) = 1\,.
\end{equation}
Our aim is to study the fluctuation behaviour of $\frac{1}{\sqrt{\eps}}(u^\eps - \hat{u})$.
Consider the SPDE
\begin{equation}\label{eq:star-FSCL}
     \displaystyle du^\star + \text{div}(F^\prime(\hat{u})u^\star)\,dt + (-\Delta)^\theta[{\tt \Phi}'(\hat{u})u^\star(t,\cdot)](x)\,dt = {\tt h}(\hat{u})\,dW(t),\quad 
    u^\star(0,x) = 0\,,
\end{equation}
where $\hat{u}$ is the unique kinetic solution of \eqref{eq:deterministic-FCL}. Since $\hat{u}= 1$ (see, Section \ref{sec:CLT}), SPDE \eqref{eq:star-FSCL} becomes a fractional conservation laws driven by additive noise with flux function ${\pmb F}(\xi) = F'(1)\xi$ and the degenerate term ${\pmb \Phi}(\xi) = {\tt \Phi}'(1)\xi$. Hence, thanks to Theorem \ref{thm:Existence and uniqueness}, equation \eqref{eq:star-FSCL} has a unique kinetic solution $u^\star$. Now we state the main result regarding CLT for \eqref{eq:fractionalSCL}.
\begin{thm}\label{thm:clt}
Let the assumptions \ref{A1}-\ref{A3} be true and $F''$ and ${\tt \Phi''}$ are bounded. Let $u^\eps$ and $\hat{u}$ be the unique kinetic solution to \eqref{eq:small-noise} and \eqref{eq:deterministic-FCL} respectively with constant initial data. Then, it holds that
\begin{align}\label{clt:main-result}
   \underset{\eps \rightarrow 0}{\text{lim}}\,\mathbb{E}\Big[\big\|\frac{u^\eps - \hat{u}}{\sqrt{\eps}} - u^\star\big\|_{\mathcal{E}_1}\Big] = 0\,, 
\end{align}
where $u^\star$ is a kinetic solution to \eqref{eq:star-FSCL}.
\end{thm}
\subsection{Moderate deviation principle}
For $\eps>0$, consider the deviation scale $\Lambda(\eps)$ satisfying 
$\Lambda(\eps) \rightarrow +\infty,~\sqrt{\eps}\Lambda(\eps) \rightarrow 0$ as $\eps \rightarrow 0$. Let $z^\eps:= \frac{u^\eps -\Hat{u}}{\sqrt{\eps}\Lambda(\eps)}$, where
$u^\eps, \hat{u}$ are stated in Theorem \ref{thm:clt}. Then $z^\eps$
satisfies the following SPDE:
\begin{equation}\label{eq:MDP-small-noise-1}
\begin{cases}
    \displaystyle dz^\eps + \text{div}\Big(\frac{F(u^\eps) -F(\Hat{u})}{\sqrt{\eps}\Lambda(\eps)}\Big)\,dt + (-\Delta)^\theta\Big[\frac{{\tt \Phi}(u^\eps) -{\tt \Phi}(\Hat{u})}{\sqrt{\eps}\Lambda(\eps)}\Big](x)\,dt\\ = \Lambda(\eps)^{-1}{\tt h}\big(\sqrt{\eps}\Lambda(\eps)z^\eps + \Hat{u}\big)\,dW(t), \quad
    z^\eps(0) = 0\,.
\end{cases}  
\end{equation}
Since $\Hat{u} = 1$,
we define
\begin{align*}
    &\Bar{F}(\xi) := \frac{F(\sqrt{\eps}\Lambda(\eps)\xi +1) -F(1)}{\sqrt{\eps}\Lambda(\eps)}, \quad \Bar{{\tt \Phi}}(\xi) := \frac{{\tt \Phi}(\sqrt{\eps}\Lambda(\eps)\xi +1) -{\tt \Phi}(1)}{\sqrt{\eps}\Lambda(\eps)}.
\end{align*}
With $\Bar{F}(\cdot)$ and $\Bar{{\tt \Phi}}(\cdot)$ in hand, we re-write the SPDE \eqref{eq:MDP-small-noise-1} as 
\begin{align}\label{eq:mdp-small-noise-2}
    \displaystyle dz^\eps + \text{div}\Bar{F}(z^\eps)\,dt + (-\Delta)^\theta[\Bar{{\tt \Phi}}(z^\eps)](x)\,dt = \Lambda(\eps)^{-1}{\tt h}\big(\sqrt{\eps}\Lambda(\eps)z^\eps + \Hat{u}\big)\,dW(t),
\end{align}
with $z^\eps(0) = 0$. Since $\Bar{F},\Bar{{\tt \Phi}}$ satisfy the assumptions \ref{A1} and \ref{A2} respectively, by Theorem \ref{thm:Existence and uniqueness}, we infer the well-posedness of kinetic solution of \eqref{eq:mdp-small-noise-2}--- which then implies existence of a measurable map $\mathcal{\tt G}^\eps : C([0,T]; \mathcal{H}_0) \rightarrow \mathcal{E}_1$ such that $\mathcal{\tt G}^\eps(W(\cdot)):= z^\eps(\cdot).$ For any ${\ell} \in L^2([0,T]; \mathcal{H})$, we consider the following skeleton equation,
\begin{equation}\label{eq:mdp-skeleton}
        \displaystyle dz_{\ell} + \text{div}(F'(\Hat{u})z_{\ell})\,dt + (-\Delta)^\theta[{\tt \Phi}'(\Hat{u})(z_{\ell})](x)\,dt = {\tt h}(\hat{u}){\ell}(t)\,dt, \quad 
        z_{\ell}(0) = 0.
\end{equation}
Note that \eqref{eq:mdp-skeleton} is a special case of \eqref{eq:skeleton} and hence exhibits a unique kinetic solution which then yields a measurable map $\mathcal{\tt G}^0:C([0,T]; \mathcal{H}_0) \rightarrow \mathcal{E}_1$ such that $\mathcal{\tt G}^0\big(\int_0^{\cdot}{\ell}(s)\,ds\big) := z_{\ell}(\cdot).$ Now we are in a position to state one of the main results of this article regarding the MDP of \eqref{eq:fractionalSCL}
\begin{thm}\label{thm:mdp} Let the assumptions \ref{A1}-\ref{A3} be true, the initial data $u_0 = 1$ and $F''$ and ${\tt \Phi''}$ are bounded. Then $z^\eps$
satisfies the LDP on $\mathcal{E}_1$ with speed $\Lambda^2(\eps)$ and with rate function ${\tt I}$ defined in \eqref{eq:good-rate-fun}, i.e.
\begin{itemize}
    \item [i)] For any closed subset $A$ of $\mathcal{E}_1$, 
    $ \quad \underset{\eps \rightarrow 0}{\limsup}\,\frac{1}{\Lambda^2(\eps)}\log \mathbb{P}\big(z^\eps \in A\big) \le -\underset{x \in A}{\inf}{\tt I}(x)$,
    \item [ii)] For any open subset $B$ of $\mathcal{E}_1$, 
    $\quad \underset{\eps \rightarrow 0}{\limsup}\,\frac{1}{\Lambda^2(\eps)}\log \mathbb{P}\Big(z^\eps \in B\Big) \ge -\underset{x \in B}{\inf} {\tt I}(x)$.
\end{itemize}
\end{thm}

\section{Wellposedness of Skeleton Equation}\label{sec:skeleton} In this section, we will establish the well-posedness result for the solution of skeleton equation \eqref{eq:skeleton}. To do so, we will introduce the definition of solution of skeleton equation.
\begin{defi}[Kinetic solution of skeleton equation] \label{defi:kinetic-solution-skeleton-equation}
Let $u_0 \in L^2(\mathbb{T}^d)$. A measurable function $u_{\ell} : \mathbb{T}^d \times [0,T] \rightarrow \R$ is said to be a kinetic solution to \eqref{eq:skeleton}, if 
\begin{itemize}
\item[i)]  there exist a constant $C_{{\ell}} > 0$ such that $\underset{0 \le t \le T}{sup} \|u_{\ell}(t)\|_{L^2(\mathbb{T}^d)}^2 \le C_{\ell}$ and $m_{\ell} \in \mathcal{R}^+(\mathbb{T}^d \times[0,T]\times \R)$ such that $m_{\ell}\ge \eta_{\ell}$, where 
$$ \eta_{\ell}(x,t,\xi):= \int_{\R^d}\big|{\tt \Phi}(u( x+z,t)) - {\tt \Phi}(\xi)\big|{\bf 1}_{{\rm Conv}\{u_{\ell}(x,t), u_{\ell}(x+z,t)\}}(\xi)\gamma(z)\,dz,$$
\item[ii)] for all $\psi \in C_c^2(\mathbb{T}^d \times \R)$, $t \in [0,T],$ there holds 
\begin{align}\label{eq:skeleton-k-f}
\langle f_{\ell}(t), \psi \rangle = &\,\langle f_0, \psi \rangle  + \int_0^t \langle f_{\ell}(s), F'(\xi)\cdot \nabla\psi \rangle \,ds - \int_0^t  \langle f_{\ell}(s), {\tt \Phi}'(\xi)(-\Delta)^\theta[\psi] \rangle \,ds \notag \\
& +\, \sum_{k=1}^{\infty}\int_0^t\int_{\mathbb{T}^d} {\tt h}_k(x,u(x,s))\psi(x, u(x,s)){\ell}_k(s)\,dx\,ds
- m_{\ell}(\partial_\xi\psi)([0,t]),
\end{align}
where $f_0:= {\bf 1}_{u_0 > \xi},~~f_{\ell}:= {\bf 1}_{u_{\ell} > \xi}$, for fix ${\ell} \in S_N$, ${\ell}(t)= \sum_{n\ge 1}{\ell}_n(t)e_n$.
\end{itemize}
\end{defi}
Thanks to the deterministic counter-part of \cite[Proposition 8]{Vovelle2010}, left limit of $f_{\ell}$, denoted by $f_{\ell}^-(t)$ is a kinetic function and  $f_{\ell}(t) = f_{\ell}^-(t)$ for a.e. $t \in [0,T].$ We define the conjugate function $\Bar{f}_{\ell}$ of the kinetic function $f_{\ell}$ as $\Bar{f}_{\ell} = 1 - f_{\ell}$. We also define the right limit of $f_{\ell}$, denoted by $f_{\ell}^+$, which is simply $f_{\ell}$. Similar to Remark \ref{rem:kintetic-formulation}, we can re-write \eqref{eq:skeleton-k-f} in the following form: for all $\psi \in C_c^2(\mathbb{T}^d \times \R)$ and $t \in [0,T]$,
\begin{align}\label{eq:skeleton-k-ref}
&\langle f_{\ell}^+(t), \psi \rangle = \,\langle f_0, \psi \rangle  + \int_0^t \langle f_{\ell}(s), F'(\xi)\cdot \nabla\psi \rangle \,ds - \int_0^t  \langle f_{\ell}(s), {\tt \Phi}'(\xi)(-\Delta)^\theta[\psi] \rangle \,ds \\
&\quad +\, \sum_{k=1}^{\infty}\int_0^t\int_{\mathbb{T}^d}\int_{\R} {\tt h}_k(x,\xi))\psi(x,\xi){\ell}_k(s)\,d\mathcal{V}_{x,s}^{\ell}(\xi)\,dx\,ds
- m_{\ell}(\partial_\xi\psi)([0,t])\,,\notag
\end{align}
where $d\mathcal{V}_{x,s}^{\ell}(\xi) = \delta_{u_{\ell}(x,s) = \xi}$ is a Young measure on $\mathbb{T}^d \times [0,T]$.
\subsection{Uniqueness of Skeleton equation}
Following the doubling variables technique and similar lines of argument as given in proof of \cite[Proposition $3.2$]{AC_2}, \cite[Proposition $9$]{Vovelle2010} and \cite[Proposition $4.2$]{Zhang-2020}, we arrive at the following result.
\begin{prop}\label{prop:doubling-variable}
Let the assumptions \ref{A1}-\ref{A3} be true. 
Let $u_1(t)$ and $u_2(t)$ be solution to \eqref{eq:skeleton} with initial data $u_{1,0}$ and $u_{2,0}$ respectively and $f_1(t):= {\bf 1}_{u_1(t) > \xi}$ and $f_2(t):= {\bf 1}_{u_2(t) > \zeta}$. Then, for all  $t \in [0,T]$ and non-negative functions $\varrho \in C^\infty(\mathbb{T}^d)$ and $\rho \in C_c^\infty(\R)$,
\begin{align}\label{eq:doubling-variable}
   & \int_{(\mathbb{T}^d)^2}\int_{\R^2} \varrho(x-y)\rho(\xi-\zeta)\Big(f_1^{\pm}(x,t,\xi)\bar{f}_2^{\pm}(y,t,\zeta) + \bar{f}_1^{\pm}(x,t,\xi)f_2^{\pm}(y,t,\zeta)\Big)\,d\xi\,d\zeta\,dx\,dy \\
   &\,\le\, \int_{(\mathbb{T}^d)^2}\int_{\R^2} \varrho(x-y)\rho(\xi-\zeta)\Big(f_{1,0}(x,\xi)\bar{f}_{2,0}(y,\zeta) + \bar{f}_{1,0}(x,\xi)f_{2,0}(y,\zeta)\Big)\,d\xi\,d\zeta\,dx\,dy \notag\\
   & \quad + \mathcal{S}_1 + \Bar{\mathcal{S}_1} + \mathcal{S}_2 + \Bar{\mathcal{S}_2} + 2\mathcal{S}_3\,,\notag
   \end{align}
   where $f_{1,0} := {\bf 1}_{u_{1,0} > \xi}$, $f_{2,0} := {\bf 1}_{u_{2,0} > \zeta}$, and 
   \begin{align*}
       &\mathcal{S}_1 = \int_0^t\int_{(\mathbb{T}^d)^2 \times \R^2}f_1(x,s,\xi)\Bar{f}_2(y,s,\zeta)\big(F'(\xi)-F'(\zeta)\big)\rho(\xi-\zeta)\nabla_x\varrho(x-y)\,d\xi\,d\zeta\,dx\,dy\,ds\\
       &\Bar{\mathcal{S}_1} = \int_0^t\int_{(\mathbb{T}^d)^2\times\R^2}\Bar{f}_1(x,s,\xi)f_2(y,s,\zeta)\big(F'(\xi)-F'(\zeta)\big)\rho(\xi-\zeta)\nabla_x\varrho(x-y)\,d\xi\,d\zeta\,dx\,dy\,ds\\
       &\mathcal{S}_2 = - \int_0^t\int_{(\mathbb{T}^d)^2\times \R^2}f_1(x,s,\xi)\Bar{f}_2(y,s,\zeta)\rho(\xi-\zeta){\tt \Phi}'(\xi)(-\Delta)^\theta[\varrho(\cdot-y)](x)\,d\xi\,d\zeta\,dx\,dy\,ds\\
       &- \int_0^t\int_{(\mathbb{T}^d)^2\times\R^2}f_1(x,s,\xi)\Bar{f}_2(y,s,\zeta)\rho(\xi-\zeta){\tt \Phi}'(\zeta)(-\Delta)^\theta[\varrho(x-\cdot)](y)\,d\xi\,d\zeta\,dx\,dy\,ds\\  
       &- \int_0^t\int_{(\mathbb{T}^d)^2\times\R^2}f_1(x,s,\xi)\partial_\xi\rho(\xi-\zeta)\varrho(x-y)\,d\eta_{{\ell},2}(y,s,\zeta)\,dx\,d\xi\\ 
       &+ \int_0^t\int_{(\mathbb{T}^d)^2\times\R^2}\Bar{f}_2(y,s,\zeta)\partial_\zeta\rho(\xi-\zeta)\varrho(x-y)\,d\eta_{{\ell},1}(x,s,\xi)\,dy\,d\zeta, \\
        &\Bar{\mathcal{S}}_2 = - \int_0^t\int_{(\mathbb{T}^d)^2\times\R^2}\Bar{f}_1(x,s,\xi)f_2(y,s,\zeta)\rho(\xi-\zeta){\tt \Phi}'(\xi)(-\Delta)^\theta[\varrho(\cdot-y)](x)\,d\xi\,d\zeta\,dx\,dy\,ds\\
       &- \int_0^t\int_{(\mathbb{T}^d)^2\times\R^2}\Bar{f}_1(x,s,\xi)f_2(y,s,\zeta)\rho(\xi-\zeta){\tt \Phi}'(\zeta)(-\Delta)^\theta[\varrho(x-\cdot)](y)\,d\xi\,d\zeta\,dx\,dy\,ds\\  
       &+ \int_0^t\int_{(\mathbb{T}^d)^2\times\R^2}\Bar{f}_1(x,s,\xi)\partial_\xi\rho(\xi-\zeta)\varrho(x-y)\,d\eta_{{\ell},2}(y,s,\zeta)\,dx\,d\xi\\ 
       &- \int_0^t\int_{(\mathbb{T}^d)^2\times\R^2}f_2(y,s,\zeta)\partial_\zeta\rho(\xi-\zeta)\varrho(x-y)\,d\eta_{{\ell},1}(x,s,\xi)\,dy\,d\zeta,        
   \end{align*}
   with $m_1$ and $m_2$ corresponding kinetic measure of $u_1(t)$ and $u_2(t)$ respectively satisfying 
   $m_1\ge \eta_{{\ell},1},~ m_2 \ge \eta_{{\ell},2}$, where
   \begin{align*}
       &\eta_{{\ell},1}(x,t,\xi) = \int_{\R^d}\big|{\tt \Phi}(u_1( x+z,t)) - {\tt \Phi}(\xi)\big|{\bf 1}_{{\rm Conv}\{u_1(x,t), u_1(x+z,t)\}}(\xi)\gamma(z)\,dz,\\
       &\eta_{{\ell},2}(y,t,\zeta) = \int_{\R^d}\big|{\tt \Phi}(u_2(y+z,t)) - {\tt \Phi}(\zeta)\big|{\bf1}_{{\rm Conv}\{u_2(y,t), u_2(y+z,t)\}}(\zeta)\gamma(z)\,dz,
   \end{align*}
   and
   \begin{align*}
       &\mathcal{S}_3 = \sum_{k \ge 1}\int_0^t\int_{(\mathbb{T}^d)^2\times\R^2}\varrho\varphi\big({\tt h}_k(x,\xi)-{\tt h}_k(y,\zeta)\big){\ell}_k(s)\,d\mathcal{V}_{x,s}^1 \otimes d\mathcal{V}_{y,s}^2(\xi,\zeta)\,dx\,dy\,ds \\
  & \text{with}~~ \varphi(\xi,\zeta): = \int_{-\infty}^\xi\rho(\xi'-\zeta) \,d\xi' = \int_{-\infty}^{\xi-\zeta}\rho(y)\,dy, \notag \\
   & \mathcal{V}_{x,s}^1(\xi):  = \partial_\xi\Bar{f}_1^+(x,s,\xi),\quad 
   \mathcal{V}_{y,s}^2(\zeta):= -\partial_\zeta f_2^{+}(y,s,\zeta)\,.
   \end{align*}
\end{prop}
With the help of Proposition \ref{prop:doubling-variable}, we establish the uniqueness of skeleton equation.
\begin{thm}\label{thm:contraction-principle} Under the assumptions \ref{A1}-\ref{A3}, there exists at most one  kinetic solution to \eqref{eq:skeleton} with initial function $u_0$.
\end{thm}
\begin{proof}Let $\varrho_\delta$  and $\rho_{\lambda}$ be the approximations to identity on $\mathbb{T}^d$ and $\R$ respectively. To be more precise, let $\varrho \in C^\infty(\mathbb{T}^d), \rho \in C_c^\infty(\R)$ be symmetric non-negative functions such that $\int_{\mathbb{T}^d}\varrho = 1, \int_{\R}\rho = 1$ and ${\rm supp}\,\rho \subset (-1,1).$  Define, for $\delta, \lambda>0$,
\begin{align*}
    \varrho_{\delta}(x) = \frac{1}{\delta^d}\varrho(\frac{x}{\delta}), \quad \rho_{\lambda}(\xi) = \frac{1}{\lambda}\rho(\frac{\xi}{\lambda}).
\end{align*}
Substituting $\varrho = \varrho_\delta$ and $\rho = \rho_\lambda$ in Proposition \ref{prop:doubling-variable}, we get, from \eqref{eq:doubling-variable}
\begin{align}\label{eq:-delta-lamada-doubling-variable}
   & \int_{(\mathbb{T}^d)^2}\int_{\R^2} \varrho_\delta(x-y)\rho_\lambda(\xi-\zeta)\Big(f_1^{\pm}(x,t,\xi)\bar{f}_2^{\pm}(y,t,\zeta) + \bar{f}_1^{\pm}(x,t,\xi)f_2^{\pm}(y,t,\zeta)\Big)\,d\xi\,d\zeta\,dx\,dy \\
   &\,\le\, \int_{(\mathbb{T}^d)^2}\int_{\R^2} \varrho_\delta(x-y)\rho_\lambda(\xi-\zeta)\Big(f_{1,0}(x,\xi)\bar{f}_{2,0}(y,\zeta) + \bar{f}_{1,0}(x,\xi)f_{2,0}(y,\zeta)\Big)\,d\xi\,d\zeta\,dx\,dy \notag \\ 
   & \hspace{2cm}+ \mathcal{S}_{1,\lambda, \delta} + \Bar{\mathcal{S}}_{1,\lambda, \delta} + \mathcal{S}_{2, \lambda, \delta} + \Bar{\mathcal{S}}_{2,\lambda, \delta} + 2\mathcal{S}_{3, \lambda, \delta}, \notag
   \end{align}
where $ \mathcal{S}_{i,\lambda, \delta}$, for $1 \le i \le 3$ and $\Bar{\mathcal{S}}_{j,\lambda, \delta}$, for $j =1,2$  are corresponding to  $\mathcal{S}_i$ and $\Bar{\mathcal{S}}_{j}$ as in \eqref{eq:doubling-variable}.
Following the calculations as done in \cite[Theorem 11]{Vovelle2010} and \cite[Theorem 3.3]{AC_2}, we estimate $\mathcal{S}_{i,\lambda, \delta }$ and  $\Bar{\mathcal{S}}_{i,\lambda, \delta}$  for $i = 1,2$ as 
\begin{align}\label{eq:bound-s-12}
    |\mathcal{S}_{1,\lambda, \delta }| + |\Bar{\mathcal{S}}_{1,\lambda, \delta}| \le C\delta^{-1}\lambda, \quad \text{and} \quad |\mathcal{S}_{2,\lambda, \delta}| + |\Bar{\mathcal{S}}_{2,\lambda, \delta}| \le C(r^{2-2\theta}\delta^{-2} + r^{-2\theta}\lambda)
\end{align}
for fixed but arbitrary $r>0$. We follow a similar lines of argument as in \cite[Theorem 4.3]{Zhang-2020} to bound $\mathcal{S}_{3, \lambda, \delta}$. We have 
\begin{align}\label{eq:bound-s-3}
  \mathcal{S}_{3, \lambda, \delta} \le & C\,\int_0^t|{\ell}(s)|_{\mathcal{H}}\int_{(\mathbb{T}^d)^2}\int_{\R^2}\varrho_\delta(x-y)\rho_\lambda(\xi-\zeta)\Big\{f_1^{\pm}(x,t,\xi)\bar{f}_2^{\pm}(y,t,\zeta) \\
  & \hspace{2cm}+ \bar{f}_1^{\pm}(x,t,\xi)f_2^{\pm}(y,t,\zeta)\Big\}\,ds\,d\xi\,d\zeta\,dx\,dy  + C(\delta + 2\lambda)\,.\notag
\end{align}
Combining \eqref{eq:bound-s-12} and \eqref{eq:bound-s-3} in \eqref{eq:-delta-lamada-doubling-variable} and applying Gronwall's lemma, we have
\begin{align}\label{eq:application-grownwall}
     & \int_{(\mathbb{T}^d)^2}\int_{\R^2} \varrho_\delta(x-y)\rho_\lambda(\xi-\zeta)\Big(f_1^{\pm}(x,t,\xi)\bar{f}_2^{\pm}(y,t,\zeta) + \bar{f}_1^{\pm}(x,t,\xi)f_2^{\pm}(y,t,\zeta)\Big)\,d\xi\,d\zeta\,dx\,dy \\
&\le \, \int_{(\mathbb{T}^d)}\int_{\R}\Big(f_{1,0}(x,\xi)\bar{f}_{2,0}(x,\xi) + \bar{f}_{1,0}(x,\xi)f_{2,0}(x,\xi)\Big)\,d\xi\,dx  + \Upsilon_0(\delta, \lambda) \notag \\
 &+ C\big(\delta^{-1}\lambda + 
r^{2-2\theta}\delta^{-2} + r^{-2\theta}\lambda + \delta + \lambda\big)  \notag \\&+\, C\,\int_0^t|{\ell}(s)|_{\mathcal{H}}\int_{(\mathbb{T}^d)^2}\int_{\R^2}\varrho_\delta(x-y)\rho_\lambda(\xi-\zeta)\Big(f_1^{\pm}(x,t,\xi)\bar{f}_2^{\pm}(y,t,\zeta)\notag  \\ &+ \bar{f}_1^{\pm}(x,t,\xi)f_2^{\pm}(y,t,\zeta)\Big)\,d\xi\,d\zeta\,dx\,dy \notag \\ &\le \Big[\int_{(\mathbb{T}^d)}\int_{\R}\Big(f_{1,0}(x,\xi)\bar{f}_{2,0}(x,\xi) + \bar{f}_{1,0}(x,\xi)f_{2,0}(x,\xi)\Big)\,d\xi\,dx  + \Upsilon_0(\delta, \lambda) \notag \\
 &+ C\big(\delta^{-1}\lambda + 
r^{2-2\theta}\delta^{-2} + r^{-2\theta}\lambda + \delta + \lambda\big)\Big]\times \exp\Big\{C\,\int_0^t |{\ell}(s)|_{\mathcal{H}}\Big\},\notag
\end{align}
where 
\begin{align}\label{eq:Upsilon}
&\Upsilon_t(\delta, \lambda):= \int_{(\mathbb{T}^d)^2}\int_{\R^2}\Big(f_1^{\pm}(x,t,\xi)\bar{f}_2^{\pm}(y,t,\zeta)+ \bar{f}_1^{\pm}(x,t,\xi)f_2^{\pm}(y,t,\zeta)\Big)\varrho_\delta(x-y)\\ & \times\rho_\lambda(\xi-\zeta)\,d\xi\,d\zeta\,dx\,dy  - \int_{(\mathbb{T}^d)}\int_{\R}\Big(f_1^{\pm}(x,t,\xi)\bar{f}_2^{\pm}(x,t,\xi) + \bar{f}_1^{\pm}(x,t,\xi)f_2^{\pm}(x,t,\xi)\Big)\,d\xi\,dx\,,\notag 
\end{align}
for $t \in [0,T]$. One can easily check that (cf.~\cite[Equation 3.16]{Vovelle-2018})
$$ \underset{\delta, \lambda \rightarrow 0}{lim}\, \Upsilon_t(\delta, \lambda) = 0, \quad t\in [0,T].$$
Since $h \in S_N$, we have, from \eqref{eq:application-grownwall} 
\begin{align}\label{eq:contraction-1}
     &\int_{(\mathbb{T}^d)}\int_{\R}\Big(f_1^{\pm}(x,t,\xi)\bar{f}_2^{\pm}(x,t,\xi) + \bar{f}_1^{\pm}(x,t,\xi)f_2^{\pm}(x,t,\xi)\Big)\,d\xi\,dx  \\ &
\le C_{T,{\ell}} \Bigg\{ \int_{(\mathbb{T}^d)}\int_{\R}\Big(f_{1,0}(x,\xi)\bar{f}_{2,0}(x,\xi) + \bar{f}_{1,0}(x,\xi)f_{2,0}(x,\xi)\Big)\,d\xi\,dx  + \Upsilon_0(\delta, \lambda) \notag \\
 &+ C\big(\delta^{-1}\lambda + 
r^{2-2\theta}\delta^{-2} + r^{-2\theta}\lambda + \delta + \lambda\big) \Bigg\} - \Upsilon_t(\delta, \lambda)\,.\notag
\end{align}
First letting $\lambda \rightarrow 0$, then sending $r \rightarrow 0$ and $\delta \rightarrow 0$ in \eqref{eq:contraction-1} yields
\begin{align}\label{eq:contraction-2}
     &\int_{(\mathbb{T}^d)}\int_{\R}\Big(f_1^{\pm}(x,t,\xi)\bar{f}_2^{\pm}(x,t,\xi) + \bar{f}_1^{\pm}(x,t,\xi)f_2^{\pm}(x,t,\xi)\Big)\,d\xi\,dx  \\ &
     \le C_{T,h}\Big[\int_{(\mathbb{T}^d)}\int_{\R}\Big(f_{1,0}(x,\xi)\bar{f}_{2,0}(x,\xi) + \bar{f}_{1,0}(x,\xi)f_{2,0}(x,\xi)\Big)\,d\xi\,dx \Big].\notag
\end{align}
Since 
\begin{align}
  \int_{\R}f_1\Bar{f}_2\,d\xi = (u_1 -u_2)^+, \quad  \int_{\R}\Bar{f}_1f_2\,d\xi = (u_1 -u_2)^-, \label{inq:f1barf2} 
\end{align}
we conclude from \eqref{eq:contraction-2} that $
    \|u_1(t) - u_2(t)\|_{L^1(\mathbb{T}^d)} = 0$ for all $ t\in [0,T]$. 
This establishes the uniqueness of kinetic solution of \eqref{eq:skeleton}.
\end{proof}
\subsection{\bf Viscous problem of \eqref{eq:skeleton} and its well-posedness} In order to establish the existence of kinetic solution for skeleton equation, we first study associated viscous problem and then send viscous parameter to zero. For given $\Bar{\eps}>0$, viscous problem of \eqref{eq:skeleton} reads as: 
\begin{equation}\label{eq:regularized-skeleton}
\begin{cases}
     \displaystyle du_{{\ell}, \Bar{\eps}}+ \text{div}F(u_{{\ell}, \Bar{\eps}}(t,x))\,dt + (-\Delta)^\theta[{\tt \Phi}(u_{{\ell}, \Bar{\eps}}(t,x))]\,dt \\ = \Bar{\eps}\Delta u_{{\ell}, \Bar{\eps}}\,dt + {\tt h}(u_{{\ell}, \Bar{\eps}}(t,x)){\ell}(t)\,dt, \quad
    u_{{\ell},\Bar{\eps}}(0,x) = u_0^{\Bar{\eps}}(x),~~ (x,t)\in \mathbb{T}^d \times[0,T],
\end{cases}
\end{equation}
where initial function $u_0^{\Bar{\eps}} \in H^1(\mathbb{T}^d)$ such that $u_0^{\Bar{\eps}} \rightarrow u_0$ in $L^2(\mathbb{T}^d)$.
\subsubsection{\bf Singular perturbation problem of \eqref{eq:regularized-skeleton} and its well-posedness}
In order to establish well-posedness result for \eqref{eq:regularized-skeleton}, we first denote $u_{\gamma, \ell,\Bar{\eps}} := u_{\gamma,\Bar{\eps}}$ and consider the following singular perturbation problem of \eqref{eq:regularized-skeleton}.
\begin{align}\label{eq:singular perturbation-skeleton}
   & du_{\gamma,\Bar{\eps}}+ \text{div}F(u_{\gamma,\Bar{\eps}}(t,x))\,dt + (-\Delta)^\theta[{\tt \Phi}(u_{\gamma,\Bar{\eps}}(t,\cdot))](x)\,dt + \gamma\Delta^2u_{\gamma,\Bar{\eps}}\,dt \\ &\hspace{1cm}=\, \Bar{\eps}\Delta u_{\gamma,\Bar{\eps}}\,dt + {\tt h}(u_{\gamma,\Bar{\eps}}(t,x)){\ell}(t)\,dt, \quad u_{\gamma,\Bar{\eps}}(0) = u_0^{\Bar{\eps}} \in H^1(\mathbb{T}^d)\,. \notag
\end{align}
 Then, we derive suitable a-priori estimates and use compactness argument to prove the existence of viscous solution. To do so,  we adhere to the work of Pr\'evot and and Rockner \cite[Chapter 4.1]{Rockner}, in order to establish the  well-posedness of strong solution of \eqref{eq:singular perturbation-skeleton}. In this regard, we consider the operator 
\begin{align*}
  &A(t,u_{\gamma,\Bar{\eps}}): = \Bar{\eps}\Delta u_{\gamma,\Bar{\eps}} - \gamma\Delta^2u_{\gamma,\Bar{\eps}} -\text{div}F(u_{\gamma,\Bar{\eps}}(t,x)) \\&\hspace{6cm}- (-\Delta)^\theta[{\tt \Phi}(u_{\gamma,\Bar{\eps}}(t,\cdot))](x) + {\tt h}(u_{\gamma,\Bar{\eps}}(t,x)){\ell}(t),  
\end{align*} 
which could be understood in the following weak sense
\begin{align*}
& \langle A(t,u_{\gamma,\Bar{\eps}}), v \rangle := - \int_{\mathbb{T}^d}\Big[\Bar{\eps}\nabla u_{\gamma,\Bar{\eps}}\nabla v + \gamma\Delta u_{\gamma,\Bar{\eps}}\Delta v - F(u_{\gamma,\Bar{\eps}})\nabla v\notag \\
& \hspace{1cm} +(-\Delta)^\frac{\theta}{2}[{\tt \Phi}(u_{\gamma,\Bar{\eps}})](-\Delta)^\frac{\theta}{2}[v] + {\tt h}(u_{\gamma,\Bar{\eps}}(t,x)){\ell}(t)v\big]\,dx, \quad \forall~ u_{\gamma,\Bar{\eps}},v \in H^2(\mathbb{T}^d)\,.
\end{align*}
\begin{lem}
For any $\gamma \in (0,1)$, there exists a unique strong solution $u_{\gamma, \Bar{\eps}}$ to the regularized problem \eqref{eq:singular perturbation-skeleton} , such that $u_{\gamma, \Bar{\eps}} \in L^2([0,T];H^2(\mathbb{T}^d))$. 
\end{lem}
\begin{proof} Observe that, assuming the noise coefficient $B$ to be $0$ in \cite[Eq. 4.2.1]{Rockner}
and \cite[Theorem 4.2.4]{Rockner} in hand, it is sufficient to validate the conditions $(H1)$-$(H4)$ of \cite[page. 66]{Rockner} for operator $A(t,x)$, to obtain the well-possedness of strong solution of \eqref{eq:singular perturbation-skeleton}. The conditions are namely,  the hemicontinuity, the weak monotonicity, the coercivity and the boundedness.
\vspace{0.1cm}

\noindent To verify the  hemicontinuity property $(H1)$, for $u_{\gamma,\Bar{\eps}}, v, w \in H^2(\mathbb{T}^d)$, we see that the map
\begin{align*}
    &\lambda \mapsto \int_{\mathbb{T}^d}\Big[\Bar{\eps}\nabla (u_{\gamma,\Bar{\eps}} + \lambda w )\nabla v + \gamma\Delta (u_{\gamma,\Bar{\eps}} + \lambda w)\Delta v - F(u_{\gamma,\Bar{\eps}} + \lambda w)\nabla v \notag \\&\hspace{4cm}+(-\Delta)^\frac{\theta}{2}[{\tt \Phi}(u_{\gamma,\Bar{\eps}} + \lambda w)](-\Delta)^\frac{\theta}{2}[v]+ {\tt h}(u_{\gamma,\Bar{\eps}} + \lambda w){\ell}(t)v\Big]\,dx
\end{align*}
is continuous, since $\nabla,~\Delta,~(\Delta)^{\frac{\theta}{2}}$ are continuous linear operators, $ F,~{\tt \Phi}$ are  Lipschitz continuous functions and ${\tt h}_k \in C(\mathbb{T}^d \times \R).$
\vspace{0.1cm}

\noindent In order to verify the weak monotoncity condition $(H2)$, we observe that for any $u_{\gamma,\Bar{\eps}}, v \in H^2(\mathbb{T}^d)$, and a.e. $t\in [0,T]$
\begin{align*}
   & 2\langle A(t, u_{\gamma,\Bar{\eps}})- A(t, v), u_{\gamma,\Bar{\eps}}-v \rangle  \notag \\ & = -2\int_{\mathbb{T}^d} \Bigg\{\Bar{\eps}|\nabla( u_{\gamma,\Bar{\eps}}- v)|^2 + \gamma|\Delta (u_{\gamma,\Bar{\eps}}- v)|^2 - [F(u_{\gamma,\Bar{\eps}}) - F(v)]\nabla(u_{\gamma,\Bar{\eps}} -v) \notag \\& +(-\Delta)^\frac{\theta}{2}[{\tt \Phi}(u_{\gamma,\Bar{\eps}}) - {\tt \Phi}(v)](-\Delta)^\frac{\theta}{2}[(u_{\gamma,\Bar{\eps}} - v)]+ [\big({\tt h}(u_{\gamma,\Bar{\eps}}) -{\tt h}(v)\big){\ell}(t)][u_{\gamma,\Bar{\eps}}-v]\Bigg\}\,dx\notag \\
   \le &\,2\Big[-\Bar{\eps}\|\nabla( u_{\gamma,\Bar{\eps}}- v)\|_{L^2(\mathbb{T}^d)}^2 - \gamma\|\Delta u_{\gamma,\Bar{\eps}}- v\|_{L^2(\mathbb{T}^d)}^2   + \|F^\prime\|_{\infty}\int_{\mathbb{T}^d}|u_{\gamma,\Bar{\eps}} - v|\,|\nabla(u_{\gamma,\Bar{\eps}} -v)|\,dx \notag \\&  + \|{\tt \Phi}'\|_{L^\infty}\int_{\mathbb{T}^d} |u_{\gamma,\Bar{\eps}} - v|\big|(-\Delta)^\theta[(u_{\gamma,\Bar{\eps}} - v)]\big|\,dx \notag \\ &\hspace{2cm}+ |{\ell}(t)|_{\mathcal{H}} \|{\tt h}(u_{\gamma,\Bar{\eps}}) -{\tt h}(v)\|_{\mathcal{L}_2(\mathcal{H}, L^2(\mathbb{T}^d))}\|u_{\gamma,\Bar{\eps}} -v\|_{L^2(\mathbb{T}^d)}\Big]\notag \\
   \le &\, C\Big[\|u_{\gamma,\Bar{\eps}} - v\|_{L^2(\mathbb{T}^d)}\big(\|\nabla(u_{\gamma,\Bar{\eps}} -v)\|_{L^2(\mathbb{T}^d)} + \|u_{\gamma,\Bar{\eps}} -v\|_{H^2(\mathbb{T}^d)}\big) + \|u_{\gamma,\Bar{\eps}} -v\|_{L^2(\mathbb{T}^d)}^2\Big]\notag \\
   \le & \,\big(C(\Bar{\eps}, \gamma)+ C\big)\|u_{\gamma,\Bar{\eps}} -v\|_{L^2(\mathbb{T}^d)}^2,
\end{align*}
where we have used Cauchy-Schwartz inequality, the assumptions \ref{A1}-\ref{A3}, the condition ${\ell}\in S_N$ and  the fact that $$\|(-\Delta)^\theta[u_{\gamma,\Bar{\eps}} -v]\|_{L^2(\mathbb{T}^d)} \le C\|u_{\gamma,\Bar{\eps}} -v\|_{H^2(\mathbb{T}^d)} \le C(\Bar{\eps}, \gamma)\|u_{\gamma,\Bar{\eps}} -v\|_{L^2(\mathbb{T}^d)}.$$
\vspace{0.1cm}

\noindent For the verification of coercivity $(H3)$, first we observe that for a.e. $t\in [0,T]$
\begin{align}\label{eq:bound-sigma-h}
&\int_{\mathbb{T}^d}{\tt h}(u_{\gamma,\Bar{\eps}}){\ell}(t)u_{\gamma,\Bar{\eps}}\,dx = \int_{\mathbb{T}^d}\Big(\sum_{k \ge 1}{\tt h}_k(u_{\gamma,\Bar{\eps}}){\ell}_k(t)e_k\Big)u_{\gamma,\Bar{\eps}}\,dx \\ &
\le  |{\ell}(t)|_{\mathcal{H}}\int_{\mathbb{T}^d}\Big(\sum_{k \ge 1}|{\tt h}_k(u_{\gamma,\Bar{\eps}})|^2\Big)^\frac{1}{2}|u_{\gamma,\Bar{\eps}}|\,dx \le C\int_{\mathbb{T}^d}|u_{\gamma,\Bar{\eps}}|^2\,dx \le C\|u_{\gamma,\Bar{\eps}}\|_{L^2(\mathbb{T}^d)}^2.\notag
\end{align}
Using Lipschitz continuity of $F, {\tt \Phi}$ and \eqref{eq:bound-sigma-h}, we obtain
\begin{align*}
    &2\langle A(t, u_{\gamma,\Bar{\eps}}),  u_{\gamma,\Bar{\eps}}\rangle \notag \\ =
    & -2\int_{\mathbb{T}^d}\Bigg\{ \Bar{\eps}|\nabla u_{\gamma,\Bar{\eps}}|^2 + \gamma|\Delta u_{\gamma,\Bar{\eps}}|^2 - F(u_{\gamma,\Bar{\eps}})\nabla u_{\gamma,\Bar{\eps}} +(-\Delta)^\frac{\theta}{2}[{\tt \Phi}(u_{\gamma,\Bar{\eps}})](-\Delta)^\frac{\theta}{2}[u_{\gamma,\Bar{\eps}}] \notag \\
     &+ {\tt h}(u_{\gamma,\Bar{\eps}}){\ell}(t)u_{\gamma,\Bar{\eps}}\Bigg\}\,dx \le -2\,\text{min}(\Bar{\eps}, \gamma)\|u_{\gamma,\Bar{\eps}}\|_{H^2(\mathbb{T}^d)}^2 + \big(C + 2 \,\text{min}(\Bar{\eps}, \gamma)\big)\|u_{\gamma,\Bar{\eps}}\|_{L^2(\mathbb{T}^d)}^2\,.
\end{align*}
We proceed as follows to verify the boundedness property $(H4)$.
\begin{align}
   &\Big\|\Bar{\eps}\Delta u_{\gamma,\Bar{\eps}} - \gamma\Delta^2u_{\gamma,\Bar{\eps}} -\text{div}F(u_{\gamma,\Bar{\eps}}) - (-\Delta)^\theta[{\tt \Phi}(u_{\gamma,\Bar{\eps}})] + {\tt h}(u_{\gamma,\Bar{\eps}}){\ell}(t)\Big\|_{H^{-2}(\mathbb{T}^d)} \notag \\
   & = \underset{\|v\|_{H^2(\mathbb{T}^d)} \neq 0}{\sup}\frac{\Big|\Big\langle\Bar{\eps}\Delta u_{\gamma,\Bar{\eps}} - \gamma\Delta^2u_{\gamma,\Bar{\eps}} -\text{div}F(u_{\gamma,\Bar{\eps}}) - (-\Delta)^\theta[{\tt \Phi}(u_{\gamma,\Bar{\eps}})] + {\tt h}(u_{\gamma,\Bar{\eps}}){\ell}(t), v \Big\rangle\Big|}{\|v\|_{H^2(\mathbb{T}^d)}}\notag \\
   & \le \underset{\|v\|_{H^2(\mathbb{T}^d)} \neq 0}{\sup}\Bigg(\Bar{\eps}\frac{\|\nabla u_{\gamma,\Bar{\eps}}\| \|\nabla v\|}{\|v\|_{H^2(\mathbb{T}^d)}} + \gamma\frac{\|\Delta u_{\gamma,\Bar{\eps}}\| \|\Delta v\|}{\|v\|_{H^2(\mathbb{T}^d)}} + \|F'\|_{\infty}\frac{\|u_{\gamma,\Bar{\eps}}\| \|\nabla v\|}{\|v\|_{H^2(\mathbb{T}^d)}}\notag \\ & \hspace{4cm} + \|{\tt \Phi}'\|_{\infty}\frac{\|u_{\gamma,\Bar{\eps}}\|_{H^1(\mathbb{T}^d)}\|v\|_{H^1(\mathbb{T}^d)}}{\|v\|_{H^2(\mathbb{T}^d)}} + C|h(t)|_{\mathcal{H}}\frac{\|u_{\gamma,\Bar{\eps}}\|\|v\|}{\|v\|_{H^2(\mathbb{T}^d)}}\Bigg),\notag
\end{align}
thus, we have
$$\Big\|\Bar{\eps}\Delta u_{\gamma,\Bar{\eps}} - \gamma\Delta^2u_{\gamma,\Bar{\eps}} -\text{div}F(u_{\gamma,\Bar{\eps}}) - (-\Delta)^\theta[{\tt \Phi}(u_{\gamma,\Bar{\eps}})] + {\tt h}(u_{\gamma,\Bar{\eps}}){\ell}(t)\Big\|_{H^{-2}(\mathbb{T}^d)} \le C \|v\|_{H^2(\mathbb{T}^d)}.$$
Hence, in view of \cite[Theorem 4.2.4]{Rockner}, there exists a unique solution to \eqref{eq:singular perturbation-skeleton}.
\end{proof}
\subsubsection{\bf A-priori estimate:} 
Applying chain rule to the function $g(v) = \|v\|_{L^2(\mathbb{T}^d)}^2$, from \eqref{eq:singular perturbation-skeleton}, we have
\begin{align}\label{eq:eps-gamma}
&\|u_{\gamma,\Bar{\eps}}(t)\|_{L^2(\mathbb{T}^d)}^2 = \, \|u_0^{\Bar{\eps}}\|_{L^2(\mathbb{T}^d)}^2 + 2\int_0^t\langle \nabla u_{\gamma,\Bar{\eps}}, F( u_{\gamma,\Bar{\eps}}) \rangle\,ds \\ & - 2\int_0^t\langle u_{\gamma,\Bar{\eps}}, (-\Delta)^{\theta}{\tt \Phi}(u_{\gamma,\Bar{\eps}}) \rangle\,ds -\,2\,\Bar{\eps}\int_0^t\int_{\mathbb{T}^d}|\nabla u_{\gamma,\Bar{\eps}}|^2\,dx\,ds \notag \\
& -\,2\,\gamma\int_0^t\int_{\mathbb{T}^d}|\Delta u_{\gamma,\Bar{\eps}}|^2\,dx\,ds  + 2\int_0^t\|u_{\gamma,\Bar{\eps}}\|_{L^2(\mathbb{T}^d)}\|{\tt h}(u_{\gamma,\Bar{\eps}})\|_{\mathcal{L}_2(\mathcal{H};L^2(\mathbb{T}^d))}|{\ell}(s)|_{\mathcal{H}}\,ds\,.\notag
\end{align}
Observe that $\int_{\mathbb{T}^d}\nabla u_{\gamma,\Bar{\eps}}\cdot F( u_{\gamma,\Bar{\eps}})\,dx = 0$. Since ${\tt \Phi}$ is non-decreasing, we have
\begin{align} \label{esti:nondecreasing-phi}
    &\frac{1}{\|{\tt \Phi}'\|_{L^\infty}}\big|{\tt \Phi}(\xi_1)-{\tt \Phi}(\xi_2)\big|^2 \le \big({\tt \Phi}(\xi_1)-{\tt \Phi}(\xi_2)\big)\big(\xi_1-\xi_2\big) \quad \forall~ \xi_1, \xi_2 \in \R. 
\end{align}
Using \eqref{esti:nondecreasing-phi} along with the assumption \ref{A3}, \eqref{eq:eps-gamma} yields
\begin{align}
&\|u_{\gamma,\Bar{\eps}}(t)\|_{L^2(\mathbb{T}^d)}^2  +\,2\,\Bar{\eps}\int_0^t\int_{\mathbb{T}^d}|\nabla u_{\gamma,\Bar{\eps}}|^2\,dx\,ds + \,2\,\gamma\int_0^t\int_{\mathbb{T}^d}|\Delta u_{\gamma,\Bar{\eps}}|^2\,dx\,ds \\ & + \frac{1}{\|{\tt \Phi}'\|_{\infty}}\int_0^t \big[{\tt \Phi}(u_{\gamma, \Bar{\eps}})\big]_{H^\theta}^2\,ds \notag \le \|u_0^\eps\|_{L^2(\mathbb{T}^d)} + C\int_0^t\|u_{\gamma,\Bar{\eps}}\|_{L^2(\mathbb{T}^d)}^2|{\ell}(s)|_{\mathcal{H}}\,ds, \notag 
\end{align}
where $[\cdot]_{H^\theta}$ denotes the Gagliardo (semi) norm of the fractional Sobolev space $W^{\theta,2}$.
An application of Gronwall's Lemma gives, for all $t\in [0,T]$,
\begin{align}\label{eq:energy-estimation-eps-gamma}
&\|u_{\gamma,\Bar{\eps}}(t)\|_{L^2(\mathbb{T}^d)}^2  +\,2\,\Bar{\eps}\int_0^t\int_{\mathbb{T}^d}|\nabla u_{\gamma,\Bar{\eps}}|^2\,dx\,ds + \,2\,\gamma\int_0^t\int_{\mathbb{T}^d}|\Delta u_{\gamma,\Bar{\eps}}|^2\,dx\,ds\\ &+ \frac{1}{\|{\tt \Phi}'\|_{\infty}}\int_0^t \big[{\tt \Phi}(u_{\gamma, \Bar{\eps}})\big]_{H^\theta}^2\,ds \le C\|u_0^\eps\|_{L^2(\mathbb{T}^d)}\exp\Big[{C\int_0^t|{\ell}(s)|_{\mathcal{H}}\,ds}\Big] \le C \,. \notag
\end{align}
 A-priori estimates in \eqref{eq:energy-estimation-eps-gamma} is not sufficient to pass to the limit in the weak formulation of \eqref{eq:singular perturbation-skeleton}. It is required to use suitable compactness argument for $\{u_{\gamma,\Bar{\eps}}\}_{\gamma>0}$ for passing to the limit as $\gamma\goto 0$ and hence to show the existence of viscous solution for \eqref{eq:regularized-skeleton}. In this regard, let $\mathbb{X}$ be a separable Banach space endowed with norm $\|\cdot\|_{\mathbb{X}}$ and for given $\beta \in (0,1)$ and $p > 1$, $W^{\beta,p}([0,T]; \mathbb{X})$ be the Sobolev space of all functions $u \in L^p([0,T];\mathbb{X})$ such that
\begin{align*}
    \int_0^T\int_0^T\frac{\|u(t) - u(s)\|_{\mathbb{X}}^p}{|t-s|^{1+\beta p}}\,dt\,ds < \infty,
\end{align*}
equipped with the norm 
$$\|u\|_{W^{\beta,p}([0,T];\mathbb{X})}^p = \int_0^T\|u(t)\|_{\mathbb{X}}^p\,dt + \int_0^T\int_0^T\frac{\|u(t) - u(s)\|_{\mathbb{X}}^p}{|t-s|^{1+\beta p}}\,dt\,ds.$$
Recall the following theorem as in \cite{Flandoli:1995}.
\begin{thm}\label{thm:compactness}
    Let $\mathcal{Y}_0 \subset \mathcal{Y} \subset \mathcal{Y}_1$ be Banach spaces. Suppose $\mathcal{Y}_0$ and $\mathcal{Y}_1$ are reflexive  and $\mathcal{Y}_0$ is compactly embedded in $\mathcal{Y}$. Let $p \in (1,\infty)$, $\beta \in (0,1)$ and $\Bar{\mathcal{Y}} := L^p([0,T]; \mathcal{Y}_0) \cap W^{\beta, p}([0,T];\mathcal{Y}_1)$ equipped with the natural norm. Then the embedding of $\Bar{\mathcal{Y}}$ in $L^p([0,T]; \mathcal{Y})$ is compact.
\end{thm}
\noindent Setting $\mathcal{Y}_0 := H^2(\mathbb{T}^d), \mathcal{Y} := L^2(\mathbb{T}^d)$ and $\mathcal{Y}_1 := H^{-2}(\mathbb{T}^d)$, where $H^{-2}(\mathbb{T}^d)$ is the dual space of $H^2(\mathbb{T}^d)$ in Theorem \ref{thm:compactness}, we have the following result regarding the compactness of $\{u_{\gamma, \Bar{\eps}}\}$.
\begin{lem}\label{lem:compactness-gamma-eps}
 For any $\Bar{\eps} >0$, $\{u_{\gamma,\Bar{\eps}}: \gamma > 0\}$ is compact in $L^2([0,T]; L^2(\mathbb{T}^d)).$   
\end{lem}
\begin{proof}
From \eqref{eq:singular perturbation-skeleton}, we have
\begin{align}
     &u_{\gamma,\Bar{\eps}}(t) = u_0^{\Bar{\eps}} - \int_0^t\text{div}F(u_{\gamma,\Bar{\eps}}(r,x))\,dr -\int_0^t (-\Delta)^\theta[{\tt \Phi}(u_{\gamma,\Bar{\eps}}(r,\cdot))](x)\,dr \notag \\ &\hspace{2cm}+ \, \Bar{\eps}\int_0^t\Delta u_{\gamma,\Bar{\eps}}\,dr -\gamma\int_0^t\Delta^2 u_{\gamma,\Bar{\eps}}\,dr + \int_0^t{\tt h}(u_{\gamma,\Bar{\eps}}(r,x)){\ell}(r)\,dr\,
     =: \sum_{i=1}^{6}\mathcal{A}_i^{\gamma, \Bar{\eps}}\notag.
\end{align}
Clearly, $\|\mathcal{A}_1^{\gamma, \Bar{\eps}}\|_{L^2(\mathbb{T}^d)} \le C.$
Using Lipschitz continuity of $F$ and integration by parts formula, we infer that
\begin{align}\label{inq:bound-div-f}
    \|\text{div}F(u_{\gamma, \Bar{\eps}}(s))\|_{H^{-2}(\mathbb{T}^d)}  = & \underset{\|v\|_{H^2(\mathbb{T}^d)} \le 1}{sup} \big|\langle v, \text{div}F(u_{\gamma, \Bar{\eps}}(s)) \rangle\big|\notag  
     =  \underset{\|v\|_{H^2(\mathbb{T}^d)} \le 1}{sup} \big|\langle \nabla v, F(u_{\gamma, \Bar{\eps}}(s)) \rangle\big| \notag \\&\hspace{1cm}\le C \underset{\|v\|_{H^2(\mathbb{T}^d)} \le 1}{sup}\int_{\mathbb{T}^d}|\nabla v||u_{\gamma, \Bar{\eps}}(s)|\,dx \le C\|u_{\gamma, \Bar{\eps}}(s)\|_{L^2(\mathbb{T}^d)}, 
\end{align}
and therefore thanks to Jensen inequality along with \eqref{inq:bound-div-f} and \eqref{eq:energy-estimation-eps-gamma} we get,
\begin{align*}
   & \|\mathcal{A}_2^{\gamma, \Bar{\eps}}(t) - \mathcal{A}_2^{\gamma, \Bar{\eps}}(s)\|_{H^{-2}}^2\le C(t-s)\int_s^t \|\text{div}F(u_{\gamma, \Bar{\eps}}(r))\|_{H^{-2}}^2\,dr \notag \\ & \hspace{4cm}\le C(t-s)\int_s^t \|u_{\gamma, \Bar{\eps}}(r)\|_{L^{2}(\mathbb{T}^d)}^2\,dr \le C(t-s).
\end{align*}
Hence, we have, for $\beta \in (0,\frac{1}{2})$
\begin{align*}
 &\,\|\mathcal{A}_2^{\gamma, \Bar{\eps}}\|_{W^{\beta,2}([0,T];\mathcal{Y}_1)}^2  \le \int_0^T\|\mathcal{A}_2^{\gamma, \Bar{\eps}}\|_{H^{-2}}^2\,dt + \int_0^T\int_0^T\frac{ \|\mathcal{A}_2^{\gamma, \Bar{\eps}}(t) - \mathcal{A}_2^{\gamma, \Bar{\eps}}(s)\|_{H^{-2}}^2}{|t-s|^{1+2\beta}}\,ds\,dt \le C(\beta).   
\end{align*}
In view of \ref{A2} and the fact that 
$\|(-\Delta)^\theta[v(\cdot)]\|_{L^2(\mathbb{T}^d)} \le C \|v\|_{H^2(\mathbb{T}^d)}$, we have
\begin{align}
    \|(-\Delta)^\theta[{\tt \Phi}(u_{\gamma,\Bar{\eps}}(t,\cdot)]\|_{H^{-2}(\mathbb{T}^d)} &= \underset{\|v\|_{H^2(\mathbb{T}^d)} \le 1}{sup} \big|\big\langle v, (-\Delta)^\theta[{\tt \Phi}(u_{\gamma, \Bar{\eps}}(t,\cdot))](\cdot)\big\rangle\big| \notag \\ & = \underset{\|v\|_{H^2(\mathbb{T}^d)} \le 1}{sup} \big|\big\langle (-\Delta)^\theta[v], {\tt \Phi}(u_{\gamma, \Bar{\eps}})\big\rangle\big| \le C\| u_{\gamma,\Bar{\eps}}(t,\cdot)\|_{L^2(\mathbb{T}^d)},\notag
\end{align}
which then implies 
\begin{align*}
    \|\mathcal{A}_3^{\gamma, \Bar{\eps}}(t) - \mathcal{A}_3^{\gamma, \Bar{\eps}}(s)\|_{H^{-2}(\mathbb{T}^d)}^2 
    &\le C(t-s)\int_s^t \|u_{\gamma,\Bar{\eps}}(r,\cdot))\|_{L^2(\mathbb{T}^d)}^2\,dr  \le C(t-s).
\end{align*}
Thus, for $\beta \in (0,\frac{1}{2})$
\begin{align*}
    &\,\|\mathcal{A}_3^{\gamma, \Bar{\eps}}\|_{W^{\beta,2}([0,T]; \mathcal{Y}_1)}^2  \le \int_0^T\|\mathcal{A}_3^{\gamma, \Bar{\eps}}\|_{H^{-2}}^2\,dt + \int_0^T\int_0^T\frac{ \|\mathcal{A}_3^{\gamma, \Bar{\eps}} - \mathcal{A}_3^{\gamma, \Bar{\eps}}\|_{H^{-2}}^2}{|t-s|^{1+2\beta}}\,ds\,dt \le C(\beta).
\end{align*} 
We invoke the integration by parts formula to have
\begin{equation} \label{inq:A34-eps-gammma}
\begin{aligned}
\|\Delta u_{\gamma, \Bar{\eps}}\|_{H^{-2}(\mathbb{T}^d)} & = \underset{\|v\|_{H^2(\mathbb{T}^d)} \le 1}{sup}\big|\langle \nabla v, \nabla u_{\gamma, \Bar{\eps}} \rangle\big| \le C \|\nabla u_{\gamma, \Bar{\eps}}\|_{L^2(\mathbb{T}^d)}, \\
 \|\Delta^2 u_{\gamma, \Bar{\eps}}\|_{H^{-2}(\mathbb{T}^d)} & = \underset{\|v\|_{H^2(\mathbb{T}^d)} \le 1}{sup}\big|\langle \Delta v, \Delta u_{\gamma, \Bar{\eps}} \rangle\big| \le C \|\Delta u_{\gamma, \Bar{\eps}}\|_{L^2(\mathbb{T}^d)}.    
 \end{aligned}
\end{equation}
Using \eqref{inq:A34-eps-gammma} and \eqref{eq:energy-estimation-eps-gamma} together with Jensen's inequality, we get
\begin{align}
    \|\mathcal{A}_4^{\gamma, \Bar{\eps}}(t) - \mathcal{A}_4^{\gamma, \Bar{\eps}}(s)\|_{H^{-2}(\mathbb{T}^d)}^2
    & \le  \Bar{\eps}C(t-s)\int_s^t\|\nabla u_{\gamma, \Bar{\eps}}\|_{L^2(\mathbb{T}^d)}^2\,dr \le C(t-s), \notag \\
     \|\mathcal{A}_5^{\gamma, \Bar{\eps}}(t) - \mathcal{A}_5^{\gamma, \Bar{\eps}}(s)\|_{H^{-2}(\mathbb{T}^d)}^2
     &  \le\gamma C(t-s)\int_s^t\|\Delta u_{\gamma, \Bar{\eps}}\|_{L^2(\mathbb{T}^d)}^2\,dr \le C(t-s), \notag
\end{align}
and therefore, for any $\beta \in (0, \frac{1}{2})$
\begin{align*}
    &\,\|\mathcal{A}_4^{\gamma, \Bar{\eps}}\|_{W^{\beta,2}([0,T];\mathcal{Y}_1)}^2 \le C(\beta)\quad  \text{and} \quad \,\|\mathcal{A}_5^{\gamma, \Bar{\eps}}\|_{W^{\beta,2}([0,T]; H^{-2}(\mathbb{T}^d))}^2 \le C(\beta).
\end{align*}
Using Cauchy-Schwartz inequality along with the assumption \ref{A3}, we have
\begin{align}\label{inq:A6-eps-gamma}
  &\|{\tt h}(u_{\gamma, \Bar{\eps}}){\ell}(r)\|_{H^{-2}(\mathbb{T}^d)} = \underset{\|v\|_{H^2(\mathbb{T}^d)} \le 1}{sup}\big|\langle v , {\tt h}(u_{\gamma, \Bar{\eps}}){\ell}(r) \rangle\big| \\ &   \le \underset{\|v\|_{H^2(\mathbb{T}^d)} \le 1}{sup}\Big|\Big\langle v , \Big(\sum_{k \ge 1}{\tt h}_k(u_{\gamma, \Bar{\eps}})\Big)^\frac{1}{2}\Big(\sum_{k \ge 1}|{\ell}_k(r)|^2\Big)^{\frac{1}{2}}\Big\rangle\Big| \le C|{\ell}(r)|_{\mathcal{H}}\|u_{\gamma, \Bar{\eps}}\|_{L^2(\mathbb{T}^d)}.\notag
\end{align}
Thanks to \eqref{inq:A6-eps-gamma}, one can easily see that
\begin{align}
 &\|\mathcal{A}_6^{\gamma, \Bar{\eps}}(t) - \mathcal{A}_6 ^{\gamma, \Bar{\eps}}(s)\|_{H^{-2}(\mathbb{T}^d)}^2  \le C(t-s)\int_s^t\|{\tt h}(u_{\gamma, \Bar{\eps}}){\ell}(r)\|_{H^{-2}(\mathbb{T}^d)}^2 \, dr \notag \\ &\le C(t-s)\int_s^t|{\ell}(r)|_{\mathcal{H}}^2\Big(\underset{s \le r \le t}{sup}\|u_{\gamma, \Bar{\eps}}\|_{L^2(\mathbb{T}^d)}^2\Big) \, dr \le C(t-s)\int_s^t|{\ell}(r)|_{\mathcal{H}}^2\,dr \le C(t-s),\notag
\end{align}
which then yields, for any $\beta \in (0, \frac{1}{2})$,
\begin{align*}
\|\mathcal{A}_6^{\gamma, \Bar{\eps}}\|_{W^{\beta,2}([0,T]; H^{-2}(\mathbb{T}^d))}^2 \le C(\beta).
\end{align*}
Thus, in view of \eqref{eq:energy-estimation-eps-gamma} and Theorem \ref{thm:compactness}, we obtain the compactness of $\{u_{\gamma, \Bar{\eps}}: \gamma>0\}$ in $L^2([0,T]; L^2(\mathbb{T}^d)).$ 
\end{proof}
\subsubsection{\bf Existence and uniqueness of viscous solution} \label{eq:Existence and uniqueness of viscous solution} With the help of {\it a-priori} estimates \eqref{eq:energy-estimation-eps-gamma} and Lemma \ref{lem:compactness-gamma-eps}, we now show existence of a weak solution for \eqref{eq:regularized-skeleton}.
\begin{prop}
Under the assumptions \ref{A1}-\ref{A3}, there exists a weak solution to \eqref{eq:regularized-skeleton}. 
\end{prop}
\begin{proof}
Thanks to \eqref{eq:energy-estimation-eps-gamma} and Lemma \ref{lem:compactness-gamma-eps}, we have for any sequence $\gamma_n \rightarrow 0$ as $n \rightarrow \infty$, there exists a subsequence $\{\gamma_{n_k}\}_{k \ge 1}$ (still denote as $\gamma_n$) and an element $u_{\Bar{\eps}} \in L^\infty([0,T]; L^2(\mathbb{T}^d)) \cap L^2([0,T];L^2(\mathbb{T}^d))$ such that 
\begin{align}\label{eq:convergence-gamma}
 u_{\gamma_n, \Bar{\eps}} \rightarrow u_{\Bar{\eps}} \quad \text{in}~~~ L^2([0,T]; L^2(\mathbb{T}^d)), \quad 
    u_{\gamma_n, \Bar{\eps}} \rightharpoonup u_{\Bar{\eps}} \quad \text{in}~~ L^\infty([0,T]; L^2(\mathbb{T}^d)).
\end{align}
Since $u_{\gamma, \Bar{\eps}}(\cdot)$ is a strong solution of \eqref{eq:singular perturbation-skeleton},  for any test function $\psi \in C_c^\infty(\mathbb{T}^d)$, it holds that 
\begin{align}\label{eq:convergence-gamma-test-psi}
    &\langle u_{\gamma_n, \Bar{\eps}}(t), \psi \rangle - \langle u_0^{\Bar{\eps}}, \psi \rangle = \int_0^t\langle F(u_{\gamma_n,\Bar{\eps}}(r,\cdot)), \nabla \psi \rangle\,dr \\ & -\int_0^t \langle{\tt \Phi}(u_{\gamma_n,\Bar{\eps}}(r,\cdot)), (-\Delta)^\theta[\psi](\cdot)\rangle\,dr-\gamma\int_0^t\langle u_{\gamma_n,\Bar{\eps}}, \Delta^2\psi\rangle\,dr 
   \notag \\ & +  \Bar{\eps}\int_0^t\langle u_{\gamma_n,\Bar{\eps}}, \Delta\psi\rangle\,dr + \int_0^t\langle{\tt h}(u_{\gamma_n,\Bar{\eps}}(r,x)){\ell}(r), \psi\rangle\,dr.\notag
\end{align}
In view of \eqref{eq:convergence-gamma}, we have
$ \big|\langle u_{\gamma_n,\Bar{\eps}}(t)- u_{\Bar{\eps}}(t), \psi \rangle \big| \rightarrow 0$ as $n \rightarrow \infty$. 
An application of Lipschitz continuity  of $F$ and \eqref{eq:convergence-gamma} yields
\begin{align*}
    \int_0^t\langle F(u_{\gamma_n,\Bar{\eps}}(r,\cdot))- F(u_{\Bar{\eps}}(r,\cdot)), \nabla \psi \rangle\,dr \le C\|\nabla \psi\|_{L^\infty(\mathbb{T}^d)}\int_0^t\|u_{\gamma_n,\Bar{\eps}} - u_{\Bar{\eps}}\|_{L^2(\mathbb{T}^d)}^2\,dr \rightarrow 0.
\end{align*}
 Note that,  for $\psi(x) \in C_c^\infty(\mathbb{T}^d)$
 \begin{align}\label{fractionlbound-gamma}
     \|(-\Delta)^\theta[\psi](\cdot)\|_{L^\infty(\mathbb{T}^d)} <  \infty. 
 \end{align}
With \eqref{fractionlbound-gamma}, \eqref{eq:convergence-gamma},  and the assumption \ref{A2} in hand,  we get
\begin{align*}
    &\int_0^t \langle{\tt \Phi}(u_{\gamma_n,\Bar{\eps}}(r,\cdot)) - {\tt \Phi}(u_{\Bar{\eps}}(r,\cdot)), (-\Delta)^\theta[\psi(\cdot)](\cdot)\rangle\,dr  \notag \\& \hspace{1cm} \le C \,\|(-\Delta)^\theta[\psi](\cdot)\|_{L^\infty(\mathbb{T}^d)}\int_0^T\|u_{\gamma_n,\Bar{\eps}} - u_{\Bar{\eps}}\|_{L^2(\mathbb{T}^d)}^2\,dt \rightarrow 0.
\end{align*}
Similarly, using \eqref{eq:convergence-gamma}, we get
\begin{align*}
     &\Bar{\eps}\int_0^t\langle u_{\gamma_n,\Bar{\eps}}- u_{\Bar{\eps}}, \Delta\psi\rangle\,dr \le C(\Bar{\eps})\|\Delta\psi\|_{L^\infty(\mathbb{T}^d)}\int_0^t\|u_{\gamma_n,\Bar{\eps}} - u_{\Bar{\eps}}\|_{L^2(\mathbb{T}^d)}^2\,dr \rightarrow 0,\notag \\
     & \gamma_n\int_0^t\langle u_{\gamma_n,\Bar{\eps}}- u_{\Bar{\eps}}, \Delta^2\psi\rangle\,dr \le C(\gamma_n)\|\Delta^2\psi\|_{L^\infty(\mathbb{T}^d)}\int_0^t\|u_{\gamma_n,\Bar{\eps}} - u_{\Bar{\eps}}\|_{L^2(\mathbb{T}^d)}^2\,dr \rightarrow 0\,.
\end{align*}
We use the assumption \ref{A3} and the fact that ${\ell}(\cdot)\in S_N$ to conclude
\begin{align*}
&\int_0^t\langle\big({\tt h}(u_{\gamma_n,\Bar{\eps}})- {\tt h}(u_{\Bar{\eps}})\big){\ell}(r), \psi\rangle\,dr  \notag \\ &\hspace{2cm}\le C(\|\psi\|_{L^\infty}^2)\Big(\int_0^t\|{\ell}(r)\|^2_{\mathcal{H}}\,dr\Big)^{\frac{1}{2}}\Big(\int_0^t\|u_{\gamma_n, \Bar{\eps}} -u_{\Bar{\eps}}\|_{L^2(\mathbb{T}^d)}^2\,dr\Big)^{\frac{1}{2}} \rightarrow 0.
\end{align*}
In view of the above convergence results, we send $n \rightarrow 0$ in \eqref{eq:convergence-gamma-test-psi} and have
\begin{align*}
    &\langle u_{\Bar{\eps}}(t), \psi \rangle - \langle u_0^{\Bar{\eps}}, \psi \rangle = \int_0^t\langle F(u_{\Bar{\eps}}(r)), \nabla \psi \rangle\,dr -\int_0^t \langle{\tt \Phi}(u_{\Bar{\eps}}(r)), (-\Delta)^\theta[\psi](\cdot)\rangle\,dr \notag \\ &\hspace{4cm}+ \, \Bar{\eps}\int_0^t\langle u_{\Bar{\eps}}, \Delta\psi\rangle\,dr + \int_0^t\langle{\tt h}(u_{\Bar{\eps}}(r,\cdot)){\ell}(r), \psi\rangle\,dr.
\end{align*}
This shows that $u_{\Bar{\eps}} \in L^2([0,T]; L^2(\mathbb{T}^d))$  is a weak solution of \eqref{eq:regularized-skeleton}. 
\end{proof}
To prove uniqueness of viscous solution, we follow a similar line of arguments as done in Theorem \ref{thm:contraction-principle}.  
\begin{prop}
Let the assumptions \ref{A1}-\ref{A3} hold true. Suppose that $u_{\Bar{\eps}}(t)$ and $\Bar{u}_{\Bar{\eps}}(t)$ are weak solutions to \eqref{eq:regularized-skeleton} with initial data $u_{0}$ and $\Bar{u}_{0}$ respectively. Then, for all $t \in [0,T]$ 
\begin{align}
 \|u_{\Bar{\eps}}(t)- \Bar{u}_{\Bar{\eps}}(t)\|_{L^1(\mathbb{T}^d)} \le \|u_{0} -\Bar{u}_{0}\|_{L^1(\mathbb{T}^d)}.   \label{eq:uniquness-vis-eq}
\end{align}
\end{prop}
\begin{proof} Based on the proof of Theorem \ref{thm:contraction-principle}, we can conclude that for all $t \in [0,T]$
    \begin{align}
   & \int_{(\mathbb{T}^d)^2}\int_{\R^2} \varrho_\delta(x-y)\rho_\lambda(\xi-\zeta)\Big(f_1^{\pm}(x,t,\xi)\bar{f}_2^{\pm}(y,t,\zeta) + \bar{f}_1^{\pm}(x,t,\xi)f_2^{\pm}(y,t,\zeta)\Big)\,d\xi\,d\zeta\,dx\,dy \notag\\
   &\,\le\, \int_{(\mathbb{T}^d)^2}\int_{\R^2} \varrho_\delta(x-y)\rho_\lambda(\xi-\zeta)\Big(f_{1,0}(x,\xi)\bar{f}_{2,0}(y,\zeta) + \bar{f}_{1,0}(x,\xi)f_{2,0}(y,\zeta)\Big)\,d\xi\,d\zeta\,dx\,dy \notag\\
   &+ \mathcal{S}_{1,\Bar{\eps}} + \Bar{\mathcal{S}}_{1,\Bar{\eps}} + \mathcal{S}_{2, \Bar{\eps}} + \Bar{\mathcal{S}}_{2,\Bar{\eps}} + 2\mathcal{S}_{3, \Bar{\eps}} + \mathcal{S} + \Bar{\mathcal{S}},  \notag
   \end{align}
where $f_1 = {\bf 1}_{u_{\Bar{\eps}}(t) > \xi}$, $f_2 = {\bf 1}_{\Bar{u}_{\Bar{\eps}}(t) > \zeta}$, $f_{1,0} = {\bf 1}_{u_0 > \xi}$ and $f_{2,0} = {\bf 1}_{\Bar{u}_0 > \zeta}$. Moreover, $\mathcal{S}$  and $\Bar{\mathcal{S}}$ are corresponding to the terms $\Bar{\eps}\Delta_x u_{\Bar{\eps}}$ and $\Bar{\eps}\Delta_y\Bar{u}_{\Bar{\eps}}$ respectively (see \cite[Section 6]{Hofmanova-2016}). For example $\mathcal{S}$ takes the form
\begin{align}
    \mathcal{S} = &\,2\Bar{\eps}\int_0^t \int_{(\mathbb{T}^d)^2}\int_{\R^2}f_1\Bar{f}_2\Delta_x\varrho_\delta(x-y)\rho_\lambda(\xi-\zeta)\,d\xi\,d\zeta\,dx\,dy\,ds \notag \\
    & -\int_0^t \int_{(\mathbb{T}^d)^2}\int_{\R^2}\varrho_\delta(x-y)\rho_\lambda(\xi-\zeta)\,d\mathcal{V}_{x,s}^{\Bar{\eps}}(\xi)\,dx\,d\Bar{\eta}^{\Bar{\eps}}(y,s,\zeta)\notag \\
    & -\int_0^t \int_{(\mathbb{T}^d)^2}\int_{\R^2}\varrho_\delta(x-y)\rho_\lambda(\xi-\zeta)\,d\mathcal{V}_{y,s}^{\Bar{\eps}}(\zeta)\,dx\,d\eta^{\Bar{\eps}}(x,s,\xi)\notag \\
     = &-2\Bar{\eps}\int_0^t \int_{(\mathbb{T}^d)^2\times \R^2}\varrho_\delta\rho_\lambda(u_{\Bar{\eps}}(x)- \Bar{u}_{\Bar{\eps}}(y))|\nabla_x u_{\Bar{\eps}}(x) - \nabla_y \Bar{u}_{\Bar{\eps}}(y)|^2\,d\xi\,d\zeta\,dx\,dy\,ds \le 0,\notag 
\end{align}
where 
$d\eta^{\Bar{\eps}}(x,s,\xi):= |\eps\nabla u_{\Bar{\eps}}|^2d\delta_{u_{\Bar{\eps}} = \xi}\,dx\,ds$ and $d\Bar{\eta}^{\Bar{\eps}}(y,s,\zeta):= |\Bar{\eps}\nabla \Bar{u}_{\Bar{\eps}}|^2 d\delta_{\Bar{u}_{\Bar{\eps}} = \zeta}\,dy\,ds$.
Similarly, one has $\Bar{\mathcal{S}} \le 0$. A verbatim copy
of Theorem \ref{thm:contraction-principle} then gives \eqref{eq:uniquness-vis-eq}.
\end{proof}
\subsection{\bf Existence of Skeleton Equation}\label{sec:existance-skeleton}
We are ready to show the existence of kinetic solution for \eqref{eq:skeleton}.
However, the method is fairly technical, so we proceed with the following many steps. Before going into the technical details, let us define $u_{{\ell}, \Bar{\eps}} := u^{\Bar{\eps}}$, solution to \eqref{eq:regularized-skeleton}. Moreover, we have the following  kinetic formulation of \eqref{eq:regularized-skeleton}.
\begin{align}\label{eq:kinetic-formulation-regularized-viscous}
    &\langle f^{\Bar{\eps}}(t), \psi \rangle =\,\langle f_0^{\Bar{\eps}}, \psi\rangle + \int_0^t  \langle f^{\Bar{\eps}}(s), F'(\xi)\cdot\nabla_x\psi \rangle\,ds -\int_0^t\langle f^{\Bar{\eps}}(s), {\tt \Phi}'(\xi)(\Delta)^\theta[\psi] \rangle\,ds \\ &  + \Bar{\eps}\int_0^t \langle f^{\Bar{\eps}}(t), \Delta\psi \rangle\,ds\notag + \sum_{k=1}^{\infty}\int_0^t\int_{\mathbb{T}^d}\int_{\R}{\tt h}_k(x,\xi){\ell}_k(s)\partial_{\xi}\psi(x,\xi)\,d\mathcal{V}_{x,s}^{\Bar{\eps}}\,dx\,ds \\ &- m^{\Bar{\eps}}(\partial_\xi\psi)([0,t]),\notag
\end{align}
where $f^{\Bar{\eps}} (x,t,\xi):={\bf 1}_{u^{\Bar{\eps}}(x,t) > \xi}$, $m^{\Bar{\eps}} \ge \Bar{m}^{\Bar{\eps}} + \eta^{\Bar{\eps}}$ with $\Bar{m}^{\Bar{\eps}}(x,t,\xi): = \Bar{\eps}|\nabla u^{\Bar{\eps}}|^2\delta_{u^{\Bar{\eps}}(x,t) = \xi}$ and
\begin{align*}
\eta^{\Bar{\eps}}(x,t,\xi):&= \int_{\R}\big|{\tt \Phi}(u^{\Bar{\eps}}(x+z,t)) - {\tt \Phi}(\xi)\big|{\bf 1}_{{\rm Conv}\{u^{\Bar{\eps}}(x+z, t), u^{\Bar{\eps}}(x,t)\}}\,(\xi)\gamma(z)\,dz\,.
\end{align*}
 \noindent{\bf Step I:~ Strong Convergence.} In this step, our goal is to validate the strong convergence for the sequence $\{u^{\Bar{\eps}}\}_{\Bar{\eps} > 0}.$ Following the proof of Theorem \ref{thm:contraction-principle}, one can easily observe that for all $t \in [0,T]$
 \begin{align*}
&\int_{(\mathbb{T}^d)^2}\int_{\R}\varrho_\delta(x-y)f^{\Bar{\eps}}(x,t,\xi)\bar{f}^{\Bar{\eps}}(y,t,\xi)\,d\xi\,dx\,dy
    \notag \\ & \le \int_{(\mathbb{T}^d)^2}\int_{\R}\varrho_\delta(x-y)f_0(x,t,\xi)\bar{f}_0(y,t,\xi)\,d\xi\,dx\,dy + 2\lambda + \mathcal{S}_{1,\Bar{\eps}} + \mathcal{S}_{2,\Bar{\eps}} + \mathcal{S}_{3,\Bar{\eps}} + \mathcal{S},
 \end{align*}
 where $\mathcal{S}_{i,\Bar{\eps}}$ for $1 \le i \le 3$, are similar to $\mathcal{S}_{i,\eps}$ of Theorem \ref{thm:contraction-principle} and $\mathcal{S}$ is due to the term $\Bar{\eps}\Delta u^{\Bar{\eps}}$, which is non-positive, (see, \cite[Step 1, Theorem 3.5]{AC_2}).
 Thus, for all $t \in [0,T]$, we get
 \begin{equation}\label{eq:G1-bound}
 \begin{aligned}
     &\int_{(\mathbb{T}^d)^2}\varrho_\delta(x-y)\big|u^{\Bar{\eps}}(x,t)- u^{\Bar{\eps}}(y,t)\big|\,dx\,dy \le \int_{(\mathbb{T}^d)^2}\varrho_\delta\big|u_0^{\Bar{\eps}}(x)- u_0^{\Bar{\eps}}(y)\big|\,dx\,dy \\
     & \hspace{1cm}+C\big(\delta^{-1}\lambda + 
r^{2-2\theta}\delta^{-2} + r^{-2\theta}\lambda + \delta + \lambda\big).   
 \end{aligned}
 \end{equation}
Following the calculations as done in the proof of Theorem \ref{thm:contraction-principle}, for any two viscous solution $u^{\Bar{\eps}}$ and $ u^{\hat{\eps}}$, we have
\begin{align}\label{eq:Upsilon-1}
    &\int_{\mathbb{T}^d} \big(u^{\Bar{\eps}}(t)- u^{\hat{\eps}}(t)\big)^+ \,dx = \int_{\mathbb{T}^d}\int_{\R}f^{\Bar{\eps}}(x,t,\xi)\bar{f}^{\hat{\eps}}(x,t,\xi)\,d\xi\,dx  \\ &
=  \int_{(\mathbb{T}^d)^2}\int_{\R^2} \varrho_\delta(x-y)\rho_\lambda(\xi-\zeta)f^{\Bar{\eps}}(x,t,\xi)\bar{f}^{\hat{\eps}}(y,t,\zeta) \,d\xi\,d\zeta\,dx\,dy + \Bar{\Upsilon}(\delta, \lambda),\notag
\end{align}
where the error term $\Bar{\Upsilon}(\delta, \lambda)$ is given by
\begin{align*}
    &\Bar{\Upsilon}(\delta, \lambda):=  \int_{\mathbb{T}^d}\int_{\R}f^{\Bar{\eps}}(t)\bar{f}^{\hat{\eps}}(t)\,d\xi\,dx 
    - \int_{(\mathbb{T}^d)^2}\int_{\R^2} \varrho_\delta\rho_\lambda f^{\Bar{\eps}}(x,t,\xi)\bar{f}^{\hat{\eps}}(y,t,\zeta) \,d\xi\,d\zeta\,dx\,dy\,.
    \end{align*}
Following \cite[Step 2, Theorem $3.5$]{AC_2} together with \eqref{eq:G1-bound}, we have
\begin{align*}
\big| \Bar{\Upsilon}(\delta, \lambda)\big| \le 
 \int_{(\mathbb{T}^d)^2}\varrho_\delta|u_0^{\Bar{\eps}}(y) - u_0^{\Bar{\eps}}(x) |\,dx\,dy
   + C\big(\delta^{-1}\lambda + 
r^{2-2\theta}\delta^{-2} + r^{-2\theta}\lambda + \delta + \lambda\big)\,.
\end{align*}
In view of Theorem \ref{thm:contraction-principle} and above computations, \eqref{eq:Upsilon-1} yields, 
\begin{align}\label{eq:strong-convergence}
& \int_{\mathbb{T}^d}\Big(u^{\Bar{\eps}}(t) - u^{\hat{\eps}}(t)\Big)^+ \,dx \le \int_{(\mathbb{T}^d)^2}\varrho_\delta(x-y)\big|u_0^{\Bar{\eps}}(y) - u_0^{\Bar{\eps}}(x) \big|\,dx\,dy
  \\ & + C\big(\delta^{-1}\lambda + 
r^{2-2\theta}\delta^{-2} + r^{-2\theta}\lambda + \delta + \lambda\big)+ \mathcal{S}_{1, \Bar{\eps}} +  \mathcal{S}_{2, \Bar{\eps}} +  \mathcal{S}_{3, \Bar{\eps}} + \mathcal{S}^*,\notag
\end{align}
where $\mathcal{S}_{i, \Bar{\eps}}$ as defined previously and can be bounded as done in Theorem \ref{thm:contraction-principle}. The term $\mathcal{S}^*$ is defined as 
\begin{align*}
    \mathcal{S}^* = &(\Bar{\eps} + \hat{\eps})\int_0^t\int_{(\mathbb{T}^d)^2}\int_{\R^2} f^{\Bar{\eps}}\Bar{f}^{\hat{\eps}}\varrho_\delta(x-y)\rho_{\lambda}(\xi-\zeta)\,d\xi\,d\zeta\,dx\,dy\,ds\notag \\
    &- \int_0^t\int_{(\mathbb{T}^d)^2}\int_{\R^2} \varrho_\delta(x-y)\rho_{\lambda}(\xi-\zeta)\,d\mathcal{V}_{x,s}^{\Bar{\eps}}(\xi)\,dx\,d\eta^{\hat{\eps}}(y,s,\zeta)\notag \\
    &- \int_0^t\int_{(\mathbb{T}^d)^2}\int_{\R^2}\varrho_\delta(x-y)\rho_{\lambda}(\xi-\zeta)\,d\mathcal{V}_{y,s}^{\hat{\eps}}(\zeta)\,dy\,d\eta^{\Bar{\eps}}(x,s,\xi)\,.
\end{align*}
We can bound $\mathcal{S}^*$ as (see, \cite[Step 2, Thorem 3.2]{AC_2})
\begin{align}\label{eq:bound-s^*}
    |\mathcal{S}^*| \le C(\Bar{\eps}-\hat{\eps})^2\delta^{-2}.
\end{align}
With \eqref{eq:bound-s-12}, \eqref{eq:bound-s^*} and \eqref{eq:bound-s-3} in hand, we can further bound \eqref{eq:strong-convergence} as:
\begin{align*}
    & \int_{\mathbb{T}^d}\Big(u^{\Bar{\eps}}(t) - u^{\hat{\eps}}(t)\Big)^+ \,dx \le \int_{(\mathbb{T}^d)^2}\varrho_\delta(x-y)\big|u_0^{\Bar{\eps}}(y) - u_0^{\Bar{\eps}}(x) \big|\,dx\,dy\notag \\
  &\hspace{3cm} + C\big(\delta^{-1}\lambda + 
r^{2-2\theta}\delta^{-2} + r^{-2\theta}\lambda + \delta + \lambda\big) + C(\Bar{\eps} + \hat{\eps})\delta^{-2},
\end{align*}
and letting $\lambda \rightarrow 0$, for all $t \in [0,T]$, we get
\begin{align}\label{eq:strong-convergence-I}
     \int_{\mathbb{T}^d}\Big(u^{\Bar{\eps}}(t) - u^{\hat{\eps}}(t)\Big)^+ \,dx &\le \int_{(\mathbb{T}^d)^2}\varrho_\delta(x-y)\big|u_0^{\Bar{\eps}}(y) - u_0^{\Bar{\eps}}(x) \big|\,dx\,dy
   \\ &+ C\big( r^{2-2\theta}\delta^{-2}  + \delta\big)
+ C(\Bar{\eps} + \hat{\eps})\delta^{-2}\,.\notag
\end{align}
Set $ r = \delta^c$, with $c > \frac{1}{1-\theta}$ in \eqref{eq:strong-convergence-I}. Given $\Bar{\delta} > 0$, we can fix  $\delta$ small enough such that the first and second term in the right hand side of \eqref{eq:strong-convergence-I} is bounded by $\frac{\Bar{\delta}}{2}$. Also,  we can find $\hat{\delta} > 0$ such that  for any $\Bar{\eps}, \hat{\eps} < \hat{\delta}$, the third term of \eqref{eq:strong-convergence-I} is bounded by $\frac{\Bar{\delta}}{2}.$ It infers that the sequence $\{u^{\Bar{\eps}}\}$ is Cauchy sequence in $L^1([0,T];L^1(\mathbb{T}^d))$. Thus, we arrive at the following proposition.
\begin{prop}\label{prop:strong-convergence}
There exists $u \in \mathcal{E}_1$ such that $u^{\Bar{\eps}} \rightarrow u$ as $\Bar{\eps} \rightarrow 0$ in $\mathcal{E}_1$. 
\end{prop}
\noindent{\bf Step II:~Uniform $L^2$-estimate.} In this step, we establish the uniform bound for  $u^{\Bar{\eps}}$ in  $L^\infty([0,T]; L^2(\mathbb{T}^d))$. Applying chain rule to the function $g(v) = \|v\|_{L^2(\mathbb{T}^d)}^2$, we obtain
\begin{align}\label{eq:chain-rule}
    &\|u^{\Bar{\eps}}(t)\|_{L^2(\mathbb{T}^d)}^2 =\, \|u_0\|_{L^2(\mathbb{T}^d)}^2 - 2\int_0^t
    \langle u^{\Bar{\eps}}, \text{div}F( u^{\Bar{\eps}}) \rangle\,ds \\
    & - 2\int_0^t
    \langle u^{\Bar{\eps}}, (-\Delta)^\theta{\tt \Phi}(u^{\Bar{\eps}}) \rangle\,ds\ +\,2\Bar{\eps}\int_0^t
    \langle u^{\Bar{\eps}}, \Delta u^{\Bar{\eps}} \rangle\, ds \notag \\
    & + 2\int_0^t\|u^{\Bar{\eps}}\|_{L^2(\mathbb{T}^d)}\|{\tt h}(u^{\Bar{\eps}})\|_{\mathcal{L}_2(\mathcal{H};L^2(\mathbb{T}^d))}|{\ell}(s)|_{\mathcal{H}}\,ds\notag =: \|u_0\|_{L^2(\mathbb{T}^d)}^2 + \sum_{i=1}^4{\mathcal{J}_i}.
\end{align}
By setting $\mathcal{F}(z) = \int_0^z F(\Bar{z})\,d\Bar{z}$, we see that $\mathcal{J}_1=0$. 
By \eqref{esti:nondecreasing-phi}, we see that
\begin{align}\label{eq:boundJ2}
    \mathcal{J}_2 & 
    \le -\frac{1}{\|{\tt \Phi}'\|_{L^\infty}}\int_0^t 
    \big[{\tt \Phi}(u^{\Bar{\eps}})\big]_{H^\theta}^2\,ds, \quad \mathcal{J}_3 = -2\Bar{\eps}\, \int_0^t\int_{\mathbb{T}^d}
    |\nabla u^{\Bar{\eps}}|^2\,dx\,ds. 
\end{align}
Since $\|{\tt h}(u)\|_{\mathcal{L}_2(\mathcal{H};L^2(\mathbb{T}^d))} \le C(1 + \|u\|_{L^2(\mathbb{T}^d)})$, by using Young's inequality and the condition ${\ell}\in S_N$, we bound $\mathcal{J}_4$:
\begin{align}\label{eq:boundJ4}
  \mathcal{J}_4  \le C\Big(1 + \int_0^t\|u^{\Bar{\eps}}\|_{L^2(\mathbb{T}^d)}^2|{\ell}(s)|_{\mathcal{H}}\,ds\Big).  
\end{align}
Since ${\ell}\in S_N$, an application of Gronwall's lemma along with \eqref{eq:boundJ2}, 
and \eqref{eq:boundJ4}, \eqref{eq:chain-rule} yields
\begin{align}\label{eq:bound-fractional-1}
&\|u^{\Bar{\eps}}\|_{L^2(\mathbb{T}^d)}^2 + \frac{1}{\|{\tt \Phi}'\|_{L^\infty}}\int_0^t 
\big[{\tt \Phi}(u^{\Bar{\eps}})\big]_{H^\theta}^2\,ds + 2\Bar{\eps}\, \int_0^t
\|\nabla u^{\Bar{\eps}}\|_{L^2(\mathbb{T}^d)}^2\,ds \le C.
\end{align}
Again, in view of \eqref{eq:bound-fractional-1}, we have 
\begin{align}\label{eq:bound-fractional}
\sup_{0\le t\le T}\|u^{\Bar{\eps}}\|_{L^2(\mathbb{T}^d)}^2  + 
\int_0^T 
\|{\tt \Phi}(u^{\Bar{\eps}})\|_{H^\theta}^2\,ds \le C.
\end{align}
Moreover, we  also need to show that the kinetic measure $m^{\Bar{\eps}}$ corresponding to $u^{\Bar{\eps}}$  have uniform decay. In this regard, we follow \cite{Hofmanova-2016} and define a sequence of functions $\varphi_{n}(y) \in C^2(\R),~n \in \N$  having quadratic growth at infinity: for $p\in [2, \infty)$
\begin{equation}\label{eq:smooth-approx-kinetic}
    \varphi_{n}(y) =
    \begin{cases}
        |y|^p, \quad |y| \le n,
        \\
        n^{p-2}\Big[\frac{p(p-2)}{2}y^2 - p(p-2)n|y| + \frac{(p-1)(p-2)}{2}n^2\Big], \quad |y| > n.
    \end{cases}
\end{equation}
One can easily deduce that
\begin{equation}\label{bounds-for-varphi}
\begin{cases}
    |y\varphi_n^\prime(y)| \le C\varphi_n(y),~ |\varphi_n^\prime(y)| \le p(1 +\varphi_n(y) ),~
     |\varphi_n^\prime(y)| \le |y|\varphi_n^{\prime\prime}(y),  \\
      y^2\varphi_n^{\prime\prime}(y) \le p(p-1)\varphi_n(y),~
    \varphi_n^{\prime\prime}(y) \le p(p-1)(1 + \varphi_n(y)),~\\( \varphi_n^{\prime\prime}(y))^{\frac{p}{p-2}} \le C\varphi_n(y).
\end{cases} 
\end{equation}
Applying chain rule to $\varphi_n(u^{\bar{\eps}})$, for all $t \in [0,T]$ we get
\begin{align}\label{eq:chain-rule-varphi}
    &\int_{\mathbb{T}^d}\varphi_n(u^{\bar{\eps}})\,dx = \int_{\mathbb{T}^d}\varphi_n(u_0)\,dx - \int_0^t\big\langle \varphi_n^{\prime}(u^{\bar{\eps}}), \text{div}F(u^{\bar{\eps}})\big\rangle\,ds  \\ & - \int_0^t\big\langle \varphi_n^{\prime}(u^{\bar{\eps}}), (-\Delta)^\theta(u^{\bar{\eps}})\big\rangle\,ds + 
    \Bar{\eps}\int_0^t \big\langle \varphi_n^{\prime}(u^{\bar{\eps}}), \Delta u^{\Bar{\eps}}\big\rangle\,ds \notag \\ &+ \int_0^t \big\langle \varphi_n^{\prime}(u^{\bar{\eps}}), {\tt h(u^{\bar{\eps}})}{\ell}(s) \big\rangle\,ds\, := \int_{\mathbb{T}^d}\varphi_n(u_0)\,dx + \sum_{i =1}^4\Bar{\mathcal{J}}_i\notag
\end{align}
Similarly, $\Bar{\mathcal{J}}_1 = 0$ and $\mathcal{J}_4$ can be bounded as in \eqref{eq:boundJ4}. Also, we have
\begin{align}
    &-\int_0^t \big\langle \varphi_n^{\prime}(u^{\bar{\eps}}), (-\Delta)^\theta(u^{\bar{\eps}})\big\rangle\,ds  = -\underset{\kappa \rightarrow 0}{\text{lim}}\int_0^t \int_{\mathbb{T}^d}\int_{\R}
    \varphi_n^{\prime\prime}(\xi)\,d\eta_\kappa^{\Bar{\eps}}(x,\xi,s),\notag \\
    & \Bar{\eps}\int_0^t \big\langle \varphi_n^{\prime}(u^{\bar{\eps}}), \Delta u^{\Bar{\eps}}\big\rangle\,ds = - \Bar{\eps}\int_0^t \int_{\mathbb{T}^d}\varphi_n^{\prime\prime}(u^{\bar{\eps}})|\nabla u^{\bar{\eps}}|^2\,dx\,ds,\notag
\end{align}
where $\eta_{\kappa}^{\Bar{\eps}} = \int_{\R^d}\big|{\tt \Phi}(u_\kappa^{\Bar{\eps}}(x+z,t)) - {\tt \Phi}(\xi)\big|{\bf 1}_{{\rm Conv}\{u_\kappa^{\Bar{\eps}}(x+z, t), u_\kappa^{\Bar{\eps}}(x,t)\}}\,(\xi)\gamma(z)\,dz$
 and $u_\kappa^{\Bar{\eps}}$ is the mollification of $u^{\Bar{\eps}}$ in the space variable (see, \cite[Propostion $3.6$]{AC_2}).  With these estimations in hand, an application of Gronwall's lemma  and sending $n \rightarrow \infty$, for $p =2$, from \eqref{eq:chain-rule-varphi} we get
 \begin{align*}
     \sup_{0\le t\le T}\|u^{\Bar{\eps}}(t)\|_{L^2(\mathbb{T}^d)}^2 + 2\int_0^t \int_{\mathbb{T}^d}\int_{\R}
   d\eta^{\Bar{\eps}}(x,\xi,s) + 2\Bar{\eps}\int_0^t \int_{\mathbb{T}^d}|\nabla u^{\bar{\eps}}|^2\,dx\,ds \le C,
 \end{align*}
 where $\eta^{\Bar{\eps}}$ is the $weak^*$ limit of $\eta_{\kappa}^{\Bar{\eps}}$ in $\mathcal{R}^+(\mathbb{T}^d \times [0,T] \times \R)$ (see \cite[Appendix A]{AC_2}). Setting
$m^{\Bar{\eps}} := \Bar{m}^{\Bar{\eps}} + \eta^{\Bar{\eps}}$, with $\Bar{m}^{\Bar{\eps}} = \Bar{\eps}|\nabla u^{\Bar{\eps}}|^2\delta_{u^{\Bar{\eps}} = \xi}$, we  have 
\begin{align}\label{eq:bound-kinetic-measure}
\underset{0 \le t \le T}{\text{sup}}\|u^{\Bar{\eps}}(t)\|_{L^2(\mathbb{T}^d)}^2  +\int_0^t\int_{\mathbb{T}^d}\int_{\R}\,dm^{\Bar{\eps}}(x,s,\xi) \le C, \quad \forall~t\in [0,T]\,.
\end{align}
\noindent{\bf Step III:~Convergence of approximate kinetic function.} Let us define Young measure on $\mathbb{T}^d \times [0,T]$ as $\mathcal{V}^{\Bar{\eps}_n}_{x,t}(\xi) = \delta_{u^{\Bar{\eps}_n}(t,x)= \xi}$ and $\mathcal{V}_{x,t}(\xi) = \delta_{u(x,t)= \xi}$, where $u$ is given in Proposition \ref{prop:strong-convergence}. By the uniform $L^2$-bound on $u^{\Bar{\eps}}$ in \eqref{eq:bound-fractional-1}, we have
\begin{align}
    \underset{t \in [0,T]}{sup}\int_{\mathbb{T}^d}\int_{\R}|\xi|^2\,d\mathcal{V}_{x,t}^{\Bar{\eps}_n}(\xi)\,dx\,\le C.\label{eq:youngs-measure-bound}
\end{align}
Define a kinetic function $f(x,t,\xi):=\mathcal{V}_{x,t}(\xi, \infty).$
A direct consequence of compactness result of Young measure and kinetic function from \cite[Theorem 5, Corollary 6, $\&$~ Propostion 4.3]{Vovelle2010} yields, $\mathcal{V}^{\Bar{\eps}_n} \rightarrow \mathcal{V}$ 
and $f^{\Bar{\eps}_n} \rightarrow f$ in $L^\infty(\mathbb{T}^d \times [0,T] \times \R)-weak^*,$ and  $\mathcal{V}$ satisfies \eqref{eq:youngs-measure-bound}. 

\noindent{\bf Step IV:~Convergence of Kinetic measure.}  Recall that, $\mathcal{R}^+(\mathbb{T}^d \times [0,T] \times \mathbb{R})$ is space of bounded non-negative Borel measures on $\mathbb{T}^d \times [0,T] \times \mathbb{R}$ equipped with the total variation norm of measure, and it is the dual space to $C_0(\mathbb{T}^d \times [0,T] \times \mathbb{R})$, the collection of all continuous function vanishing at infinity endowed with supremum norm. Due to \eqref{eq:bound-kinetic-measure}, we have for $p=2$,
$$|m^{\Bar{\eps}_n}|([0,T] \times \mathbb{T}^d \times \R) \le \int_0^T\int_{\mathbb{T}^d}\int_{\R}\,dm^{\Bar{\eps}_n}(x,s,\xi) \le C ,$$
where $|m^{\Bar{\eps}_n}|$ is the total variation of measure. Thus, $m^{\Bar{\eps}_n} \rightarrow m$~(up to a subsequence) in $\mathcal{R}^+(\mathbb{T}^d \times [0,T] \times \R)-weak^*$. It only remains to prove that $m$ is a kinetic measure. For this purpose, first we establish the vanishing property of $m$ for large enough $\xi$. Let $(\mathcal{Z}_{\delta_0})$ be a truncation on $\R$. From \eqref{eq:bound-kinetic-measure}, we have
\begin{align*}
&\int_0^t\int_{\mathbb{T}^d}\int_{\R}\,dm(x,s,\xi)  \le \underset{\delta_0 \rightarrow 0}{\lim\inf} \int_0^t\int_{\mathbb{T}^d}\int_{\R}\mathcal{Z}_{\delta_0}\,dm(x,s,\xi)\notag \\
& =  \underset{\delta_0 \rightarrow 0}{\lim\inf}\,  \underset{\Bar{\eps}_n \rightarrow 0}{\lim\inf} \int_0^t\int_{\mathbb{T}^d}\int_{\R}\mathcal{Z}_{\delta_0}\,dm^{\Bar{\eps}_n}(x,s,\xi) \le C.
\end{align*}
It infers that $m$ vanishes for large $\xi$. A similar arguments as in \cite[Lemma 3.8]{AC_2} yields that
$$ m(x,t,\xi) \ge \eta(x,t,\xi):=\int_{\R}\big|{\tt \Phi}(u(x+z,t)) - {\tt \Phi}(\xi)\big| {\bf 1}_{{\rm Conv}\{u(x+z, t), u(x,t)\}}\,(\xi)\lambda(z)\,dz\,.$$
\noindent\textbf{ Step-V:~Passing to the limit along a subsequence:} In view of strong convergence of $u^{\Bar{\eps}}$ (cf.~ Proposition \ref{prop:strong-convergence}), the convergence results form Steps ${\rm III}\& {\rm IV}$, and \cite[Lemma $2.1 ~\&$ Propostition $4.9$]{Vovelle-2020}, we pass to the limit in \eqref{eq:kinetic-formulation-regularized-viscous} and conclude that $f(x,t,\xi) ={\bf 1}_{u(x,t) > \xi}$  satisfies \eqref{eq:skeleton-k-ref}. In other words, $u(t)$ is a kinetic solution of skeleton equation \eqref{eq:skeleton}. Thus, we conclude the following theorem.
\begin{thm}
    Let $u_0 \in L^2(\mathbb{T}^d)$ and the assumptions \ref{A1}-\ref{A3} be true. Then for any $t \in [0,T]$, \eqref{eq:skeleton} has a kinetic solution $u_{\ell}$.
\end{thm}
\section{Large deviation principle}\label{sec:LDP}
In this section, we prove large deviation principle for $u^\eps$ i.e., Theorem \ref{thm:ldp}. To do so, we wish to validate ${\rm i)}$ and ${\rm ii)}$ of Condition \ref{cond:2}. 
\subsection{Validation of ${\rm ii)}$ of Condition \ref{cond:2}}\label{subsec:Validation-of-continuity-of-g-nut}
Let ${\ell}, \{{\ell}_\eps\} \subset S_N$ be such that ${\ell}_\eps$ converges to ${\ell}$ weakly in $L^2([0,T]; \mathcal{H})$. Let $u_{{\ell}_\eps}$ be the kinetic solution of \eqref{eq:skeleton} for ${\ell}={\ell}_\eps$, and $u_{{\ell}_\eps}^{\Bar{\eps}}$ be the unique weak solution of \eqref{eq:regularized-skeleton}
with ${\ell}$ replaced by ${\ell}_\eps$. To validate ${\rm ii)}$ of Condition \ref{cond:2}, it suffices to  show the convergence of $u_{{\ell}_\eps}$ to the kinetic solution $u_{\ell}$ of \eqref{eq:skeleton} in $\mathcal{E}_1$. Since, for any $\Bar{\eps}, \eps > 0$
\begin{align}\label{inq:trinagle-imp}
    &\|u_{{\ell}_\eps} - u_{\ell}\|_{\mathcal{E}_1} \le   \|u_{{\ell}_\eps}^{\Bar{\eps}} - u_{{\ell}_\eps}\|_{\mathcal{E}_1} + \|u_{{\ell}_\eps}^{\Bar{\eps}} - u_{{\ell}}^{\Bar{\eps}}\|_{\mathcal{E}_1} + \|u_{{\ell}}^{\Bar{\eps}} - u_{\ell}\|_{\mathcal{E}_1},
\end{align}
we need to show that right hand side of \eqref{inq:trinagle-imp} goes to zero as $\eps, \Bar{\eps}$ tend to zero. 
\begin{prop}\label{prop:compactness}
    For any $\Bar{\eps} >0$, $\{u_{{\ell}_\eps}^{\Bar{\eps}}; \eps > 0\}$ is compact in $L^2([0,T]; L^2(\mathbb{T}^d)).$
\end{prop}
\begin{proof}
From \eqref{eq:regularized-skeleton}, we have
\begin{align}
     &u_{{\ell}_\eps}^{\Bar{\eps}}(t) = u_0 - \int_0^t\text{div}F(u_{{\ell}_\eps}^{\Bar{\eps}}(r,x))\,dr -\int_0^t (-\Delta)^\theta[{\tt \Phi}(u_{{\ell}_\eps}^{\Bar{\eps}}(r,\cdot))](x)\,dr \notag \\ &\hspace{2cm}+ \, \Bar{\eps}\int_0^t\Delta u_{{\ell}_\eps}^{\Bar{\eps}}(r)\,dr + \int_0^t{\tt h}(u_{{\ell}_\eps}^{\Bar{\eps}}(r,x)){\ell}_\eps(r)\,dr\,
     =: \sum_{i=1}^{5}\mathcal{A}_i^\eps(t)\,.\notag
\end{align}    
Observe that 
\begin{align}\label{eq:bound-A3}
    & \|(-\Delta)^\theta[{\tt \Phi}(u_{{\ell}_\eps}^{\Bar{\eps}}(r,\cdot))]\|_{H^{-1}(\mathbb{T}^d)} = \underset{\|v\|_{H^1(\mathbb{T}^d)} \le 1}{\sup} \big|\big\langle v, (-\Delta)^\theta[{\tt \Phi}(u_{{\ell}_\eps}^{\Bar{\eps}}(r,\cdot))](\cdot)\big\rangle\big|\\
        & \le \underset{\|v\|_{H^1(\mathbb{T}^d)} \le 1}{\sup}C(d,\theta)\Big(\|v\|_{H^1(\mathbb{T}^d)} \times \|{\tt \Phi}(u_{{\ell}_\eps}^{\Bar{\eps}}(r,\cdot))\|_{H^\theta(\mathbb{T}^d)}\Big)\le C \|{\tt \Phi}(u_{{\ell}_\eps}^{\Bar{\eps}}(r,\cdot))\|_{H^\theta(\mathbb{T}^d)}. \notag  
\end{align}
Invoking Jensen's inequality, \eqref{eq:bound-fractional} and \eqref{eq:bound-A3}, we have 
\begin{align*}
    \|\mathcal{A}_3^\eps(t) - \mathcal{A}_3^\eps(s)\|_{H^{-1}(\mathbb{T}^d)}^2
    & \le C(t-s)\int_s^t \|{\tt \Phi}(u_{{\ell}_\eps}^{\Bar{\eps}}(r,\cdot))\|_{H^{\theta}}^2\,dr  \le C(t-s),
\end{align*}
and therefore, for $\beta \in (0,\frac{1}{2})$, $~~ \underset{\eps}{\sup}\,\|\mathcal{A}_3^\eps\|_{W^{\beta,2}([0,T]; H^{-1}(\mathbb{T}^d))}^2  \le C(\beta, \|u_0\|)$.
By the assumption \ref{A3}, we have
\begin{align}\label{eq:bound-nosie-coefficent}
\|{\tt h}(u_{{\ell}_\eps}^{\Bar{\eps}}){\ell}_\eps(r)\|_{H^{-1}} 
& \le C \|{\tt h}(u_{{\ell}_\eps}^{\Bar{\eps}})\|_{\mathcal{L}_2(\mathcal{H},L^2(\mathbb{T}^d))}|{\ell}_\eps(r)|_{\mathcal{H}} \le C(1 + \|u_{{\ell}_\eps}^{\Bar{\eps}}\|_{L^2(\mathbb{T}^d)})|{\ell}_\eps(r)|_{\mathcal{H}}.
\end{align} 
Thanks to \eqref{eq:bound-nosie-coefficent}, Jensen inequality and the fact that ${\ell}_\eps \in S_N$,
\begin{align*}
    \|\mathcal{A}_5^\eps(t) - \mathcal{A}_5^\eps(s)\|_{H^{-1}}^2 
    & \le C(t-s)\Big(1 + \underset{t \in [0,T]}{sup}\|u_{{\ell}_\eps}^{\Bar{\eps}}(t)\|_{L^2(\mathbb{T}^d)}^2\Big)\int_s^t|{\ell}_\eps(r)|_{\mathcal{H}}^2\,dr \le C(t-s).
\end{align*}
Thus, there exists a constant $C>0$, independent of $\Bar{\eps}, \eps$ such that 
\begin{align*}
    \underset{\eps}{\sup}\|\mathcal{A}_5^\eps\|_{W^{\beta,2}([0,T]; H^{-1}(\mathbb{T}^d))} \le C(\beta) \quad \beta \in (0, \frac{1}{2}). 
\end{align*}
With \eqref{eq:bound-fractional-1} in hand, one can follow similar line of arguments as done in Lemma \ref{lem:compactness-gamma-eps} to bound  the remaining terms. Thus, 
for $\beta \in (0, \frac{1}{2})$, we have
$$\underset{\eps}{\sup}\,\|u_{{\ell}_\eps}^{\Bar{\eps}}\|_{W^{\beta,2}([0,T]; H^{-1}(\mathbb{T}^d))}^2 \le C(\beta).$$
By using \eqref{eq:bound-fractional-1} and applying Theorem \ref{thm:compactness}, we conclude the desire compactness. 
\end{proof}
\begin{prop}\label{prop:convergence}
    Let ${\ell}_\eps \rightarrow {\ell}$ weakly in $L^2([0,T];\mathcal{H})$, then, for fixed $\Bar{\eps}>0$, $$\underset{\eps \rightarrow 0}{\lim}\|u_{{\ell}_\eps}^{\Bar{\eps}} - u_{\ell}^{\Bar{\eps}}\|_{L^1([0,T];L^1(\mathbb{T}^d))}= 0.$$
\end{prop}
\begin{proof}
The required result can be accomplished if we infer that for any sequence $\eps_m$, there exist a subsequence $\eps_{m_k}$ such that 
$
   \underset{k \rightarrow \infty}{\lim}\,\|u_{{\ell}_{\eps_{m_k}}}^{\Bar{\eps}}- u_{\ell}^{\Bar{\eps}}\|_{\mathcal{E}_1}= 0$,
where $u_{{\ell}}^{\Bar{\eps}}$ is the weak solution of \eqref{eq:regularized-skeleton}.
By \eqref{eq:bound-fractional} and Proposition \ref{prop:compactness}, we have for any sequence $\eps_m \rightarrow 0$, there exist a subsequence $\{\eps_{m_k}\}_{\{k \ge 1\}}$ and an element ${\pmb u} \in L^\infty([0,T]; L^2(\mathbb{T}^d))$
$ \cap L^2([0,T];H^\theta(\mathbb{T}^d))$ such that 
\begin{align}\label{eq:uhat-convergence}
     &u_{{\ell}_{\eps_{m_k}}}^{\Bar{\eps}} \rightarrow {\pmb u}~~~\text{in}~~~L^2([0,T]; L^2(\mathbb{T}^d)), \\ & 
     u_{{\ell}_{\eps_{m_k}}}^{\Bar{\eps}} \rightharpoonup {\pmb u}~~~\text{in}~~~ L^\infty([0,T]; L^2(\mathbb{T}^d)) \cap  L^2([0,T]; H^\theta(\mathbb{T}^d))\,.\notag
\end{align}    
For any test function $\psi \in C_c^2(\mathbb{T}^d)$, it holds that 
\begin{align}\label{eq:prop-convergence}
    &\langle u_{{\ell}_{\eps_{m_k}}}^{\Bar{\eps}}(t), \psi \rangle - \langle u_0, \psi \rangle = -\int_0^t\langle F(u_{{\ell}_{\eps_{m_k}}}^{\Bar{\eps}}(r,\cdot)), \nabla \psi \rangle\,dr +  \Bar{\eps}\int_0^t\langle u_{{\ell}_{\eps_{m_k}}}^{\Bar{\eps}}, \Delta\psi\rangle\,dr \\ &-\int_0^t \langle{\tt \Phi}(u_{{\ell}_{\eps_{m_k}}}^{\Bar{\eps}}(r,\cdot)), (-\Delta)^\theta[\psi(\cdot)](\cdot)\rangle\,dr + \int_0^t\langle{\tt h}(u_{{\ell}_{\eps_{m_k}}}^{\Bar{\eps}}(r,\cdot)){\ell}_{{\eps_{m_k}}}(r), \psi\rangle\,dr\,.\notag
\end{align}
In view of \eqref{eq:uhat-convergence}, we have, for a.e. $t\in [0,T]$
\begin{align*}
   \big|\langle u_{{\ell}_{\eps_{m_k}}}^{\Bar{\eps}}(t)- {\pmb u}(t), \psi \rangle \big| \longrightarrow 0 \quad \text{as} \quad k \rightarrow \infty. 
\end{align*}
An application of Lipschitz continuity property of $F$ and \eqref{eq:uhat-convergence} yields
\begin{align*}
    \int_0^t\langle F(u_{{\ell}_{\eps_{m_k}}}^{\Bar{\eps}}(r,\cdot))- F({\pmb u}(r,\cdot)), \nabla \psi \rangle\,dr \le C\|u_{{\ell}_{\eps_{m_k}}}^{\Bar{\eps}} - {\pmb u}\|_{L^2([0,T];L^2(\mathbb{T}^d))}^2\rightarrow 0.
\end{align*}
In view of \eqref{eq:uhat-convergence},  and the assumption \ref{A2}, we see that 
\begin{align*}
&\int_0^t \langle{\tt \Phi}(u_{{\ell}_{\eps_{m_k}}}^{\Bar{\eps}}(r,x)) - {\tt \Phi}({\pmb u}(r,x)), (-\Delta)^\theta[\psi(\cdot)](x)\rangle\,dr  \notag \\& \hspace{1cm} \le C \,\|(-\Delta)^\theta[\psi](\cdot)\|_{L^\infty(\mathbb{T}^d)}\|u_{{\ell}_{\eps_{m_k}}}^{\Bar{\eps}} - {\pmb u}\|_{L^2([0,T]; L^2(\mathbb{T}^d))}^2 \rightarrow 0\,, \\
 & \Bar{\eps}\int_0^t\langle u_{{\ell}_{\eps_{m_k}}}^{\Bar{\eps}}- {\pmb u}, \Delta\psi\rangle\,dt \le C(\Bar{\eps})\|\nabla\psi\|_{L^\infty(\mathbb{T}^d)}\int_0^t\|u_{{\ell}_{\eps_{m_k}}}^{\Bar{\eps}} - {\pmb u}\|_{L^2([0,T];L^2(\mathbb{T}^d))}^2\, \rightarrow 0.
\end{align*}
We now focus on the controlled drift term. Observe that 
\begin{align}
&\int_0^t\big\langle{\tt h}(u_{{\ell}_{\eps_{m_k}}}^{\Bar{\eps}}){\ell}_{{\eps_{m_k}}}(r)-{\tt h}({\pmb u}){\ell}(r), \psi\big\rangle\,dr\notag  \\&= \int_0^t\big\langle\big({\tt h}(u_{{\ell}_{\eps_{m_k}}}^{\Bar{\eps}})-{\tt h}({\pmb u})\big){\ell}_{{\eps_{m_k}}}(r), \psi\big\rangle\,dr +  \int_0^t\big\langle{\tt h}({\pmb u})\big({\ell}_{{\eps_{m_k}}}(r)-{\ell}(r)\big),\psi\big\rangle\,dr\equiv \mathcal{D}_1 + \mathcal{D}_2.\notag
\end{align}
Since ${\tt h}({\pmb u})$ is a Hilbert-Schmidt operator from $\mathcal{H}$ to $L^2(\mathbb{T}^d)$, it is compact and therefore ${\tt h}({\pmb u}){\ell}_{{\eps_{m_k}}} \goto {\tt h}({\pmb u}){\ell}$ in 
$L^2([0,T]; L^2(\mathbb{T}^d))$. In particular $\mathcal{D}_2\goto 0$. 
Since ${\ell}_\eps \in S_N$, we have, thanks to \eqref{eq:uhat-convergence}
\begin{align*}
\mathcal{D}_1 \le\, C(\psi)\|u_{{\ell}_{\eps_{m_k}}}-{\pmb u}\|_{L^2([0,T];L^2(\mathbb{T}^d))}\|{\ell}_{{\eps_{m_k}}}\|_{L^2([0,T]; \mathcal{H})}
 \rightarrow 0.
\end{align*}
In view of above convergence results, sending $k \rightarrow \infty$ in \eqref{eq:prop-convergence}, we have
\begin{align*}
    &\langle {\pmb u}(t), \psi \rangle - \langle u_0, \psi \rangle = -\int_0^t\langle F({\pmb u}(r,\cdot)), \nabla \psi \rangle\,dr -\int_0^t \langle{\tt \Phi}({\pmb u}(r,\cdot)), (-\Delta)^\theta[\psi(\cdot)](\cdot)\rangle\,dr \notag \\ &\hspace{4cm}+ \, \Bar{\eps}\int_0^t\langle {\pmb u}, \Delta\psi\rangle\,dr + \int_0^t\langle{\tt h}({\pmb u}(r,\cdot)){\ell}(r), \psi\rangle\,dr \quad \forall~\psi\in C_c^2(\mathbb{T}^d)\,.
\end{align*}
In other words, ${\pmb u}$ is a weak solution of \eqref{eq:regularized-skeleton}. By uniqueness of solution of \eqref{eq:regularized-skeleton}, we infer that ${\pmb u} = u_{\ell}^{\Bar{\eps}}$ and hence 
$\underset{\eps \rightarrow 0}{\lim}\|u_{{\ell}_\eps}^{\Bar{\eps}} - u_{\ell}^{\Bar{\eps}}\|_{\mathcal{E}_1}= 0$.
\end{proof}
In view of Propositions \ref{prop:strong-convergence} and \ref{prop:convergence}, the assertion Condition \ref{cond:2}, ${\rm ii)}$ follows from \eqref{inq:trinagle-imp}. 
\subsection{Validation of Condition \ref{cond:2}, ${\rm i)}$}
\noindent For any family $\{{\ell}_\eps; 0 < \eps < 1\} \subset\mathcal{A}_N$ with ${\ell}_\eps = \sum_{k \ge 1}{\ell}_{\eps, k}e_k$, let us consider the following SPDE:
\begin{equation}\label{eq:small-noise-ldp}
\begin{cases}
     \displaystyle d\Bar{u}^\eps + \text{div}F(\Bar{u}^\eps(t,x))\,dt + (-\Delta)^\theta[{\tt \Phi}(\Bar{u}^\eps(t,\cdot))](x)\,dt \\=  {\tt h}(\Bar{u}^\eps(t,x)){\ell}_\eps(t)\,dt + \sqrt{\eps}{\tt h}(\Bar{u}^\eps(t,x))\,dW(t),\quad
    \Bar{u}^\eps(0,x) = u_0(x)\,.
\end{cases}
\end{equation}
In view of Theorems \ref{thm:Existence and uniqueness}, \ref{thm:contraction-principle}, Section \ref{sec:existance-skeleton}
and Girsanov theorem, there exists a unique kinetic solution $\Bar{u}^{\eps}$ of equation \eqref{eq:small-noise-ldp}. Moreover, following \cite[Propostion 3.6]{AC_2} and Step ${\rm II}$ of Section \ref{sec:existance-skeleton}
\begin{align}\label{inq:ldp-cp}
   \mathbb{E}\Big[\underset{0 \le t \le T}{sup}\|\Bar{u}^\eps(t)\|_{L^2(\mathbb{T}^d)}^2\Big] \le C.
\end{align}
By definition of $\mathcal{G}^\eps$, we have $\mathcal{G}^\eps\Big(W(\cdot) + \frac{1}{\sqrt{\eps}}\int_0^{\cdot}{\ell}_\eps(s)\,ds\Big) = \Bar{u}^\eps(\cdot).$ In order to validate ${\rm i)}$ of Condition \ref{cond:2},  we only need to prove the following: for every $N < \infty$ and any family $\{{\ell}_\eps; \eps > 0\} \subset S_N$ 
\begin{align}
  \big\|\bar{u}^\eps(\cdot) - v^\eps(\cdot)\big\|_{\Large L^1([0,T];L^1(\mathbb{T}^d))} \longrightarrow 0 \quad \text{in Probability}, \label{conv:final-i0-LDP}
\end{align}
where $ v^\eps := \mathcal{G} ^0\Big(\int_0^{\cdot}{\ell}_\eps(s)\,ds\Big)$ is the kinetic solution of \ref{eq:skeleton} with ${\ell}$ replaced by ${\ell}_\eps$. To prove \eqref{conv:final-i0-LDP}, we use doubling the variables technique.  Let  $f_\eps(x,t,\xi) := {\bf 1}_{\Bar{u}^\eps(x,t) > \xi}$ and $g_\eps(y,t,\zeta):={\bf 1}_{{v}^\eps(y,t) > \zeta}$. Then, following Proposition \ref{prop:doubling-variable}, we have 
\begin{align*}
   & \langle\langle f_\eps^{\pm}(t)\Bar{g}_\eps^{\pm}(t), \psi \rangle \rangle\, =\, \langle\langle f_{\eps,0}\Bar{g}_{\eps,0}, \psi \rangle \rangle \notag \\ &+ \int_0^t\int_{(\mathbb{T}^d)^2}\int_{\R^2} f_\eps\Bar{g}_\eps\big(F'(\xi)\cdot \nabla_x + F'(\zeta)\cdot \nabla_y\big)\psi \,d\xi\,d\zeta\,dx\,dy\,ds\notag\\
   &- \int_0^t\int_{(\mathbb{T}^d)^2}\int_{\R^2}f_\eps\Bar{g}_\eps\big( {\tt \Phi}'(\xi)(-\Delta)^\theta [\psi](x)+ {\tt \Phi}'(\zeta)(-\Delta)^\theta[\psi](y)\big)\,d\xi\,d\zeta\,dx\,dy\,ds\notag \\
   &+ \sum_{k=1}^{\infty} \int_0^t\int_{(\mathbb{T}^d)^2}\int_{\R^2}\Bar{g}_\eps^{\pm}(s,y,\zeta){\tt h}_k(x,\xi)\psi {\ell}_{k,\eps}(s)\,d\mu_{x,\eps}^{1,k}(\xi)\,d\zeta\,dx\,dy\,ds\notag \\
   &-\sum_{k=1}^{\infty} \int_0^t\int_{(\mathbb{T}^d)^2}\int_{\R^2}f_\eps^{\pm}(s,x,\xi){\tt h}_k(y,\zeta)\psi {\ell}_{k,\eps}(s)\,d\mu_{y,\eps}^{2,k}(\zeta)\,d\xi\,dx\,dy\,ds \notag \\
   & + \frac{\eps}{2}\int_0^t\int_{(\mathbb{T}^d)^2}\int_{\R^2}\partial_{\xi}\psi \Bar{g}_\eps^{\pm}(s,y,\zeta)G^2(x,\xi)d\mu_{x,s}^{1,\eps}(\xi)\,d\zeta\,dx\,dy\,ds \notag \\
& + \sqrt{\eps}\sum_{k 
 \ge 1}\int_0^t\int_{(\mathbb{T}^d)^2}\int_{\R^2}\Bar{g}_\eps^{\pm}(s,y,\zeta){\tt h}_k(x,\xi)\psi\,d\zeta\,d\mu_{x,s}^{1,\eps}(\xi)\,dx\,dy\,d\beta_k(s)\notag \\
   &-\int_0^t\int_{(\mathbb{T}^d)^2}\int_{\R^2}\Bar{g}_\eps^+(s,y,\zeta)\partial_\xi\psi\,d\Bar{m}^\eps(x,\xi,s)\,d\zeta\,dy\notag \\
  & + \int_0^t\int_{(\mathbb{T}^d)^2}\int_{\R^2}f_\eps^-(s,x,\xi)\partial_\zeta\psi\,dm_{{\ell}_\eps}(y, \zeta, s)\,d\xi\,dx =: \langle\langle f_{\eps,0}\Bar{g}_{\eps,0}(t), \psi \rangle \rangle + \sum_{i=1}^{8}{\mathcal{M}_i}\quad  \mathcal{P} \text{-a.s}.
\end{align*}
In the above, $\psi(x,y,\xi,\zeta) = \varrho_\delta(x-y)\rho_\lambda(\xi -\zeta)$, $g_{\eps,0}:={\bf 1}_{u_0 > \zeta},~\mu_{y,s}^{2,k}(\zeta):= \partial_\zeta \Bar{g}_\eps^{+}(s,y,\zeta)$, $f_{\eps,0}:={\bf 1}_{u_0 > \xi}$, $\mu_{x,s}^{1,\eps}(\xi):= -\partial_\xi f_\eps^{+}(s,x,\xi) $, $m_{{\ell}_\eps}$ is the kinetic measure for $v^\eps$ and $\Bar{m}^\eps \in \mathcal{R}^+([0,T] \times \mathbb{T}^d \times \R)$ is the kinetic measure for $\Bar{u}^\eps$.
It is easy to observe that,
\begin{align}
&\sum_{i=7}^{8}\mathcal{M}_i\notag \le \, \int_0^t\int_{(\mathbb{T}^d)^2\times \R^2}\Bar{g}_\eps\partial_\zeta\psi\,d\Bar{\eta}^\eps_{1}(x,\xi,s)\,dy\,d\zeta \\ & \hspace{3cm}- \int_0^t\int_{(\mathbb{T}^d)^2\times \R^2}f_\eps\partial_\xi\psi \,d\eta_{2}^\eps(y, \zeta, s)\,dx\,d\xi \notag :=\mathcal{M}_{7,8}\,.
\end{align}
Similar to  Proposition \ref{prop:doubling-variable}, we get
\begin{align}\label{inq:similar-to pro-3.1}
& \int_{(\mathbb{T}^d)^2}\int_{\R^2} \varrho_\delta\rho_\lambda f_\eps^{\pm}(x,t,\xi)\bar{g}_\eps^{\pm}(y,t,\zeta) 
\,d\xi\,d\zeta\,dx\,dy \\
&\,\le \int_{(\mathbb{T}^d)^2}\int_{\R^2} \varrho_\delta\rho_\lambda f_{\eps,0}(x,\xi)\bar{g}_{\eps,0}(y,\zeta) 
\,d\xi\,d\zeta\,dx\,dy + \sum_{i=1}^{6}\mathcal{M}_i + \mathcal{M}_{7,8} 
\quad \mathcal{P}-a.s. \notag
\end{align}
By following similar line of arguments as done in \cite[Theorem $11$]{Vovelle2010}, \cite[Theorem 3.2]{AC_1}
\begin{align}
 |\mathcal{M}_1| \le C\delta^{-1}\lambda,\, (|\Bar{\mathcal{M}}_2| + |\mathcal{M}_{7,8}|) \le Cr^{2-2\theta}\delta^{-2} + Cr^{-2\theta}\lambda, \,\text{and} \, |\mathcal{M}_5| \le C\eps\lambda^{-1}.\notag
\end{align}
We follow \cite[Theorem 6.1]{Zhang-2020} to estimate $\mathcal{M}_3$ and $\mathcal{M}_4$ as follows.
\begin{align}
&\mathcal{M}_3 + \mathcal{M}_4 \notag \\
   & \le C \int_0^t|{\ell}_{\eps}(s)|_{\mathcal{H}}\int_{(\mathbb{T}^d)^2}\int_{(\R)^2}\varrho_\delta(x-y)\psi_{2,\lambda}(\xi,\zeta)|x-y|\,d\mu_{x,\eps}^{1,k}\otimes\,d\mu_{y,\eps}^{2,k}(\zeta, \xi)\,dx\,dy\,ds \notag \\
   &\,+  \,C \int_0^t|{\ell}_{\eps}(s)|_{\mathcal{H}}\int_{(\mathbb{T}^d)^2}\int_{(\R)^2}\varrho_\delta(x-y)\psi_{2,\lambda}(\xi,\zeta)|\xi-\zeta|\,d\mu_{x,\eps}^{1,k}\otimes\,d\mu_{y,\eps}^{2,k}(\zeta, \xi)\,dx\,dy\,ds \notag \\ &:= \mathcal{M}_{3,1} + \mathcal{M}_{4,1}, \quad \text{where} \quad \psi_{2,\lambda}(\xi,\zeta) := \int_{\zeta}^\infty\rho_\lambda(\xi-\zeta')\,d\zeta', \quad |\mathcal{M}_{3,1}| \notag \le C\delta,\\
    &\mathcal{M}_{4,1} \le C\lambda + C\int_0^t|{\ell}_{\eps}(s)|_{\mathcal{H}}\int_{(\mathbb{T}^d)^2}\int_{(\R)^2}\varrho_\delta\rho_\lambda\notag f_\eps^{\pm}(x,t,\xi)\bar{g}_\eps^{\pm}(y,t,\zeta) 
    \,dx\,dy\,d\xi\,d\zeta\,ds\,. 
\end{align}
Combining all the above estimations in \eqref{inq:similar-to pro-3.1}, it yields
\begin{align}
   & \int_{(\mathbb{T}^d)^2}\int_{\R^2} \varrho_\delta\rho_\lambda f_\eps^{\pm}(x,t,\xi)\bar{g}_\eps^{\pm}(y,t,\zeta) 
\,d\xi\,d\zeta\,dx\,dy\notag \le \Gamma(f_{\eps, 0},g_{\eps, 0}, \delta, \lambda) + |\Upsilon_0(\delta, \lambda)|+ |\mathcal{M}_6|(t) \notag \\ & \, + C\int_0^t|{\ell}_{\eps}(s)|_{\mathcal{H}}\int_{(\mathbb{T}^d)^2}\int_{(\R)^2}\varrho_\delta\notag\rho_\lambda\notag f_\eps^{\pm}(x,t,\xi)\bar{g}_\eps^{\pm}(y,t,\zeta)
\,dx\,dy\,d\xi\,d\zeta\,ds, \quad \mathcal{P}-a.s,
\end{align}
\begin{align*}
\text{where},\quad\Gamma(f_{\eps, 0},g_{\eps, 0}, \delta, \lambda) := &\int_{\mathbb{T}^d}\int_{\R}f_{\eps,0}(x,\xi)\bar{g}_{\eps,0}(x,\xi)
\,d\xi\,dx \notag \\ &+ C\big(\delta^{-1}\lambda + r^{2-2\theta}\delta^{-2} + r^{-2\theta}\lambda + \eps\lambda^{-1} + \delta + \lambda \big).\notag
\end{align*}
Applying Gronwall's lemma, we get $ \mathcal{P}-a.s.$
\begin{align*}
  & \int_{(\mathbb{T}^d)^2}\int_{\R^2} \varrho_\delta\rho_\lambda f_\eps^{\pm}(x,t,\xi)\bar{g}_\eps^{\pm}(y,t,\zeta)
\,d\xi\,d\zeta\,dx\,dy \le \Gamma(f_{\eps, 0},g_{\eps, 0}, \delta, \lambda) + |\Upsilon_0(\delta, \lambda)|+ |\mathcal{M}_6|(t)\notag \\ & + \int_0^t\big(\Gamma(f_{\eps, 0},g_{\eps, 0}, \delta, \lambda) + |\Upsilon_0(\delta, \lambda)| + |\mathcal{M}_6|(s)\big)|{\ell}_\eps(s)|_{\mathcal{H}}\exp\big(\int_s^t|{\ell}_\eps(r)|_{\mathcal{H}}\,dr\big)\,ds. 
\end{align*}
Since ${\ell}_\eps \in S_N$, we get
\begin{align}\label{eq:ldp-2}
&\int_{\mathbb{T}^d}\int_{\R}f_\eps^{\pm}(x,t,\xi)\bar{g}_\eps^{\pm}(x,t,\xi)
\,d\xi\,dx  \\ &= \int_{(\mathbb{T}^d))^2}\int_{\R^2} \varrho_\delta(x-y)\rho_\lambda(\xi-\zeta)f_\eps^{\pm}(x,t,\xi)\bar{g}_\eps^{\pm}(y,t,\zeta)
\,d\xi\,d\zeta\,dx\,dy\notag \\
&\, - \Upsilon_t(\delta, \lambda)\le C\big(\Gamma(f_{\eps, 0},g_{\eps, 0}, \delta, \lambda) + |\Upsilon_0(\delta, \lambda)|\big) + |\mathcal{M}_6|(t) + C\big(\underset{0 \le s \le t}{\text{sup}}\,|\mathcal{M}_6|(s)\big), \notag
\end{align}
where $\Upsilon_t(\delta, \lambda)$ is defined in \eqref{eq:Upsilon}. Following the proof of \cite[Theorem 6.1]{Zhang-2020}, we get
\begin{align}
    \underset{0 \le t \le T}{\text{sup}}\,\Upsilon_t(\delta, \lambda) \le C(\lambda + \delta + \lambda\delta^{-1} + \Upsilon_0(\delta, \lambda)), \quad \mathbb{E}\big[\underset{0 \le t \le T}{\text{sup}}\,|\mathcal{M}_6|(t)\big] \le C\eps^{\frac{1}{2(1+d)}},\label{eq:M-6}
\end{align}
where $\Upsilon_0(\delta, \lambda) \rightarrow 0$ as $\delta, \lambda \rightarrow 0$.  Taking expectation and using \eqref{eq:M-6} and \eqref{inq:f1barf2}, we deduce from \eqref{eq:ldp-2}
\begin{align}
&\mathbb{E}\Big[\|(\Bar{u}^{\eps}(t) - v^{\eps}(t))^+\|_{\mathcal{E}_1}\Big] \le T \mathbb{E}\Big[\underset{t \in [0,T]}{\text{ess sup}}\|(\Bar{u}^{\eps}(t) - v^{\eps}(t))^+\|_{L^1(\mathbb{T}^d)}\Big]\notag  \\ & \le 
 C(\delta^{-1}\lambda +
r^{-2\theta}\lambda +
r^{2-2\theta}\delta^{-2} + \delta +  \eps\lambda^{-1} + 
\lambda + |\Upsilon_0(\delta, \lambda)| + \eps^{\frac{1}{2(1+d)}}\big)\,.\label{inq:ldp-finla-2}
\end{align}
Setting $\lambda = \eps^{\frac{1}{2}}$, $\delta = \eps^b$ and $r = \eps^c$ for some $c \in (\frac{b}{1-\theta}, \frac{1}{4\theta})$ with $0 < b < \text{min}\{\frac{1}{2}, \frac{1-\theta}{4\theta}\}$,
and sending $\eps \rightarrow 0$ in \eqref{inq:ldp-finla-2} we conclude $\|(\Bar{u}^{\eps}(t) - v^{\eps}(t))^+\|_{\mathcal{E}_1} \rightarrow 0$ in probability. Similarly one can infer $\|(\Bar{u}^{\eps}(t) - v^{\eps}(t))^-\|_{\mathcal{E}_1} \rightarrow 0$ in probability, which essentially establish \eqref{conv:final-i0-LDP}.
\section{Central limit theorem}\label{sec:CLT}
In this section, we study central limit theorem. 
As discussed earlier, one cannot directly prove \eqref{clt:main-result} by doubling variables method although the kinetic formulation for $\frac{1}{\sqrt{\eps}}(u^\eps - \hat{u})$ and $u^\star$ are available. To overcome this technical difficulty, we introduce some  auxiliary PDE/SPDEs. Consider the following viscous approximations of \eqref{eq:deterministic-FCL}: for $\eta >0$,
\begin{equation}\label{eq:viscous-deterministic-FCL}
     \displaystyle d\hat{u}^\eta + \text{div}F(\hat{u}^\eta(t,x))\,dt + (-\Delta)^\theta[{\tt \Phi}(\hat{u}^\eta(t,\cdot))](x)\,dt = \eta\Delta\hat{u}^\eta\,dt, \quad 
    \hat{u}^\eta(0,x) = 1.
\end{equation}
Observe that, $ \hat{u}^\eta = 1$ is the unique strong solution of \eqref{eq:viscous-deterministic-FCL}. As a direct consequence of vanishing viscosity method, we infer that \eqref{eq:deterministic-FCL} has a unique kinetic solution $\hat{u}= 1$. Let $u^{\eps, \eta}$ be the unique weak solution of 
\begin{equation}\label{eq:viscous-small-noise-scl}
\begin{cases}
     \displaystyle du^{\eps, \eta} + \text{div}F(u^{\eps, \eta}(t,x))\,dt + (-\Delta)^\theta[{\tt \Phi}(u^{\eps, \eta}(t,\cdot))](x)\,dt \\ = \eta\Delta u^{\eps, \eta}\,dt + \sqrt{\eps}{\tt h}(u^{\eps, \eta}(t,x))dW(t),\quad
    u^{\eps, \eta}(0,x) = 1,
\end{cases}
\end{equation}
---which exists due to \cite[Section 8]{AC_2}. 
Then, $u^{\eps, \eta}$ satisfies the following  uniform estimate~(cf.~\cite[Propostion 3.6]{AC_2}): for $p\in [2, \infty)$  
$$\mathbb{E}\big[\underset{0 \le t \le T}{\text{sup}}\|u^{\eps, \eta}\|_{L^p(\mathbb{T}^d)}^p\big] \le C_p\,.$$
 Consider a viscous approximations of \eqref{eq:star-FSCL}: for $(x,t)\in \mathbb{T}^d \times [0,T]$
\begin{equation}\label{eq:viscous-star-FSCL}\begin{cases}
     \displaystyle du^{\star,\eta} + \text{div}(F^\prime(\hat{u})u^{\star,\eta})\,dt + (-\Delta)^\theta[{\tt \Phi}^\prime(\hat{u})u^{\star, \eta}(t,\cdot)](x)\,dt \\ = \eta\Delta u^{\star,\eta}\,dt  + {\tt h}(\hat{u})\,dW(t), \quad
    u^{\star, \eta}(0,x) = 0\,.
\end{cases}
\end{equation}
In view of \cite[Theorem 2.1]{frac non} (see also \cite[Section 8]{AC_2}), SPDE \eqref{eq:viscous-star-FSCL} admits a unique weak solution $u^{\star,\eta}$.
Moreover, thanks to \cite[Theorem 3.5]{AC_2}, one has 
\begin{align}\label{eq:strong-cong-star-eta}
    \underset{\eta \rightarrow 0}{\text{lim}}\,\mathbb{E}\big[\|u^{\star,\eta} - u^{\star}\|_{\mathcal{E}_1}\big] = 0.
\end{align}
In view of \eqref{eq:strong-cong-star-eta}, and the triangle inequality
\begin{align}\label{inq:triangle-CLT}
\mathbb{E}\Big[\big\|\frac{u^\eps - \hat{u}}{\sqrt{\eps}} - u^\star\big\|_{\mathcal{E}_1}\Big]  &\le 
\mathbb{E}\Big[\big\|\frac{u^\eps - \hat{u}}{\sqrt{\eps}} - \frac{u^{\eps, \eta} - \hat{u}^\eta}{\sqrt{\eps}}\big\|_{\mathcal{E}_1}\Big] \\&+ \mathbb{E}\Big[\big\| \frac{u^{\eps, \eta} - \hat{u}^\eta}{\sqrt{\eps}} - u^{\star, \eta}\big\|_{\mathcal{E}_1}\Big]+ \mathbb{E}\Big[\big\| u^{\star, \eta} - u^{\star}\big\|_{\mathcal{E}_1}\Big]\,,\notag
\end{align}
to prove Theorem \ref{clt:main-result}, we only need to show that the first and second terms of \eqref{inq:triangle-CLT} go to zero as $\eps \rightarrow 0$. 
\subsection{\bf A-priori estimates} In order to establish the convergence result for first two terms of \eqref{inq:triangle-CLT}, we need some a-priori estimates. 
\begin{lem}\label{lem:1-CLT}
For any $\eps  \in (0,1)$  and $p \ge 2$, there exists a constant $C>0$, independent of $\eps$ and $\eta$, such that 
\begin{align}
      \underset{t \in [0,T]}{sup}\,\underset{\eta > 0}{sup}\,\mathbb{E}\Big[\|u^{\eps, \eta}(t) - \hat{u}^\eta\|_{L^p(\mathbb{T}^d)}^{p} \Big]  \le C\eps^{\frac{p}{2}}. \label{eq:w-eps-eta-bound-by-eps}
\end{align}
\end{lem}
\begin{proof}
    Let $v^{\eps, \eta} := u^{\eps, \eta} - \hat{u}^\eta$. Since $\hat{u}^\eta = 1$, it follows from \eqref{eq:viscous-deterministic-FCL} and \eqref{eq:viscous-small-noise-scl} that $v^{\eps, \eta}$ satisfies 
\begin{align*}
    &dv^{\eps, \eta} + \text{div}F(v^{\eps, \eta} + 1)  + (-\Delta)^\theta[{\tt \Phi}(v^{\eps, \eta}(t,\cdot) + 1 )](x)\,dt \\ & \hspace{4cm}= \eta\Delta v^{\eps, \eta}\,dt + \sqrt{\eps}{\tt h}(v^{\eps, \eta} + \hat{u}^\eta)dW(t), \quad v^{\eps, \eta}(0)= 0\,.
\end{align*}
Applying generalized It\^o formula \cite[Appendix A]{Hofmanova-2016} to $\varphi_n(v^{\eps, \eta})$ (defined in \eqref{eq:smooth-approx-kinetic}), we get
\begin{align}
   & \int_{\mathbb{T}^d}\varphi_n(v^{\eps, \eta}(t))\,dx = -\int_0^t \int_{\mathbb{T}^d}\varphi_n^\prime(v^{\eps, \eta})\text{div}F(v^{\eps, \eta} + 1) \,dx\,ds \notag \\ &- \int_{\mathbb{T}^d}\varphi_n^\prime(v^{\eps, \eta})(-\Delta)^\theta[{\tt \Phi}(v^{\eps, \eta}(s,\cdot)+1)](x)\,dx\,ds + \frac{\eps}{2} \int_0^t \int_{\mathbb{T}^d}\varphi_n^{\prime\prime}(v^{\eps, \eta})G^2(v^{\eps, \eta} + 1)\,dx\,ds\notag \notag\\&
    +  \int_0^t \int_{\mathbb{T}^d} \sum_{k=1}^\infty\varphi_n^\prime(v^{\eps, \eta}){\tt h}_k(x, v^{\eps, \eta} +1)\,dx\,d\beta_k(s) + \eta\int_0^t \int_{\mathbb{T}^d}\varphi_n^\prime(v^{\eps, \eta})\Delta v^{\eps, \eta}\,dx\,ds
    := \sum_{i =1}^5 \mathcal{B}_i\,. \notag
\end{align}
Similar to \textbf{Step II} of Section \ref{sec:existance-skeleton}, we have $\mathcal{B}_1=0$ and $\mathcal{B}_2, \mathcal{B}_3 \le 0$.
Using the assumption \ref{A3} and \eqref{bounds-for-varphi}, one has 
\begin{align*}
 \mathcal{B}_3 \le  C\eps^{\frac{p}{2}} + C \int_0^t\int_{\mathbb{T}^d}(\varphi_n(v^{\eps, \eta}))\,dx\,ds.
\end{align*}
Clearly, $\mathbb{E}[{\mathcal{B}_4}]= 0$. Taking expectation and then applying Gronwall's lemma, one arrive at \eqref{eq:w-eps-eta-bound-by-eps} after sending $n \rightarrow \infty$ in the resulting inequality. 
\end{proof}
\noindent In view of Lemma \ref{lem:1-CLT} and  \cite[Section 3.2]{AC_2}, we infer the following corollary.
\begin{cor}\label{cor:1-cor}
For $p \ge 2$, there exists a constant C, independent of $\eps$ and $\eta$, such that 
\begin{align}
    \mathbb{E}\big[\|u^\eps - \hat{u}\|_{L^p([0,T]; L^p(\mathbb{T}^d))}^p\big] = \underset{\eta \rightarrow 0}{\text{lim}}\,\mathbb{E}\big[\|u^{\eps, \eta} - \hat{u}^\eta\|_{L^p([0,T]; L^p(\mathbb{T}^d))}^p\big] \le C\eps^\frac{p}{2}.
\end{align}
\end{cor}
Note that $w^\eps := \frac{u^\eps - \hat{u}}{\sqrt{\eps}}$ and $w^{\eps,\eta} := \frac{u^{\eps, \eta} - \hat{u}^\eta}{\sqrt{\eps}}$ satisfy the following SPDEs: 
\begin{align}
& dw^\eps + \frac{1}{\sqrt{\eps}}\text{div}\big(F(\sqrt{\eps}w^\eps +1) - F(1)\big)\,dt\notag \\ & + \frac{1}{\sqrt{\eps}}(-\Delta)^\theta[{\tt \Phi}((\sqrt{\eps}w^\eps +1)(t,\cdot)) - {\tt \Phi}(1)](x)\,dt = {\tt h}(\sqrt{\eps}w^\eps +1)dW(t)\,,\notag\\
\label{eq:w-eps-eta}
&dw^{\eps, \eta} + \frac{1}{\sqrt{\eps}}\text{div}\big(F(\sqrt{\eps}w^{\eps, \eta} +1) - F(1)\big)\,dt - \eta \Delta w^{\eps,\eta}\,dt \\ & + \frac{1}{\sqrt{\eps}}(-\Delta)^\theta[{\tt \Phi}((\sqrt{\eps}w^{\eps, \eta} +1)(t,\cdot)) - {\tt \Phi}(1)](x)\,dt =  {\tt h}(\sqrt{\eps}w^{\eps, \eta} +1)dW(t)\notag
\end{align}
with $w^\eps(0,x)=0=w^{\eps, \eta}$. 
Applying It\^o  formula to the function $f(u) = \|u\|_{L^2(\mathbb{T}^d)}^p$, we get from \eqref{eq:w-eps-eta} $\mathcal{P}$-a.s.
\begin{align}
     &\|w^{\eps, \eta}(t)\|_{L^2(\mathbb{T}^d)}^p = \,-\frac{p}{\sqrt{\eps}}\int_0^t\|w^{\eps, \eta}\|_{L^2(\mathbb{T}^d)}^{p-2}\big\langle w^{\eps, \eta}, \text{div}\big(F(\sqrt{\eps}w^{\eps, \eta} +1) - F(1)\big)\big\rangle\,dr \\ &\,-\frac{p}{\sqrt{\eps}}\int_0^t\|w^{\eps, \eta}\|_{L^2(\mathbb{T}^d)}^{p-2}\big\langle w^{\eps, \eta},(-\Delta)^\theta[{\tt \Phi}((\sqrt{\eps}w^{\eps, \eta} +1)(r,\cdot)) - {\tt \Phi}(1)](\cdot)\big\rangle\,dr \notag \\ & \,+ p\,\eta\int_0^t\|w^{\eps, \eta}\|_{L^2(\mathbb{T}^d)}^{p-2}\langle w^{\eps, \eta} ,\Delta w^{\eps,\eta}\rangle\,dr     \notag \\
     &\, + \frac{p(p-1)}{2}\sum_{k \ge 1}\int_0^t \|w^{\eps, \eta}\|_{L^2(\mathbb{T}^d)}^{p-4}\big\langle w^{\eps, \eta} ,{\tt h}_k(\sqrt{\eps}w^{\eps, \eta} +1)\big\rangle^2\,dr  \notag \\
      & \, + p\int_0^t\|w^{\eps, \eta}\|_{L^2(\mathbb{T}^d)}^{p-2}\big\langle w^{\eps, \eta}, {\tt h}_k(\sqrt{\eps}w^{\eps, \eta} +1)\big\rangle \,d\beta_k(r)
     =: \sum_{i =1}^{5}\mathcal{D}_i. \notag
\end{align}
Note that $\mathcal{D}_1=0$. Thanks to \eqref{esti:nondecreasing-phi}, we estimate $\mathcal{D}_2$ as
$$\mathcal{D}_2 \le -\frac{p}{\eps\|{\tt \Phi}'\|_{L^\infty}}\int_0^t\|w^{\eps, \eta}\|_{L^p(\mathbb{T}^d)}^{p-2}\big[{\tt \Phi}(\sqrt{\eps}w^{\eps, \eta} + 1)\big]_{H^\theta(\mathbb{T}^d)}^2\,dr\,. $$
$\mathcal{D}_3$ and $\mathcal{D}_4$ are estimated by
\begin{align*}
\mathcal{D}_3 = -\,p\eta\int_0^t\|w^{\eps, \eta}\|_{L^2(\mathbb{T}^d)}^{p-2}\|\nabla w^{\eps,\eta}\|_{L^2(\mathbb{T}^d)}^2\,dr, \quad \mathcal{D}_4 \le C\Big(1 + \int_0^t\|w^{\eps, \eta}\|_{L^2(\mathbb{T}^d)}^{p}\,dr\Big).
\end{align*}
Since $\mathbb{E}[\mathcal{D}_5] = 0$, using above estimations and Gronwall's lemma, we get
\begin{align}\label{eq:bound-w-eps-eta-gronwall}
    &\mathbb{E}\big[\|w^{\eps, \eta}(t)\|_{L^2(\mathbb{T}^d)}^p \big]  +\,p\eta\mathbb{E}\Big[ \int_0^t\|w^{\eps, \eta}\|_{L^2(\mathbb{T}^d)}^{p-2}\|\nabla w^{\eps,\eta}\|_{L^2(\mathbb{T}^d)}^2\,dr\Big]\\ &\,+ \frac{p}{\eps\|{\tt \Phi}'\|_{L^\infty}}\mathbb{E}\Big[ \int_0^t\|w^{\eps, \eta}\|_{L^2(\mathbb{T}^d)}^{p-2}\big[\|{\tt \Phi}(\sqrt{\eps}w^{\eps, \eta}+1)\|\big]_{H^\theta(\mathbb{T}^d)}^2\Big]\,dr \le C.\notag
\end{align}
\begin{rem}\label{rem:about-phi-theta}
Since ${\tt \Phi}$ is Lipschitz continuous, in view of \eqref{eq:bound-w-eps-eta-gronwall} for $p = 2$, we have
    $$\eta\mathbb{E}\Big[ \int_0^t\|\nabla w^{\eps,\eta}\|_{L^2(\mathbb{T}^d)}^2\,dr\Big] + \frac{1}{\eps}\mathbb{E}\Big[ \int_0^t\|{\tt \Phi}(\sqrt{\eps}w^{\eps, \eta}+1)\|_{H^\theta(\mathbb{T}^d)}^2\Big]\,dr\le C.$$
\end{rem}
\noindent With Lemma \ref{lem:1-CLT} and Corollary \ref{cor:1-cor} in hand, we have the following results.
\begin{lem}\label{lem:bound-w-eps-eta-1} For any $p \ge 2$, there exists a constant C, independent of $\eps$ and $\eta$, such that 
\begin{align}\label{lem:bound-w-eps-eta}
    &\underset{\eps \in (0,1)}{\text{sup}}\, \mathbb{E}\big[\|w^\eps\|_{L^p([0,T]; L^p(\mathbb{T}^d))}^p\big] \le C\,, \quad 
    \underset{\eps \in (0,1)}{\text{sup}}\, \mathbb{E}\big[\|w^{\eps, \eta}\|_{L^p([0,T]; L^p(\mathbb{T}^d))}^p\big] \le C\,.
\end{align}
\end{lem}
\begin{lem}\label{lem:compactness-w-eps-eta}
The following estimation holds:  for $\beta \in (0, \frac{1}{2})$,
    $$\underset{\eps > 0}{\text{sup}}\,\mathbb{E}\big[\|w^{\eps, \eta}\|_{W^{\beta,2}([0,T]; H^{-1}(\mathbb{T}^d))}\big] < \infty.$$
\end{lem}
\begin{proof}
From \eqref{eq:w-eps-eta}, we have
\begin{align*}
    &w^{\eps, \eta}(t) =  - \frac{1}{\sqrt{\eps}}\int_0^t\text{div}\big(F(\sqrt{\eps}w^{\eps, \eta} +1) - F(1)\big)\,dr + \eta \int_0^t\Delta w^{\eps,\eta}\,dr \notag \\ &-\int_0^t\frac{1}{\sqrt{\eps}}(-\Delta)^\theta[{\tt \Phi}(\sqrt{\eps}w^{\eps, \eta} +1) - {\tt \Phi}(1)](x)\,dr + \int_0^t{\tt h}(\sqrt{\eps}w^{\eps, \eta} +1)dW(r)  =: \sum_{i=1}^4\mathcal{H}_i.
\end{align*}
In view of \eqref{eq:bound-w-eps-eta-gronwall} and Jensen's inequality, one can easily see that for $\beta \in (0,\frac{1}{2})$,
\begin{align}\label{eq:bound-for-H1-H3}
 \mathbb{E}\big[\|\mathcal{H}_1\|_{W^{\beta, 2}([0,T];H^{-1}(\mathbb{T}^d)}\big] \le C, \quad 
 \mathbb{E}\big[\|\mathcal{H}_2\|_{W^{\beta, 2}([0,T];H^{-1}(\mathbb{T}^d)}\big]   \le  C\,.
\end{align}
Since $
    \|(-\Delta)^\theta[{\tt \Phi}((\sqrt{\eps}w^{\eps, \eta} +1)(r,\cdot))](\cdot)\|_{H^{-1}(\mathbb{T}^d)}  \le C\|{\tt \Phi}((\sqrt{\eps}w^{\eps, \eta} +1)(r,\cdot))\|_{H^\theta(\mathbb{T}^d)}$, by using  Remark \ref{rem:about-phi-theta} and Jensen's inequality, we have 
\begin{align}
  &\mathbb{E}\big[\|\mathcal{H}_3(t)- \mathcal{H}_3(s)\|_{H^{-1}(\mathbb{T}^d)}^2\big]  \le \frac{C(t-s)}{\eps}\mathbb{E}\int_0^t\|{\tt \Phi}(\sqrt{\eps}w^{\eps, \eta} +1)\|_{H^\theta(\mathbb{T}^d)}^2\,dr \le C(t-s), \notag
\end{align}
which then implies, for $\beta \in (0, \frac{1}{2})$,
\begin{align}
\mathbb{E}\big[\|\mathcal{H}_3\|_{W^{\beta, 2}([0,T];H^{-1}(\mathbb{T}^d)}\big] \le C. \label{eq:bound-for-H2}
\end{align}
Thanks to It\^o isometry, the assumption \ref{A3}, the continuous embedding of $L^2(\mathbb{T}^d)$ in $H^{-1}(\mathbb{T}^d)$ together with \eqref{eq:bound-w-eps-eta-gronwall}, there exists a constant $C$ such that for $\beta \in (0, \frac{1}{2})$, 
\begin{align}
\,\mathbb{E}\big[\|\mathcal{H}_4\|_{W^{\beta, 2}([0,T];H^{-1}(\mathbb{T}^d)}\big]   \le  C.  \label{eq:bound-for-H4}
\end{align}
Combining \eqref{eq:bound-for-H1-H3}, \eqref{eq:bound-for-H2} and \eqref{eq:bound-for-H4}, we get the desired result.
\end{proof}
\subsection{\bf Proof of Theorem \ref{thm:clt}}
In view of Lemma \ref{lem:bound-w-eps-eta} and the doubling of variables technique, we show the convergence of first term on the right hand side of \eqref{inq:triangle-CLT}.
\begin{prop}\label{prop:double-varibale-1st-term-clt}
It holds that, 
\begin{align}
   \underset{\eta \rightarrow 0}{\text{lim}}\, \underset{\eps \in (0,1)}{\text{sup}} \mathbb{E}
\Big[\Big\|\frac{u^\eps - \hat{u}}{\sqrt{\eps}} -\frac{u^{\eps, \eta} - \hat{u}^\eta}{\sqrt{\eps}}\Big\|_{\mathcal{E}_1}\Big] = 0\,. \label{eq:result-prop-clt-1}
\end{align}
\end{prop}
\begin{proof}
Let us denote ${\tt f}_1(x,t,\xi) :={\bf 1}_{w^\eps(x,t) > \xi}$ and ${\tt f}_2(x,t,\zeta) :={\bf 1}_{w^{\eps, \eta}(x,t) > \zeta}$ with corresponding kinetic measure ${\tt m}^{\eps}$ and ${\tt m}^{\eps, \eta}$ and initial value ${\tt f}_{1,0} ={\bf 1}_{0 > \xi}$ and ${\tt f}_{2,0} ={\bf 1}_{0 > \zeta}$. For $\psi(x,\xi,y,\zeta)=\varrho_\delta(x-y)\rho_\lambda(\xi -\zeta)$, we have
\begin{align}\label{eq:w-eps-1}
   &\mathbb{E}\langle\langle {\tt f}_1^{\pm}(t)\Bar{\tt f}_2^{\pm}(t), \psi \rangle \rangle\, =\, \langle\langle {\tt f}_{1,0}\Bar{\tt f}_{2,0}, \psi \rangle \rangle\\& + \mathbb{E}\int_0^t\int_{(\mathbb{T}^d)^2}\int_{\R^2} {\tt f}_1\Bar{\tt f}_2\big(F'(\sqrt{\eps}\xi +1)\cdot \nabla_x + F'(\sqrt{\eps}\zeta +1)\cdot \nabla_y\big)\psi \,d\xi\,d\zeta\,dx\,dy\,ds\notag\\
   &- \mathbb{E}\int_0^t\int_{(\mathbb{T}^d)^2}\int_{\R^2}{\tt f}_1\Bar{\tt f}_2\Big( {\tt \Phi}'(\sqrt{\eps}\xi +1)(-\Delta)^\theta[\psi](x)\notag \\ & \hspace{4cm} + {\tt \Phi}'((\sqrt{\eps}\zeta +1)(-\Delta)^\theta[\psi](y)\Big)\,d\xi\,d\zeta\,dx\,dy\,ds\notag \\
   & + \frac{1}{2}\mathbb{E}\int_0^t\int_{(\mathbb{T}^d)^2}\int_{\R^2}\partial_{\xi}\psi \Bar{\tt f}_2(s,y,\zeta)G^2(x,\sqrt{\eps}\xi +1)\,d\mathcal{V}_{x,s}^{\eps}(\xi)\,d\zeta\,dx\,dy\,ds \notag \\
   & - \frac{1}{2}\mathbb{E}\int_0^t\int_{(\mathbb{T}^d)^2}\int_{\R^2}\partial_{\zeta}\psi {\tt f}_1(s,x,\xi)G^2(y,\sqrt{\eps}\zeta +1)\,d\mathcal{V}_{y,s}^{\eps, \eta}(\zeta)\,d\xi\,dx\,dy\,ds \notag \\
& - \mathbb{E}\int_0^t\int_{(\mathbb{T}^d)^2}\int_{\R^2}G_{1,2}(x,y, \sqrt{\eps}\xi +1, \sqrt{\eps}\zeta +1)\psi\,d\mathcal{V}_{x,s}^{\eps} \otimes d\mathcal{V}_{x,s}^{\eps, \eta}(\xi, \zeta)\,dx\,dy\,ds\notag \\
   &-\mathbb{E}\int_0^t\int_{(\mathbb{T}^d)^2}\int_{\R^2}\Bar{\tt f}_2^+(s,y,\zeta)\partial_\xi\psi\,d {\tt m}^\eps(x,\xi,s)\,d\zeta\,dy \notag \\ &+ \mathbb{E}\int_0^t\int_{(\mathbb{T}^d)^2}\int_{\R^2}{\tt f}_1^-(s,x,\xi)\partial_\zeta\psi\,d{\tt m}^{\eps ,\eta}(y, \zeta, s)\,d\xi\,dx,\notag \\ & + \eta\mathbb{E}\int_0^t\int_{(\mathbb{T}^d)^2}\int_{\R^2}{\tt f}_1\Bar{\tt f}_2\nabla_y \psi\,d\xi\,d\zeta\,dx\,dy\,ds 
   =: \langle\langle {\tt f}_{1,0}\Bar{\tt f}_{2,0}, \psi \rangle \rangle + \sum_{i=1}^{8}{\mathcal{C}_i},\notag
\end{align}
where  $G_{1,2}(x,y,\xi,\zeta):= \sum_{k \ge 1}{\tt h}_k(x,\xi){\tt h}_k(y, \zeta)$, $\mathcal{V}_{x,s}^{\eps}(\xi):= -\partial_\xi {\tt f}_1^{+}(s,x,\xi)$, $\mathcal{V}_{y,s}^{\eps,\eta}(\zeta) := \partial_\zeta {\tt f}_2^{+}(s,y,\zeta)$,\, ${\tt m}^\eps(x,t,\xi) \ge \Bar{\eta}^\eps(x,t,\xi)$,\, ${\tt m}^{\eps, \eta}(y,t,\zeta)\ge 
  \Bar{\eta}^{\eps, \eta}(y,t,\zeta)$ with
\begin{align*}
  &\Bar{\eta}^\eps:= \frac{1}{\sqrt{\eps}}\int_{\R^d}\big|{\tt \Phi}(\sqrt{\eps}w^\eps(t, x+z) +1) - {\tt \Phi}(\sqrt{\eps}\xi +1)\big|{\bf 1}_{{\rm Conv}\{w^\eps(t,x), w^\eps(t,x+z)\}}(\xi)\gamma(z)\,dz\,, \notag
\end{align*}
and $\Bar{\eta}^{\eps, \eta}(y,t,\zeta)$ defined similarly as $\Bar{\eta}^\eps$ i.e. $w^{\eps}$ replaced with $w^{\eps, \eta}$.
Observe that, 
\begin{align*}
\sum_{i=6}^7\mathcal{C}_i \le \, &\mathbb{E}\Big[\int_0^t\int_{(\mathbb{T}^d)^2}\int_{\R^2}\Bar{\tt f}_2\partial_\zeta\psi\,d\Bar{\eta}^\eps\,dy\,d\zeta - \int_0^t\int_{(\mathbb{T}^d)^2}\int_{\R^2}{\tt f}_1\partial_\xi\psi \,d\Bar{\eta}^{\eps, \eta}\,dx\,d\xi\Big]:=\sum_{i=1}^2\mathcal{J}_{6,7,i}\, \\
& \mathcal{C}_8 \le \eta\,\mathbb{E}\int_0^t\int_{(\mathbb{T}^d)^2}\int_{\R^2}{\tt f}_1\Bar{\tt f}_2 |\Delta_y\varrho_\delta(x-y)|\rho_\lambda(\xi-\zeta)\,d\xi\,d\zeta\,dx\,dy\,ds \le C\eta\delta^{-2}.
\end{align*}
In view of the assumption \ref{A1}, we have
\begin{align*}
    \mathcal{C}_1 &\le \sqrt{\eps}\|F''\|_{L^\infty}\mathbb{E}\int_0^t\int_{(\mathbb{T}^d)^2}\int_{\R^2}{\tt f}_1\Bar{\tt f}_2 |\xi-\zeta||\nabla_x \varrho_\delta(x-y)|\rho_\lambda(\xi-\zeta)\,d\xi\,d\zeta\,dx\,dy\,ds \notag \\
    & \le \sqrt{\eps}C\mathbb{E}\int_0^t\int_{(\mathbb{T}^d)^2}|\varrho_\delta(x-y)|\int_{\R^2}h(\xi,\zeta)\, d\mathcal{V}_{x,s}^\eps \otimes \mathcal{V}_{y,s}^{\eps, \eta}(\xi,\zeta)\,dx\,dy\,ds,
\end{align*}
where $h(\xi,\zeta):= \int_\zeta^\infty\int_{-\infty}^\xi |\xi'-\zeta'|\rho_\lambda(\xi-\zeta)\,d\xi'\,d\zeta'.$ Since $h(\xi-\zeta) \le C\lambda$, we get
$$\mathcal{C}_1  \le C\sqrt{\eps}\lambda\delta^{-1}.$$
Thanks to \ref{A3}, we get
\begin{align*}
\sum_{i=3}^5 \mathcal{C}_i & \le C \mathbb{E}\int_0^t\int_{(\mathbb{T}^d)^2}\int_{\R^2}\varrho_\delta\rho_\lambda\big(|x-y|^2 + |\sqrt{\eps}(\xi-\zeta)|^2\big)d\mathcal{V}_{x,s}^\eps \otimes d\mathcal{V}_{y,s}^{\eps, \eta}(\xi,\zeta)\,dx\,dy\,ds \notag \\ & \le C(\delta^2 + \eps\,\lambda).
\end{align*}
Following \cite{AC_2},  we re-write the term $\mathcal{J}_{6,7,1}$ as
\begin{align*}
 \mathcal{J}_{6,7,1}= & \,\frac{1}{\sqrt{\eps}}\mathbb{E}\int_0^t\int_{(\mathbb{T}^d)^2}\int_{\R^{d+1}}({\tt \Phi}(\sqrt{\eps}w^\eps(s, x+z)+1) - {\tt \Phi}(\sqrt{\eps}w^\eps(s,x) +1))\notag \\
 &  \hspace{3cm}\times \rho_\lambda(w^\eps(s,x) - \zeta)\varrho_\delta(x-y)\Bar{\tt f}_2(y,s,\zeta)\gamma(z)\,dz\,d\zeta\,dx\,dy\,ds \notag \\
   + &\mathbb{E}\int_0^t\int_{(\mathbb{T}^d)^2}\int_{\R^{d+1}}\int_{\R}\varrho_\delta(x-y)\Bar{\tt f}_2(y,s,\zeta)\bigg(\int_{-\infty}^{w^\eps(s,x+z)}\rho_\lambda(\xi-\zeta){\tt \Phi}'(\sqrt{\eps}\xi +1)\,d\xi \notag \\
 &-\int_{-\infty}^{w^\eps(s,x)}\rho_\lambda(\xi-\zeta){\tt \Phi}'(\sqrt{\eps}\xi +1)\,d\xi\bigg)  \gamma(z)\,dz\,d\zeta\,dx\,dy\,ds =: \Bar{J}_{6,7,1} + \hat{J}_{6,7,1}\,. 
\end{align*}
Using the change of variables i.e. $x \mapsto x +z$ and $z \mapsto -z$ in $\hat{J}_{6,7,1}$, we have
\begin{align*}
    \hat{J}_{6,7,1} =  &\,\mathbb{E}\int_0^t\int_{(\mathbb{T}^d)^2}\int_{\R^{d+1}}\int_{\R}\Big(\int_{-\infty}^{w^\eps(s,x)}\rho_\lambda(\xi-\zeta){\tt \Phi}'(\sqrt{\eps}\xi +1)\,d\xi\Big)\notag \\&\hspace{2cm}\times\big(\varrho_\delta(x+z-y)-\varrho_\delta(x-y)\big)\Bar{\tt f}_2(y,s,\zeta)\gamma(z)\,dz\,d\zeta\,dx\,dy\,ds \notag \\ = &\, \mathbb{E}\int_0^t\int_{(\mathbb{T}^d)^2}\int_{\R^2} {\tt f}_1\Bar{\tt f}_2\rho(\xi-\zeta){\tt \Phi}'(\sqrt{\eps}\xi +1)(-\Delta)^\theta[\varrho(\cdot-y)](x)\,d\xi\,d\zeta\,dx\,dy\,ds\,.
\end{align*}
Similarly, we reformulate $\mathcal{J}_{6,7,2}$ as 
\begin{align*}
    \mathcal{J}_{6,7,2} =  & -\,\frac{1}{\sqrt{\eps}}\mathbb{E}\int_0^t\int_{(\mathbb{T}^d)^2}\int_{\R^{d+1}}({\tt \Phi}(\sqrt{\eps}w^{\eps, \eta}(s,y+z)+1) - {\tt \Phi}(\sqrt{\eps}w^{\eps, \eta}(s,y) +1))\notag \\
    & \hspace{2cm} \times\rho_\lambda(\xi - w^{\eps, \eta}(s,y)) \varrho_\delta(x-y){\tt f}_1(x,s,\xi)\gamma(z)\,dz\,d\xi\,dx\,dy\,ds \notag \\
    & + \mathbb{E}\int_0^t\int_{(\mathbb{T}^d)^2}\int_{\R^2}{\tt f}_1\Bar{\tt f}_2\rho(\xi-\zeta){\tt \Phi}'(\sqrt{\eps}\zeta +1)(-\Delta)^\theta[\varrho(x-\cdot)](y)\,d\xi\,d\zeta\,dx\,dy\,ds\,.
\end{align*}
Thus, we have 
\begin{align*}
& \mathcal{C}_2 + \sum_{i=1}^2 \mathcal{J}_{6,7,i} \notag \\
 &=\,\frac{1}{\sqrt{\eps}}\mathbb{E}\int_0^t\int_{(\mathbb{T}^d)^2}\int_{\R^{d}}\Big(\int_{w^{\eps, \eta}(s,y)}^{\infty}\rho_\lambda(w^\eps(s,x) - \zeta)\,d\zeta\Big)({\tt \Phi}(\sqrt{\eps}w^\eps(s, x+z)+1) \notag \\
 & \hspace{4cm}- {\tt \Phi}(\sqrt{\eps}w^\eps(s,x) +1)) \varrho_\delta(x-y)\gamma(z)\,dz\,dx\,dy\,ds \notag \\ & -\,\frac{1}{\sqrt{\eps}}\mathbb{E}\int_0^t\int_{(\mathbb{T}^d)^2}\int_{\R^{d}}\Big(\int^{w^{\eps}(s,x)}_{-\infty}\rho_\lambda(\xi - w^{\eps, \eta}(s,y))\,d\xi\Big)({\tt \Phi}(\sqrt{\eps}w^{\eps, \eta}(s, y+z)+1)\notag \\
 & \hspace{5cm}- {\tt \Phi}(\sqrt{\eps}w^{\eps, \eta}(s,y) +1)) \varrho_\delta(x-y)\gamma(z)\,dz\,dx\,dy\,ds\equiv \mathcal{K}\,. 
\end{align*}
To estimate $\mathcal{K}$, we consider a special function $\beta_\lambda(r):= \lambda\beta(\frac{r}{\lambda})$, where $\beta(r)$ is a $ C^\infty(\R)$ function satisfying  $\beta(0) = 0,~ \beta(r) = \beta(-r)$ and
\begin{align*}
\beta'(r) = - \beta'(r),~\beta_\lambda^{\prime\prime}(r) = \rho_\lambda(r),~~
\beta^\prime(r)=
\begin{cases} 
-1,\quad &\text{if} ~ r\le -1,\\
\in [-1,1], \quad &\text{if}~ |r|<1,\\
+1, \quad &\text{if} ~ r\ge 1.
\end{cases}
\end{align*}
We reformulate $\mathcal{K}$ as sum of two terms $\mathcal{K}_1$ and $\mathcal{K}_2$, where 
\begin{align*}
  \mathcal{K}_1  = &\,\frac{1}{\sqrt{\eps}}\mathbb{E}\int_0^t\int_{(\mathbb{T}^d)^2}\int_{|z| > \bar{r}}\beta'_\lambda(w^{\eps}(s,x)- w^{\eps, \eta}(s,y))\Big(\big({\tt \Phi}(\sqrt{\eps}w^\eps(s, x+z)+1) \notag \\ 
  & -{\tt \Phi}(\sqrt{\eps}w^{\eps, \eta}(s, y+z)+1)\big)- \big({\tt \Phi}(\sqrt{\eps}w^\eps(s,x) +1))- {\tt \Phi}(\sqrt{\eps}w^{\eps, \eta}(s,y) +1)\big)\Big)\notag \\ &\times\varrho_\delta(x-y)\gamma(z)\,dz\,dx\,dy\,ds\,, \notag \\
    \mathcal{K}_2:&=\frac{1}{\sqrt{\eps}}\mathbb{E}\int_0^t\int_{(\mathbb{T}^d)^2}\int_{|z| \le \bar{r}}\beta'_\lambda(w^{\eps}(s,x)- w^{\eps, \eta}(s,y))\Big(\big({\tt \Phi}(\sqrt{\eps}w^\eps(s, x+z)+1)\notag \\
    &-{\tt \Phi}(\sqrt{\eps}w^{\eps, \eta}(s, y+z)+1)\big) - \big({\tt \Phi}(\sqrt{\eps}w^\eps(s,x) +1)- {\tt \Phi}(\sqrt{\eps}w^{\eps, \eta}(s,y) +1)\big)\Big)\notag \\&\times\varrho_\delta(x-y)\gamma(z)\,dz\,dx\,dy\,ds \,.
\end{align*}
Re-formulating $\mathcal{K}_1$, we have 
\begin{align}
    \mathcal{K}_1 = &\,\frac{1}{\sqrt{\eps}}\mathbb{E}\int_0^t\int_{(\mathbb{T}^d)^2}\int_{|z| > \bar{r}}\beta'_\lambda(w^{\eps}(s,x)- w^{\eps, \eta}(s,y))\big({\tt \Phi}(\sqrt{\eps}w^\eps(s, x+z)+1) \notag \\ &-{\tt \Phi}(\sqrt{\eps}w^{\eps, \eta}(s, y+z)+1)\big) - \text{sgn}(w^{\eps}(s,x)- w^{\eps, \eta}(s,y))\notag \\ &\times\big({\tt \Phi}(\sqrt{\eps}w^\eps(s,x) +1))- {\tt \Phi}(\sqrt{\eps}w^{\eps, \eta}(s,y) +1)\big)\varrho_\delta(x-y)\gamma(z)\,dz\,dx\,dy\,ds \notag \\
    +\,&\,\frac{1}{\sqrt{\eps}}\mathbb{E}\int_0^t\int_{(\mathbb{T}^d)^2}\int_{|z| > \bar{r}}\varrho_\delta\big(\text{sgn}(w^{\eps}(s,x)- w^{\eps, \eta}(s,y))-\beta'_\lambda(w^{\eps}(s,x)- w^{\eps, \eta}(s,y))\big)\notag \\ & \times\big({\tt \Phi}(\sqrt{\eps}w^\eps(s,x) +1))- {\tt \Phi}(\sqrt{\eps}w^{\eps, \eta}(s,y) +1)\big)\Big)\gamma(z)\,dz\,dx\,dy\,ds \notag =: \mathcal{K}_{1,1} + \mathcal{K}_{1,2}
\end{align}
Since $|\beta'_\lambda| \le 1$ and ${\tt \Phi}$ is non-decreasing function, it is easy to observe that
\begin{align}
    \mathcal{K}_{1,1} \le &\frac{1}{\sqrt{\eps}}\mathbb{E}\int_0^t\int_{(\mathbb{T}^d)^2}\int_{|z| > \bar{r}}\big|{\tt \Phi}(\sqrt{\eps}w^\eps(s, x+z)+1) -{\tt \Phi}(\sqrt{\eps}w^{\eps, \eta}(s, y+z)+1)\big|\notag \\&-\big|{\tt \Phi}(\sqrt{\eps}w^\eps(s,x) +1))- {\tt \Phi}(\sqrt{\eps}w^{\eps, \eta}(s,y) +1)\big|\varrho_\delta(x-y)\gamma(z)\,dz\,dx\,dy\,ds \le 0\, \notag 
\end{align}
where in the second term we have used the change of variables i.e. $x \mapsto x+z, y \mapsto y+z$ and $z \mapsto -z$. Using the properties of sign function and $\beta'_\lambda$ along with \ref{A2}, we bound the term $\mathcal{K}_{1,2}$.
\begin{align}
    \mathcal{K}_{1,2} \le  &\frac{C}{\sqrt{\eps}}\mathbb{E}\int_0^t\int_{(\mathbb{T}^d)^2}\int_{|z| > \bar{r}}\big|\text{sgn}(w^{\eps}(s,x)- w^{\eps, \eta}(s,y))-\beta'_\lambda(w^{\eps}(s,x)- w^{\eps, \eta}(s,y))\big|\notag \\ &\hspace{4cm}\times \sqrt{\eps}|w^\eps(s,x)- w^{\eps, \eta}(s,y)|\varrho_\delta(x-y)\gamma(z)\,dz\,dx\,dy\,ds
    \notag \\
    \le & C(t)2\lambda\int_{|z| > \bar{r}}\gamma(z)\,dz \le C\bar{r}^{-2\theta}\lambda. \notag
\end{align}
Next we estimate $\mathcal{K}_2$. Note that, since ${\tt \Phi}' \ge 0$ and $\beta''_\lambda \ge 0$, for any $a,b,k \in \R$,
\begin{align}\label{inq:identity-beta-phi}
   & \frac{1}{\sqrt{\eps}}\beta'_\lambda(b-k)[{\tt \Phi}(\sqrt{\eps}a+1)-{\tt \Phi}(\sqrt{\eps}b+1)] \\
   &\le \int_k^a \beta'_\lambda(r-k){\tt \Phi}'(\sqrt{\eps}r +1)\,dr -  \int_k^b \beta'_\lambda(r-k){\tt \Phi}'(\sqrt{\eps}r +1)\,dr.\notag
\end{align}
Using \eqref{inq:identity-beta-phi}, we get
\begin{align*}
    \mathcal{K}_2 \le  &\,\mathbb{E}\int_0^t\int_{(\mathbb{T}^d)^2}\int_{|z| \le \bar{r}}\bigg(\int_{ w^{\eps, \eta}(s,y)}^{w^\eps(s, x+z)}\beta'_\lambda(r- w^{\eps, \eta}(s,y)){\tt \Phi}'(\sqrt{\eps}r +1)\,dr \notag \\ & - \int_{ w^{\eps,\eta}(s,y)}^{w^\eps(s, x)}\beta'_\lambda(r- w^{\eps, \eta}(s,y)){\tt \Phi}'(\sqrt{\eps}r +1)\,dr \bigg)\varrho_\delta(x-y)\gamma(z)\,dz\,dx\,dy\,ds \notag \\
     & +\mathbb{E}\int_0^t\int_{(\mathbb{T}^d)^2}\int_{|z| \le \bar{r}}\bigg(\int_{w^\eps(s, x)}^{ w^{\eps, \eta}(s,y+z)}\beta'_\lambda(r- w^{\eps}(s,y)){\tt \Phi}'(\sqrt{\eps}r +1)\,dr \notag \\ &- \int_{w^\eps(s, x)}^{ w^{\eps,\eta}(s,y)}\beta'_\lambda(r- w^{\eps}(s,y)){\tt \Phi}'(\sqrt{\eps}r +1)\,dr \bigg)\varrho_\delta\gamma(z)\,dz\,dx\,dy\,ds \notag =: \mathcal{K}_{2,1} + \mathcal{K}_{2,2}\,.
\end{align*}
An application of Lemma \ref{lem:bound-w-eps-eta}, the boundedness of $\beta^\prime_\lambda(\cdot)$, \ref{A2} and Taylor's expansion yields 
\begin{align*}
\mathcal{K}_{2,1} \le C\frac{\Bar{r}^{2-2\theta}}{\delta^2},\quad \mathcal{K}_{2,2} \le C\frac{\Bar{r}^{2-2\theta}}{\delta^2}\,.
\end{align*}
Combining all these estimations in \eqref{eq:w-eps-1}, we finally have
\begin{align}
&\mathbb{E}\int_0^T\int_{(\mathbb{T}^d)}\int_{\R} {\tt f}_1^{+}(x,t,\xi)\bar{\tt f}_2^{+}(x,t,\xi)\,d\xi\,dx\,dt 
 \le  \, \int_{\mathbb{T}}\int_{\R}{\tt f}_{1,0}(x,\xi)\bar{\tt f}_{2,0}(x,\xi) \,d\xi\,dx
  + |\mathcal{E}_0(\delta, \lambda)| \notag \\ 
  & + T\int_0^T\mathcal{E}_t(\delta, \lambda)\,dt 
   + C(T)\big(\sqrt{\eps}\lambda\delta^{-1} + \delta^2 + \sqrt{\eps}\lambda + \eta\delta^{-2} + \Bar{r}^{-2\theta}\lambda + \Bar{r}^{2-2\theta}\delta^{-2}\big) \notag
\end{align}
which then implies 
\begin{align}
   & \mathbb{E}\|\big(w^\eps-w^{\eps, \eta}\big)^+\|_{L^1([0,T];L^1(\mathbb{T}^d))}  \le  |\mathcal{E}_0(\delta, \lambda)| + T\int_0^T\mathcal{E}_t(\delta, \lambda)\,dt  \notag \\
    & \hspace{1cm}+C(T)\big(\sqrt{\eps}\lambda\delta^{-1} + \delta^2 + \sqrt{\eps}\lambda + \eta\delta^{-2} + \Bar{r}^{-2\theta}\lambda + \Bar{r}^{2-2\theta}\delta^{-2}\big). \notag
\end{align}
where $\mathcal{E}_t(\delta, \lambda)$ is defined similar to \eqref{eq:Upsilon}.
Choosing $\Bar{r} = \delta^a$ for some $a \in \big(\frac{1}{1-\theta}, \frac{b}{2\theta}\big)$, $\lambda = \delta^b$ for some $b > max\big\{1, \frac{2\theta}{1-\theta}\big\}$ and $\delta = \eta^\frac{1}{4}$, we get
\begin{align}
    &\mathbb{E}\|\big(w^\eps-w^{\eps, \eta}\big)^+\|_{\mathcal{E}_1} \le  |\mathcal{E}_0(\delta, \delta^b)| + T\int_0^T\mathcal{E}_t(\delta, \delta^b)\,dt \notag \\
    & \qquad +
    C(T)\big(\sqrt{\eps}\delta^{b-1} + \delta^2 + \sqrt{\eps}\delta^b + \eta^\frac{1}{2} + \delta^{b - 2\theta} + \delta^{\Bar{a}}\big), \notag
\end{align}
for some $\Bar{a} > 0$. Similarly, working on $\Bar{\tt f}_1^{\pm}{\tt f}_2^{\pm}$, one has 
\begin{align}
    &\mathbb{E}\|\big(w^\eps-w^{\eps, \eta}\big)^-\|_{L^1([0,T];L^1(\mathbb{T}^d))}\notag \\ \le & |\mathcal{E}_0(\delta, \delta^b)| + T\int_0^T\mathcal{E}_t(\delta, \delta^b)\,dt +
    C(T)\big(\sqrt{\eps}\delta^{b-1} + \delta^2 + \sqrt{\eps}\delta^b + \eta + \delta^{b - 2\theta} + \delta^{\Bar{a}}\big), \notag
\end{align}
Combining these inequality and sending $\eta \rightarrow 0$ in the resulting inequality, we obtain \eqref{eq:result-prop-clt-1}
\end{proof}
Our next goal is to estimate the the second term of right hand side of \eqref{inq:triangle-CLT}. In this regard, let us denote by $\nu^{\eps,\eta}$ the joint law of $(w^{\eps,\eta}, W)$ in the space $\mathbb{U}:= L^2([0,T]; L^2(\mathbb{T}^d)) \times C([0,T]; \mathcal{H}_0)$.
\begin{cor}\label{cor:tightness-w-eps-eta}The sequence $\{\nu^{\eps, \eta}\}$ is tight on $(\mathbb{U}, \mathcal{B}(\mathbb{U})),$ where $\mathcal{B}(\mathbb{U})$ is the Borel $\sigma$-algebra of $\mathbb{U}$.
\end{cor}
\begin{proof}Using Theorem \ref{thm:compactness}, we get $L^2([0,T];H^1(\mathbb{T}^d)) \cap W^{\beta, 2}([0,T]; H^{-1}(\mathbb{T}^d))$ is compactly embedded in $L^2([0,T]; L^2(\mathbb{T}^d))$. So for any $M > 0$, the set
$$K_M = \Big\{w^{\eps,\eta} \in L^2([0,T]; L^2(\mathbb{T}^d)) : \|w^{\eps,\eta}\|_{L^2([0,T];H^1)} + \|w^{\eps,\eta}\|_{W^{\beta, 2}([0,T]; H^{-1})} \le M\Big\}$$
is a compact subset of $ L^2([0,T]; L^2(\mathbb{T}^d))$. An application of Markov's inequality along with Lemma \ref{lem:compactness-w-eps-eta} gives
\begin{align*}
    &\underset{M \rightarrow \infty}{\text{lim}}\,\underset{\eps > 0}{\text{sup}}\,\mathbb{P}(w^{\eps, \eta} \notin K_M)\notag \\ &\,= \underset{M \rightarrow \infty}{\text{lim}}\,\underset{\eps > 0}{\text{sup}}\,\mathbb{P}\big(\|w^{\eps,\eta}\|_{L^2([0,T];H^1(\mathbb{T}^d))} + \|w^{\eps,\eta}\|_{W^{\beta, 2}([0,T]; H^{-1}(\mathbb{T}^d))} > M\big) \notag \\&\,\le \underset{M \rightarrow \infty}{\text{lim}}\frac{1}{M}\,\underset{\eps > 0}{\text{sup}}\,\mathbb{E}\Big[\|w^{\eps,\eta}\|_{L^2([0,T];H^1(\mathbb{T}^d))} + \|w^{\eps,\eta}\|_{W^{\beta, 2}([0,T]; H^{-1}(\mathbb{T}^d))}\Big] = 0.
\end{align*}
Thus, $\{\nu^{\eps, \eta}\}$ is tight on $(\mathbb{U}, \mathcal{B}(\mathbb{U})).$
\end{proof}
 We apply Prokhorov compactness theorem to the family of laws $\{\nu^{\eps,\eta}\}_{\eps}$ and the modified version of Skorokhod representation theorem \cite[Theorem C.1]{Potential-2018} to have the following result.
 \begin{lem}
 Passing to a subsequence of $\{\eps\}$, still denotes by $\{\eps\}$, there exist a new probability space $(\Tilde{\Omega}, \Tilde{\mathcal{F}}, \Tilde{\mathbb{P}})$ and random variables $(\Tilde{w}^{\eps, \eta}, \Tilde{W}^\eps)$ and $(\Tilde{w}^\eta, \Tilde{W})$ taking values in $\mathbb{U}$ such that 
 \begin{itemize}
     \item[i)]$\forall \eps > 0,$ $(\Tilde{w}^{\eps, \eta}, \Tilde{W}^\eps)$ and $(w^{\eps, \eta}, W)$   have the same law $\nu^{\eta, \eps}$ on $\mathbb{U}$,
     \item[ii)] $(\Tilde{w}^{\eps, \eta}, \Tilde{W}^\eps) \rightarrow (\Tilde{w}^{ \eta}, \Tilde{W})$ in $\mathbb{U}$ $\,\Tilde{\mathbb{P}}$-a.s., as$\quad\eps \rightarrow 0.$
     \item[iii)] $\Tilde{W}^\eps(\Tilde{\omega}) = \Tilde{W}(\Tilde{\omega}),\,$ $\forall \Tilde{\omega} \in \Tilde{\Omega}$.
 \end{itemize}     
 \end{lem}
\noindent Moreover, let $\{\Tilde{\mathcal{F}}_t\}$ be the filtration generated by $\Tilde{w}^{\eps, \eta}$ and $\Tilde{W}$. Then $\Tilde{W}$ is a $(\Tilde{\mathcal{F}}_t)$-adapted, $\mathcal{H}$-valued cylindrical Wiener process and $\Tilde{w}^{\eps, \eta}$ is a square integrable $(\Tilde{\mathcal{F}}_t)$-predictable stochastic process. Since $(\Tilde{w}^{\eps, \eta}, \Tilde{W}^\eps)$ and $(w^{\eps, \eta}, W)$ have the same law, by an adapted step wise approximations
of It\^o integral, we have that $\Tilde{w}^{\eps, \eta}$ is the solution to equation \eqref{eq:w-eps-eta} on $(\Tilde{\Omega}, \Tilde{\mathcal{F}},  \Tilde{\mathbb{P}}, (\Tilde{\mathcal{F}}_t),\Tilde{W}^\eps)$. Also, $\Tilde{w}^{\eps, \eta}$ satisfies the uniform estimations mentioned in Lemmas \ref{lem:bound-w-eps-eta-1}-\ref{lem:compactness-w-eps-eta} and \eqref{eq:bound-w-eps-eta-gronwall}.
\subsubsection{\textbf{Identification of limit}}\label{sec:Identification-of-limit}
Observe that, since $\tilde{\mathbb{P}}$-a.s., $\tilde{w}^{\eps, \eta} \rightarrow \Tilde{w}^{\eta}$ in $L^2([0,T]\times\mathbb{T}^d)$, by Vitali convergence theorem together with 
Lemmas \ref{lem:bound-w-eps-eta-1},
\begin{align}
\tilde{\mathbb{E}}\Big[\| \Tilde{w}^{\eps,\eta}- \Tilde{w}^{\eta}\|_{L^2([0,T]; L^2(\mathbb{T}^d))}^2\Big] \goto 0\,. \label{conv:1-new}
\end{align}
 Since $\Tilde{w}^{\eps, \eta}$ is a solution of  \eqref{eq:w-eps-eta}, it holds that for any test function $\psi \in C_c^\infty(\mathbb{T}^d)$, $t\in [0,T]$
\begin{align}\label{eq:tilde-w-eps-eta}
     &\langle\Tilde{w}^{\eps, \eta}(t), \psi\rangle = \int_0^t\big\langle F'(\Bar{\theta}\sqrt{\eps}\Tilde{w}^{\eps, \eta} +1)\Tilde{w}^{\eps, \eta}, \nabla \psi \big\rangle\,dr  \\ &- \int_0^t\big\langle  {\tt \Phi}'(\Bar{\theta}\sqrt{\eps}\Tilde{w}^{\eps, \eta} +1)\Tilde{w}^{\eps, \eta}, (-\Delta)^\theta[\psi]\big\rangle\,dr \notag \\ &+ \eta\int_0^t \langle\Tilde{w}^{\eps, \eta}, \Delta\psi\rangle\,dr +  \int_0^t\big\langle{\tt h}(\sqrt{\eps}\Tilde{w}^{\eps,\eta}+1),\psi\big\rangle \,d\tilde{W}(r),\notag  \quad\Tilde{\mathbb{P}}\text{-a.s.,}
\end{align}
 for some $\Bar{\theta} \in (0,1).$ Observe that, in view of  \eqref{conv:1-new}, $\Tilde{\mathbb{E}}\big|\langle\Tilde{w}^{\eps, \eta} -\Tilde{w}^{\eta}, \psi\rangle\big| \rightarrow 0$ and
\begin{align*}
& \eta\,\Tilde{\mathbb{E}}\Big|\int_0^t \langle\Tilde{w}^{\eps, \eta}-\Tilde{w}^{\eta} , \Delta\psi\rangle\,dr\Big|  \le C(\eta, \psi, T) \Big\{\Tilde{\mathbb{E}}\int_0^T \|\Tilde{w}^{\eps, \eta} - \Tilde{w}^{, \eta}\|_{L^2(\mathbb{T}^d)}^2\,dr\Big\}^\frac{1}{2}  \rightarrow 0\,.
\end{align*}
Now, by triangle inequality 
\begin{align*}
    &\Tilde{\mathbb{E}} \Big|\int_0^t\big\langle F'(\Bar{\theta}\sqrt{\eps}\Tilde{w}^{\eps, \eta} +1)\Tilde{w}^{\eps, \eta} - F'(1)\Tilde{w}^{\eta}, \nabla \psi \big\rangle\,dr\Big|\notag \\
    &\le  \Tilde{\mathbb{E}} \Big| \int_0^t\big\langle F'(\Bar{\theta}\sqrt{\eps}\Tilde{w}^{\eps, \eta} +1)\Tilde{w}^{\eps, \eta} - F'(1)\Tilde{w}^{\eps, \eta}, \nabla \psi \big\rangle\,dr\Big| \notag \\ &+ \Tilde{\mathbb{E}}\Big|\int_0^t\big\langle F'(1)\Tilde{w}^{\eps, \eta}-F'(1)\Tilde{w}^{\eta}, \nabla \psi \big\rangle\,dr\Big| =: \Tilde{\mathcal{A}_1} + \Tilde{\mathcal{A}_2}.
\end{align*}
Using the boundedness of $F'$ and $F''$, a-priori estimates of $\Tilde{w}^{\eps, \eta}$ and 
\eqref{conv:1-new}, we obtain 
\begin{align*}
&\Tilde{\mathcal{A}_1}\le C\sqrt{\eps}\|F''\|_{L^\infty}\|\nabla\psi\|_{L^\infty}\underset{0 \le t \le T}{\text{sup}}\Tilde{\mathbb{E}}\big[\|\Tilde{w}^{\eps, \eta}\|_{L^2(\mathbb{T}^d)}^2\big] \le C\sqrt{\eps} \rightarrow 0, \notag \\   
&\Tilde{\mathcal{A}_2}\le C\|F'\|_{L^\infty}\|\nabla\psi\|_{L^\infty} \Big(\Tilde{\mathbb{E}}\int_0^T \|\Tilde{w}^{\eps, \eta} -\Tilde{w}^{\eta}\|_{L^2(\mathbb{T}^d)}^2 \,dr\Big)^\frac{1}{2} \rightarrow 0\,.
\end{align*}
Regarding the fractional term, we have
\begin{align*}
    &\Tilde{\mathbb{E}} \Big|\int_0^t\big\langle  {\tt \Phi}'(\Bar{\theta}\sqrt{\eps}\Tilde{w}^{\eps, \eta} +1)\Tilde{w}^{\eps, \eta} - {\tt \Phi}'(1)\Tilde{w}^{\eps}, (-\Delta)^\theta[\psi]\big\rangle\,dr\Big|\notag \\
    & \le \Tilde{\mathbb{E}} \Big| \int_0^t\big\langle  {\tt \Phi}'(\Bar{\theta}\sqrt{\eps}\Tilde{w}^{\eps, \eta} +1)\Tilde{w}^{\eps, \eta} - {\tt \Phi}'(1)\Tilde{w}^{\eps, \eta}, (-\Delta)^\theta[\psi]\big\rangle\,dr\Big| \notag \\ &\hspace{1cm}+\Tilde{\mathbb{E}} \Big| \int_0^t\big\langle {\tt \Phi}'(1)\Tilde{w}^{\eps, \eta}- {\tt \Phi}'(1)\Tilde{w}^{\eta}, (-\Delta)^\theta[\psi]\big\rangle\,dr\Big| =: \Tilde{\mathcal{B}}_1 + \Tilde{\mathcal{B}}_2.
\end{align*}
By using the boundedness property of ${\tt \Phi}'$ and ${\tt \Phi}''$  along with \eqref{fractionlbound-gamma}, we have
\begin{align*}
 &\Tilde{\mathcal{B}}_1  \le \sqrt{\eps}C\|{\tt \Phi}''\|_{L^\infty}\|(-\Delta)^\theta[\psi]\|_{L^\infty}\Big(\underset{0 \le t \le T}{\text{sup}}\Tilde{\mathbb{E}}\big[\|\Tilde{w}^{\eps, \eta}\|_{L^2(\mathbb{T}^d)}^2\big] \Big)\le C\sqrt{\eps} \rightarrow 0, \notag \\
 &\Tilde{\mathcal{B}}_2  \le C\|{\tt \Phi}'\|_{L^\infty}\|(-\Delta)^\theta[\psi]\|_{L^\infty} \Big(\Tilde{\mathbb{E}}\int_0^T \|\Tilde{w}^{\eps, \eta} - \Tilde{w}^{\eta}\|_{L^2(\mathbb{T}^d)}^2\,dr\Big)^\frac{1}{2} \rightarrow 0\,.
\end{align*}
Using It\^o isometry, Cauchy-Schwartz inequality, a-priori estimates on $\Tilde{w}^{\eps, \eta}$ and \ref{A3}, we get
\begin{align*}
&\Tilde{\mathbb{E}}\Big|\int_0^t\big\langle{\tt h}(\sqrt{\eps}\Tilde{w}^{\eps,\eta}+1) - {\tt h}(1),\psi\big\rangle d\Tilde{W}(s) \Big|
\notag \\ &\le C(\psi)\Big(\Tilde{\mathbb{E}}\int_0^T\|{\tt h}(\sqrt{\eps}\Tilde{w}^{\eps,\eta}+1) - {\tt h}(1)\|_{\mathcal{L}_2(\mathcal{H}, L^2(\mathbb{T}^d))}^2 \,dt\,\Big)^{\frac{1}{2}}\notag \\ &  \le \sqrt{\eps}C(\psi,T)\Big(\Tilde{\mathbb{E}}\underset{0 \le t \le T}{\text{sup}}\|\Tilde{w}^{\eps,\eta}\|_{L^2(\mathbb{T}^d)}^2\Big)^{\frac{1}{2}} \le C\sqrt{\eps} \rightarrow 0, \quad \text{as} \quad  \eps \rightarrow 0.
\end{align*} 
In view of above convergence result sending $\eps \rightarrow 0$ in \eqref{eq:tilde-w-eps-eta}, we get, for $t\in [0,T]$,
\begin{align}
    &\langle\Tilde{w}^{\eta}(t), \psi\rangle = \int_0^t\big\langle F'(1)\Tilde{w}^{\eta}, \nabla \psi \big\rangle\,dr - \int_0^t\big\langle  {\tt \Phi}'(1)\Tilde{w}^{\eta}, (-\Delta)^\theta[\psi]\big\rangle\,dr\notag \\ &+ \eta\int_0^t \langle\Tilde{w}^{ \eta}, \Delta\psi\rangle\,dr +  \int_0^t\big\langle{\tt h}(1),\psi\big\rangle d\tilde{W}(r), \quad \Tilde{\mathbb{P}}-\text{a.s}, \notag
\end{align}
In other words, $\Tilde{w}^{\eta}$ is a martingale solution of \eqref{eq:viscous-star-FSCL}.
\subsubsection{\textbf{Pathwise Solutions}}
We wish to show that $w^{\eps, \eta}$ converges to the unique strong solution of \eqref{eq:viscous-star-FSCL}. 
Thanks to the pathwise uniqueness ~(see, \cite[Section 3.2]{AC_2}) and existence of martingale solution, we make use of Gy\"{o}ngy-Krylov characterization of convergence in probability \cite{Gyongy-96} to infer the existence of pathwise solution. In this regard, we recall the following results from \cite{Gyongy-96}. 
\begin{prop}\label{prop:Gyongy}Let $\mathbb{M}$ be a Polish space equipped with the Borel $\sigma$-algebra. A sequence of $\mathbb{M}$- valued random variables $\{Y_n: n \in \N\}$ convergence in probability if and only if for every subsequene of joint laws $\{\mu_{n_k,m_k}: k \in \N\}$, there exists a further subsequence which converges weakly to a probability measure $\mu$ such that
\begin{align}
 \mu\big((x,y) \in \mathbb{M}\times\mathbb{M}: x=y \big) = 1. \notag  
\end{align}
\end{prop}
 Let us denote $\mathrm{X} = \{v: v \in L^2([0,T]; L^2(\mathbb{T}^d))\}$. We consider the collections of joint laws of $(w^{\eps, \eta}, w^{\Bar{\eps},\eta})$ on $\mathrm{X} \times \mathrm{X}$ and denote it by $\nu^{\eps, \Bar{\eps}}_\eta$. The extend path space is defined as $\Bar{\mathrm{X}} = \mathrm{X} \times \mathrm{X} \times C([0,T]; \mathcal{H}_0).$ Let the law of $W$ be $\nu_W$ and $\nu^{\eps, \Bar{\eps}}$ be the joint law of $(w^{\eps, \eta}, w^{\Bar{\eps},\eta}, W)$. Similar to  Corollary \ref{cor:tightness-w-eps-eta}, we see that the joint law $\{\nu^{\eps, \Bar{\eps}}\}$ is tight on $\Bar{\mathrm{X}}.$ By Prokhorov theorem to the family $\{\nu^{\eps, \Bar{\eps}}\}_{\eps, \Bar{\eps}}$, we get it is relatively weakly compact. Thus, passing to a weakly convergent subsequence, still denoted $\nu^{\eps, \Bar{\eps}}$ and its limit law $\nu^\eta,$  applying modified version of Skorokhod representation theorem, we obtain a new probability space $(\Bar{\Omega}, \Bar{\mathcal{F}}, \Bar{\mathbb{P}}),$ and a subsequence $(\Hat{w}^{\eps, \eta},\check{w}^{\Bar{\eps}, \eta}, \Bar{W}^{\eps,\Bar{\eps}})$ converges to $(\Hat{w}^{\eta},\check{w}^{\eta}, \Bar{W})$ $\Bar{\mathbb{P}}$-a.s. in $\Bar{\mathrm{X}}$. Moreover, for any $\eps, \Bar{\eps}$, $\Bar{W}^{\eps, \Bar{\eps}} = \Bar{W}$ and 
 $$\Bar{\mathbb{P}}\big((\Hat{w}^{\eps, \eta},\check{w}^{\Bar{\eps}, \eta}, \Bar{W}^{\eps,\Bar{\eps}}) \in \cdot\big) = \Bar{\nu}^{\eps, \Bar{\eps}} = \nu^{\eps, \Bar{\eps}} = \mathbb{P}\big((w^{\eps, \eta}, w^{\Bar{\eps}, \eta}, W)\in \cdot\big),\,\Bar{\mathbb{P}}\big((\hat{w}^\eta, \check{w}^\eta, \Bar{W}) \in \cdot\big) = \nu^\eta.$$
 To be precise, $\nu^{\eps, \Bar{\eps}}_\eta$ converges weakly to a measure $\nu$ such that $\nu = \Bar{\mathbb{P}}\big((\Hat{w}^\eta, \check{w}^\eta) \in \cdot\big).$ Observe that, applying homologous arguments as done in Subsection \ref{sec:Identification-of-limit} to $(\Hat{w}^{\eps, \eta}, \Bar{W}^{\eps})$, $(\Hat{w}^{\eta}, \Bar{W})$ and  $(\check{w}^{\Bar{\eps}, \eta}, \Bar{W}^{\Bar{\eps}})$, $(\check{w}^{\eta}, \Bar{W})$, we obtain that $\Hat{w}^\eta$ and $\check{w}^\eta$ are martingale solution to \eqref{eq:viscous-star-FSCL} with initial data $\Hat{w}_0^\eta = 0 = \check{w}_0^\eta.$ Thanks to uniqueness of solution of \eqref{eq:viscous-star-FSCL}, we conclude that $\Hat{w}^\eta = \check{w}^\eta$ in $\mathrm{X}$, $\Bar{\mathbb{P}}$-a.s.. Thus,
 $$\nu\big((x,y) \in \mathrm{X} \times \mathrm{X}: x=y\big) = \Bar{\mathbb{P}}(\Hat{w}^\eta = \check{w}^\eta \quad \text{in} \quad \mathrm{X} ) = 1.$$
 In view of the above discussion and recalling Proposition \ref{prop:Gyongy}, we conclude that the original sequence $w^{\eps, \eta}$ defined in the probability space $(\Omega, \mathcal{F}, \mathbb{P})$ convergence in probability in the topology of $\mathrm{X}$ to a random variable $w^\eta$ i.e., $\|w^{\eps, \eta} -w^\eta\|_{L^2([0,T];L^2(\mathbb{T}^d))} \rightarrow 0$ as $\eps \rightarrow 0$ $\mathbb{P}$-a.s. 
 Again, using Vitali convergence theorem along with a-priori estimates and recalling $w^{\eps, \eta} = \frac{u^{\eps, \eta} - \hat{u}^\eta}{\sqrt{\eps}}$, we have, for fixed $\eta > 0$
 \begin{align}
\underset{\eps \rightarrow 0}{\text{lim}}\,\mathbb{E}\Big[\Big\|\frac{u^{\eps, \eta} - \hat{u}^\eta}{\sqrt{\eps}} - w^\eta \Big\|_{L^2([0,T]; L^2(\mathbb{T}^d))}\Big] = 0\,. \label{conv:2-new}
 \end{align}
 A similar line of arguments as done in Section \ref{sec:Identification-of-limit} yields that $w^\eta$ is a solution to \eqref{eq:viscous-star-FSCL}. However, thanks to the unique solution of \eqref{eq:viscous-star-FSCL}, we have $w^\eta = u^{\star, \eta}.$ Thus, from \eqref{conv:2-new}, we get 
 \begin{align}\label{eq:clt-2-eps-0} 
 &\underset{\eps \rightarrow 0}{\text{lim}}\,\mathbb{E}\Big[\Big\|\frac{u^{\eps, \eta} - \hat{u}^\eta}{\sqrt{\eps}} - u^{\star, \eta}\Big\|_{\mathcal{E}_1}\Big] = 0.  
 \end{align}
 Combining \eqref{eq:strong-cong-star-eta}, \eqref{eq:clt-2-eps-0} and Proposition \ref{prop:double-varibale-1st-term-clt} in \eqref{inq:triangle-CLT}, one may easily conclude the desired result of Theorem \ref{thm:clt}.
\section{Moderate Deviation Principle}\label{sec:MDP}
\subsection{Proof of Theorem \ref{thm:mdp}} To establish Theorem \ref{thm:mdp}, we need to validate Condition \ref{cond:2}.  Thanks to Section \ref{subsec:Validation-of-continuity-of-g-nut}, part ${\rm ii)}$ of Condition \ref{cond:2} holds.
Thus, we only need to verify the part ${\rm i)}$ of Condition \ref{cond:2}. For any $\{\ell_\eps\} \in \mathcal{A}_N$, we consider
\begin{align}\label{eq:mdp-controlled-eq}
    \begin{cases}
    \displaystyle d\tilde{v}^\eps + \text{div}\Bar{F}(\tilde{v}^\eps)\,dt + (-\Delta)^\theta[\Bar{{\tt \Phi}}(\Tilde{v}^\eps)](x)\,dt \\= {\tt h}\big(\sqrt{\eps}\Lambda(\eps)\tilde{v}^\eps + \Hat{u}\big){\ell}_\eps\,dt + \Lambda(\eps)^{-1}{\tt h}\big(\sqrt{\eps}\Lambda(\eps)\tilde{v}^\eps + \Hat{u}\big)\,dW(t),\quad
    \tilde{v}^\eps(0) = 0, 
    \end{cases} 
\end{align}
Observe that, \eqref{eq:mdp-controlled-eq} is a special case of \eqref{eq:small-noise-ldp} with $\hat{u} = 1$. Thus, there exists a unique kinetic solution $\tilde{v}^\eps$ of \eqref{eq:mdp-controlled-eq} satisfying \eqref{inq:ldp-cp} and the associated kinetic formulation. 
Recall that $\Tilde{v}^\eps = \mathcal{\tt G}^\eps\Big(W(\cdot) + \Lambda(\eps)\int_0^{\cdot}{\ell}_\eps(s)\,ds\Big)$ and $z_{{\ell}_\eps} = \mathcal{\tt G}^0\Big(\int_0^{\cdot}{\ell}_\eps(s)\,ds\Big)$ are kinetic solutions of \eqref{eq:mdp-controlled-eq} and \eqref{eq:mdp-skeleton} with ${\ell}$ replaced by ${\ell}_\eps$ respectively. To complete the proof of Theorem \ref{thm:mdp}, we need to prove  that
\begin{align}
    \|\Tilde{v}^\eps - z_{{\ell}_\eps}\|_{\mathcal{E}_1} \rightarrow 0 \quad \text{in probability}. \label{eq:Mdp-conv-p}
\end{align}
Due to lack of symmetry we can't directly establish \eqref{eq:Mdp-conv-p} by applying doubling of variables technique to $\Tilde{v}^\eps$ and $z_{{\ell}_\eps}$. To resolve this technical issue, we approach in same manner as done in the proof CLT and introduce some auxiliary  parabolic approximation equations. To be more precise, we consider the following  viscous approximation of \eqref{eq:mdp-controlled-eq}.
\begin{align}\label{eq:mdp-controlled-eq-viscous}
 &d\tilde{v}^{\eps, \eta} + \text{div}\Bar{F}(\tilde{v}^{\eps, \eta})\,dt + (-\Delta)^\theta[\Bar{{\tt \Phi}}(\Tilde{v}^{\eps, \eta})](x)\,dt  \\&=  \eta\Delta\tilde{v}^{\eps, \eta} + {\tt h}\big(\sqrt{\eps}\Lambda(\eps)\tilde{v}^{\eps, \eta} + \Hat{u}\big){\ell}_\eps\,dt + \Lambda(\eps)^{-1}{\tt h}\big(\sqrt{\eps}\Lambda(\eps)\tilde{v}^{\eps, \eta} + \Hat{u}\big)\,dW(t), \notag
\end{align}
with $\tilde{v}^{\eps, \eta}(0) = 0.$ In view of well-possedness for viscous problem  from \cite[Section 3.2]{AC_2} and Subsection \ref{eq:Existence and uniqueness of viscous solution} along with Girsanov theorem one can conclude the existence and uniqueness of \eqref{eq:mdp-controlled-eq-viscous} with a solution $\tilde{v}^{\eps, \eta}$. Thanks to Subsection \ref{eq:Existence and uniqueness of viscous solution}, we have the well-posedness of following approximate PDE.
\begin{align}
    &dz_{{\ell}_\eps}^\eta +  \text{div}(F'(\Hat{u})z_{{\ell}_\eps}^\eta)\,dt + (-\Delta)^\theta[{\tt \Phi}'(\Hat{u})z_{{\ell}_\eps}^\eta](x)\,dt = \eta\Delta z_{{\ell}_\eps}^\eta +{\tt h}(\hat{u}){\ell}_\eps(t)\,dt,\label{eq:mdp-skeleton-viscous}
\end{align}
with $z_{{\ell}_\eps}^\eta(0) = 0$ and it has a unique solution $z_{{\ell}_\eps}^\eta$. With $\tilde{v}^{\eps, \eta}$ and  $z_{{\ell}_\eps}^\eta$ in hand, we have
\begin{align}
    &\mathbb{E}\|\tilde{v}^{\eps}-z_{{\ell}_\eps}\|_{\mathcal{E}_1} \le  \mathbb{E}\|\tilde{v}^{\eps}-\tilde{v}^{\eps, \eta}\|_{\mathcal{E}_1}+ \mathbb{E}\|\tilde{v}^{\eps, \eta} - z_{{\ell}_\eps}^\eta\|_{\mathcal{E}_1} + \mathbb{E}\| z_{{\ell}_\eps}^\eta - z_{{\ell}_\eps}\|_{\mathcal{E}_1}\,. \notag 
    \end{align}
By Proposition \ref{prop:strong-convergence} and well-possedness of \eqref{eq:mdp-skeleton-viscous}, we have $\underset{\eps \in (0,1)}{\text{sup}}\underset{\eta \rightarrow 0}{\text{lim}}\mathbb{E}\| z_{{\ell}_\eps}^\eta - z_{{\ell}_\eps}\|_{\mathcal{E}_1} = 0$. Following the proof of Proposition \ref{prop:double-varibale-1st-term-clt}, it is easy to observe that, for fixed $\eps>0$, $\mathbb{E}\|\tilde{v}^{\eps}-\tilde{v}^{\eps, \eta}\|_{\mathcal{E}_1} \rightarrow 0$ as $\eta \rightarrow 0$. Merely a consequence of obtained results in LDP and CLT and using the deviation scaling property 
$\sqrt{\eps} \Lambda(\eps) \goto 0$ as $\eps \goto 0$, we get $\mathbb{E}\|\tilde{v}^{\eps, \eta} - z_{{\ell}_\eps}^\eta\|_{\mathcal{E}_1} \rightarrow 0$ as $\eps \rightarrow 0.$ This completes the proof of Theorem \ref{thm:mdp}.

\vspace{0.1cm}

\noindent{\bf Acknowledgement:} 
The first author would like to acknowledge the financial support by CSIR, India. The second author is supported by Department of Science and Technology, Govt. of India-the INSPIRE fellowship~(IFA18-MA119).



\begin{thebibliography}{}

\bibitem{Alibaud 2007}
\textsc{Alibaud, N.} (2007). Entropy formulation for fractal conservation laws. \textit{ J. Evol. Equ.} 7, no. 1, 145–175.

\bibitem{Alibaud 2020}\textsc{Alibaud, N., Andreianov, B. and Ouedraogo, A.} (2020).  Nonlocal dissipation measure and $L^1$ kinetic theory for fractional conservation laws. \textit{Comm. Partial Differential Equations.} 45, no. 9, 1213–1251.

\bibitem{Bauzet-2012}\textsc{Bauzet, C., Vallet, G. and Wittbold, P.}
(2012). The Cauchy problem for a conservation law with a multiplicative stochastic perturbation. 
\textit{Journal of Hyperbolic Differential Equations.} 9, no 4, 661-709. 

\bibitem{Bauzet-2015}\textsc{Bauzet, C., Vallet, G. and  Wittbold, P.} (2015).
  A degenerate parabolic-hyperbolic Cauchy problem with a stochastic force.
\textit{J. Hyperbolic Differ. Equ.} 12, no. 3, 501-533.

\bibitem{frac lin} \textsc{Bhauryal, N.,  Koley, U. and Vallet, G.} (2020). The Cauchy problem for fractional conservation laws driven by L\'evy noise. \textit{Stochastic Process. Appl.} 130, no. 9, 5310–5365.

\bibitem{frac non}\textsc{Bhauryal, N.,  Koley, U. and Vallet, G.} (2021) A fractional degenerate parabolic-hyperbolic Cauchy problem with noise.\textit{ J. Differential Equations} 284, 433–521.


\bibitem{Majee-2015}\textsc{Biswas, I. H., Karlsen, K. H. and  Majee, A. K.} (2015) Conservation laws driven by  L\'{e}vy  white noise.\textit{ J. Hyperbolic Differ. Equ.} 12, no. 3, 581-654.

\bibitem{Majee-2014}
\textsc{Biswas, I. H. and  Majee, A. K.} (2014).
 Stochastic conservation laws: weak-in-time formulation and strong entropy condition.
\textit{Journal of Functional Analysis.} 267, 2199-2252.

\bibitem{BKM-2015}\textsc{Biswas, I. H., Koley, U. and Majee, A.K.} (2015).
 Continuous dependence estimate for conservation laws with L\'{e}vy noise.
\textit{J. Differential Equations} 259, no. 9, 4683-4706.

\bibitem{Majee-2019}\textsc{Biswas, I. H., Majee, A. K. and Vallet, G.} (2019). On the Cauchy problem of a degenerate parabolic-hyperbolic PDE with L\'{e}vy noise.
\textit{Adv. Nonlinear Anal.} 8, no. 1, 809-844.

\bibitem{Potential-2018}\textsc{Brz\'ezniak, Z., Hausenblas, E. and Razafimandimby, P.A.} (2018) Stochastic reaction-diffusion equations driven by jump processes.
 \textit{Potential Anal.} 49, no. 1, 131-201.
 
\bibitem{Budhiraja}\textsc{Budhiraja, A. and  Dupuis, P.} (2000).
A variational representation for positive functional of infinite dimensional Brownian motion. \textit{Probab. Math. Statist.} 20, 39–61.


\bibitem{AC_2}\textsc{Chaudhary, A.} (2022). Stochastic fractional conservation laws. \textit{arXiv:2205.06005.}
\bibitem{AC_1}\textsc{Chaudhary, A.} (2023). Stochastic degenerate fractional conservation laws. \textit{Nonlinear Differ. Equ. Appl.} 30-42.

\bibitem{Chen-2003}\textsc{Chen, G. Q. and Perthame, B.} (2003). Well-posedeness for non-isotropic degenerate parabolic-
hyperbolic equations. \textit{Ann. Inst. H. Poincare-Anal. non Lineaire.} 20(4), 645–668.

\bibitem{Chen:2012fk}\textsc{Chen, G. Q., Ding Q. and Karlsen, K.H.} (2012) On nonlinear stochastic balance laws.
\textit{Arch. Ration. Mech. Anal.}, 204(3), 707-743.


\bibitem{cifani} \textsc{Cifani, S. and  Jakobsen, E. R.} (2011). Entropy solution theory for fractional degenerate convection-diffusion equations. \textit{Ann. Inst. Henri Poincar\'e.} 28(3), 413–441.

\bibitem{Zhang-2020} \textsc{Dong, Z., Wu, J. L., Zhang, R. and Zhang, T.} (2020). Large deviation principles for first order scalar conservation laws with stochastic forcing.
\textit{Ann. Appl. Probab.} 30, no. 1, 324–367.

\bibitem{Dong-21}\textsc{Dong, Z., Zhang, R. and Zhang, T.} (2020) Large deviations for quasi-linear parabolic stochastic partial
differential equations. \textit{Potential Analysis.} 53:183–202.


\bibitem{Hofmanova-2016} \textsc{Debussche, A., Hofmanov\'{a}, M. and Vovelle., J} (2016) Degenerate parabolic stochastic partial differential equations: quasilinear case.
\textit{Ann. Probab.} 44, no. 3, 1916–1955.

 \bibitem{Vovelle2010}\textsc{Debussche, A.,  and Vovelle., J} (2010). Scalar conservation laws with stochastic forcing. 
\textit{J. Funct. Analysis.} 259, 1014-1042. 

\bibitem{dafermos}\textsc{Dafermos., C. M.} (2007). Hyperbolic conservation laws in continuum physics
\textit{J. Funct. Anal.} 243(2):631--678.

\bibitem{Vovelle-2018}\textsc{Dotti, S. and Vovelle, J.} (2018).  Convergence of approximations to stochastic scalar conservation laws.
\textit{ Arch. Ration. Mech. Anal.} 230, no. 2, 539-591.

\bibitem{Vovelle-2020} \textsc{Dotti, S. and Vovelle, J.} (2020). Convergence of the finite volume method for scalar conservation laws with multiplicative noise: an approach by kinetic formulation. \textit{Stoch. Partial Differ. Equ. Anal. Comput.} 8, no. 2, 265-310.

\bibitem{dong-2017} \textsc{Dong, Z., Xiong, J., Zhai, J. and Zhang. T.} (2017) A moderate deviation principle for 2D stochastic Navier-Stokes equations driven by multiplicative L\'evy noises. \textit{J. Funct. Anal.} 1, 227–254.

\bibitem{Dupuis-97}\textsc{Dupuis, P. and  Ellis, R. S.} (1997). A Weak Convergence Approach to the Theory of Large Deviations.
\textit{ Wiley Series in Probability and Statistics.} Wiley, New York.

\bibitem{J-stat-12}\textsc{Ermakov, M. S.} (2012). The sharp lower bound of asymptotic efficiency of estimators in the zone of moderate deviation probabilities. \textit{Electron. J. Stat.} 6, 2150–2184.

\bibitem{Ellis-85}\textsc{Ellis, R. S.} (1985). Entropy, Large Deviations and Statistical Mechanics. \textit{Springer-Verlag}, New York.




\bibitem{Fatheddin-16}\textsc{Fatheddin, P. and Xiong, J.} (2016). Moderate deviation principle for a class of stochastic partial differential equations. \textit{J. Appl. Probab.} 1, 279–292.

\bibitem{nualart:2008}\textsc{Feng, J. and Nualart, D.} (2008). Stochastic scalar conservation laws.
\textit{J. Funct. Anal.}, 255, no. 2, 313-373.

\bibitem{godu}\textsc{Godlewski, E. and Raviart, P. A.} (1991).
Hyperbolic systems of conservation laws. \textit{Volume 3/4 of Math\'ematiques \& Applications (Paris) [Mathematics and Applications]}. Ellipses, Paris.

\bibitem{Flandoli:1995}\textsc{Flandoli, F. and Gatarek, D.} (1995). Martingale and stationary solutions for stochastic Navier-Stokes
equations. \textit{Prob. Theory Related Fields.}102, 367–391.

\bibitem{Gyongy-96}\textsc{Gy\"{o}ngy, I. and Krylov, N.} (1996). Existence of strong solutions for It\^o’s stochastic equations via approximations. \textit{Probab.
Theory Related Fields.} 105 (2), 143–158.


\bibitem{Kruzkov}\textsc{Kruzkov, S. N.} (1970) First order quasilinear equations with several independent variables. \textit{Sb. (N.S.)} 1(123): 228-255.

\bibitem{Ann-stat-91}\textsc{Ibragimov, I. A. and Khasminskii, R. Z.} (1991). Asymptotically normal families of distributions and efficient estimation. \textit{Ann. Stat.} 19, 1681–1721.
\bibitem{kavin-22} \textsc{Kavin R. and  Majee, A. K.} (2022). Stochastic evolutionary p-Laplace equation:Large Deviation Principles and Transportation
Cost Inequality. \textit{arXiv.2210.11036}.

\bibitem{kim}\textsc{Kim, J. U.} (2003).
On a stochastic scalar conservation law.
\textit{Indiana Univ. Math. J.}  52 (1), 227-256.

\bibitem{Majee-2017}\textsc{Koley U., Majee 
 A. K. and Vallet. G.} (2017) Continuous dependence estimate for a degenerate parabolic-hyperbolic equation with L\'{e}vy noise.
 \textit{Stoch. Partial Differ. Equ. Anal. Comput.} 5, no. 2, 145-191.
 

\bibitem{Lions}\textsc{Lions, P. L., Perthame, B. and Tadmor, E.} (1994)A kinetic formulation of multidimensional scalar conservation laws and related
equations.\textit{J. Amer. Math. Soc.} 7 (1), 169-191.
\bibitem{Liu-10}\textsc{Liu, W.} (2010). Large deviations for Stochastic evolution Equations with small multiplicative noise. \textit{ Appl. Math. Optim.} 61(1):
27–56.
  
\bibitem{Mariani}\textsc{Mariani, M.} (2010). Large deviations principles for stochastic scalar conservation laws. \textit{Probab. Theory
Related Fields} 147 607–648.

\bibitem{Matoussi}\textsc{Matoussi, A., Sabbagh, W. and Zhang, T.} Large Deviation Principles of Obstacle Problems for
Quasilinear Stochastic PDEs. \textit{arXiv:1712.02169}.

\bibitem{Rockner}\textsc{Prévôt, C. and  Röckner, M.} (2007). A Concise Course on Stochastic Partial Differential Equations. \textit{Lecture Notes in Mathemat-ics.} vol.1905, Springer.


\bibitem{Varadhan-1}\textsc{Varadhan, S .R .S .} (1966). Asymptotic probabilities and differential equations. \textit{Commun. Pure Appl. Math.} 19, 261–286.
\bibitem{Varadhan-2}\textsc{Varadhan, S .R .S .} (1984). Large deviations and Applications.  \textit{46, CBMS-NSF Series in Applied Mathematics, SIAM.} Philadelphia.
\bibitem{CLT}\textsc{Wu, Z. and Zhang. T.} (2022). Central limit theorem and moderate deviation principle for
stochastic scalar conservation laws. \textit{J. Math. Anal. Appl,}, 516.

\bibitem{zhang-mdp-21}\textsc{Wang, R., Zhai, J. and Zhang, T.} (2015). A moderate deviation principle for 2-D stochastic Navier-Stokes equations. \textit{J. Differ. Equ.} 10, 3363–3390.

\bibitem{zhang-21-mdp-clt} \textsc{Xiong, J. and Zhang, R.} (2021). Semilinear stochastic partial differential equations: central limit theorem and moderate deviations. \textit{Math. Methods Appl. Sci.} 44(8), 6808–6838.

\bibitem{Zhang-JD-17}\textsc{Dong, Z., Zhai,J. and Zhang, R.} (2017). Large deviation principles for 3D stochastic primitive equations. \textit{J. Differential Equations.}
263, 3110-3146.
\end{thebibliography}
\end{document}